\documentclass[11pt,a4paper,dvipsnames]{article}

\usepackage{fullpage,amssymb,amsmath,amsfonts,amsthm,paralist,graphicx,color,float, mathtools,tikz,xfrac}
\usepackage[utf8]{inputenc}
\usepackage[T1]{fontenc}
\usepackage[english]{babel}
\usepackage[colorlinks=true,linkcolor=BrickRed,citecolor=blue]{hyperref}
\usepackage[normalem]{ulem}
\usepackage{centernot}
\usepackage{ifthen,xkeyval, tikz, calc, graphicx}
\usepackage{xcolor}
\usepackage{framed}
\usepackage{dsfont}
\usepackage[mathscr]{euscript}
\usepackage{comment}
\usepackage{enumitem}
\usepackage{caption}
\usepackage{subcaption}
\usepackage{thmtools}
\usetikzlibrary{plotmarks}
\usetikzlibrary{shapes.misc}
\usepackage{orcidlink}
\usepackage{lmodern}

\newtheorem{theorem}{Theorem}[section]
\newtheorem{lemma}[theorem]{Lemma}
\newtheorem{prop}[theorem]{Proposition}
\newtheorem{corollary}[theorem]{Corollary}
\newtheorem{observation}[theorem]{Observation}

\theoremstyle{definition}
\newtheorem{definition}[theorem]{Definition}

\newtheorem{assumption}{Assumption}
\newtheorem{assumptionalt}{Assumption}[assumption]
\newenvironment{assumptionp}[1]{
  \renewcommand\theassumptionalt{#1}
  \assumptionalt
}{\endassumptionalt}

\theoremstyle{remark}
\newtheorem{remark}[theorem]{Remark}

\definecolor{shadecolor}{named}{GreenYellow}

\makeatletter
\newcommand{\pushright}[1]{\ifmeasuring@#1\else\omit\hfill$\displaystyle#1$\fi\ignorespaces}
\newcommand{\pushleft}[1]{\ifmeasuring@#1\else\omit$\displaystyle#1$\hfill\fi\ignorespaces}
\makeatother

\newcommand{\plal}{\mathbb P_{\lambda,L}}

\newcommand{\Acal}{\mathcal{A}}
\newcommand{\Gcal}{\mathcal{G}}

\newcommand{\E}{\mathbb E}
\newcommand{\Ecal}{\mathcal E}

\newcommand{\R}{\mathbb R}

\newcommand{\Rd}{\mathbb R^d}
\newcommand{\Z}{\mathbb Z}

\newcommand{\N}{\mathbb N}

\newcommand{\dd}{\mathrm{d}} 
\newcommand{\C}{\mathscr {C}}
\newcommand{\Complex}{\mathbb C}

\newcommand{\piv}[1]{\textsf {Piv}(#1)}

\DeclareMathOperator*{\esssup}{ess\,sup}
\DeclareMathOperator*{\essinf}{ess\,inf}

\newcommand{\LandauBigO}[1]{\mathcal{O}\left(#1\right)}

\newcommand{\conn}[3]{#1 \longleftrightarrow #2\textrm { in } #3}
\newcommand{\adja}[3]{#1 \sim #2\textrm { in } #3}

\newcommand{\nconn}[3]{#1 \centernot\longleftrightarrow #2\textrm { in } #3}
\newcommand{\dconn}[3]{#1 \Longleftrightarrow #2\textrm { in } #3}

\newcommand{\xconn}[4]{#1 \xleftrightarrow{\,\,#4\,\,} #2\textrm { in } #3}

\newcommand{\thinning}[2]{#1_{\langle #2 \rangle}}

\newcommand{\Id}{\mathds 1}

\newcommand{\Optau}{\mathcal{T}}

\newcommand{\Opconnf}{\varPhi}

\newcommand{\OpLace}{\varPi}

\DeclarePairedDelimiter\abs{\lvert}{\rvert}
\DeclarePairedDelimiter\norm{\lVert}{\rVert}

\DeclarePairedDelimiterX{\inner}[2]{\langle}{\rangle}{#1, #2}

\newcommand{\orig}{o}
\newcommand{\e}{\text{e}}

\newcommand{\connf}{\varphi}

\definecolor{darkorange}{RGB}{255,165,0}
\definecolor{altviolet}{RGB}{139,0,139}
\definecolor{turquoise}{RGB}{64,224,208}
\definecolor{lblue}{RGB}{173,216,230}
\definecolor{violet}{RGB}{238,130,238}
\definecolor{darkgreen}{RGB}{0,100,0}
\definecolor{lgreen}{RGB}{144,238,144}

\newcommand{\HypDim}{{\mathbb{H}^d}}
\newcommand{\HypTwo}{{\mathbb{H}^2}}
\newcommand{\HypThree}{{\mathbb{H}^3}}

\newcommand{\habsd}[1]{\abs*{#1}_\HypDim}
\newcommand{\dist}[1]{\mathrm{dist}_{\HypDim}\left(#1\right)}
\newcommand{\distThree}[1]{\mathrm{dist}_{\HypThree}\left(#1\right)}

\newcommand{\artanh}{\mathrm{artanh}}
\newcommand{\arsinh}{\mathrm{arsinh}}

\numberwithin{equation}{section}
\allowdisplaybreaks

\title{Estimating the critical threshold for stretched-out hyperbolic random connection models}
\author{Matthew Dickson\footnote{University of British Columbia, Department of Mathematics, Vancouver, BC, Canada, V6T 1Z2; Email: dickson@math.ubc.ca; \orcidlink{0000-0002-8629-4796}~https://orcid.org/0000-0002-8629-4796}}
 
\date{}

\begin{document}
\maketitle

\vspace{-1em}

{\centering{ \today}\par}

\vskip-3em

\begin{abstract}
This paper examines the model-dependent asymptotic behaviour of the critical threshold intensity for stretched-out random connection models (RCMs) on hyperbolic spaces. The proof uses lace expansion arguments, but has notable qualitative differences to the Euclidean case in how it evaluates spectral radii. The result is applied to the Boolean disc RCM and a heat kernel RCM.
\end{abstract}


\smallskip


{\footnotesize
}

\section{Introduction}

Many percolative systems exhibit a phase transition: below some critical parameter value there exists no infinite connected component almost surely, and above the critical parameter there does exist an infinite connected component almost surely. Kesten proved in \cite{kesten1980critical} that for Bernoulli bond percolation (BBP) on nearest-neighbour $\Z^2$ the critical edge probability $p_c=\frac{1}{2}$. Having such an exact answer for the critical value is atypical.

Estimates for the critical probability for spread-out Bernoulli bond percolation (BBP) on various graphs have been found in works such as \cite{HarSla90,penrose1993spread,hofstadsakai2005critical,freiperkins2016lower,hong2023upper,spanos2024spread,hong2025lower}. Most relevant to this paper,  \cite{hofstadsakai2005critical} uses lace expansion results to examine spread-out BBP on $\Z^d$ for $d>6$. Amongst other results, they prove that if  two vertices $x,y\in\Z^d$ share an edge independently with probability
\begin{equation}
    \frac{p}{\left(2L+1\right)^d-1}\Id\left\{\norm*{x-y}_\infty \leq L\right\},
\end{equation}
then the critical probability obeys
\begin{equation}
\label{eqn:BBPexpansion}
    p_\mathrm{c} = 1 + \left(U^{\star 2}\left(\orig\right) + \frac{1}{2}\sum^\infty_{n=3}\left(n+1\right) U^{\star n}\left(\orig\right)\right)L^{-d} + \LandauBigO{L^{-d-1}},
\end{equation}
as $L\to\infty$, where $U\colon \Rd\to \left[0,1\right]$ is given by
\begin{equation}
    U(x) := 2^{-d}\Id\left\{\norm*{x}_\infty < 1\right\}
\end{equation}
and $U^{\star n}\left(x\right)$ denotes the $n$-fold convolution of $U$ in $\R^d$.

In this paper we will be concerned with approximating the critical threshold for stretched-out random connection models (RCMs) on the hyperbolic space $\HypDim$ (for $d\geq 2$). We will also be using a lace expansion approach (building upon the arguments in \cite{HeyHofLasMat19} and \cite{dickson2025expansion}), but will not require the full \emph{spread-out} behaviour of \cite{hofstadsakai2005critical}, and the relevancy of terms in the asymptotics will differ from \eqref{eqn:BBPexpansion}. The high-level proof description and the result will look similar to the argument in \cite{dickson2025expansion} which dealt with RCMs on $\Rd$ as $d\to\infty$, but as we shall see, the ``easy'' and ``difficult'' parts of the proof are transposed. 
The lace expansion component of the argument is now significantly easier because we can inherit the finiteness of the critical triangle diagram from the spherical transform argument in \cite{dickson2024NonUniqueness} -- no bootstrap argument is required. However, extracting the critical intensity, $\lambda_\mathrm{c}$, from the resulting Ornstein-Zernike equation is no longer as simple a matter as applying the ($\HypDim$-analog of the) Fourier transform, and a more careful study of the $L^1\to L^1$ spectral radius of particular convolution operators is required.

\paragraph{Notation.} We write $f=\LandauBigO{g}$ to mean that $\limsup\abs*{\sfrac{f}{g}}<\infty$, and write $f=o\left(g\right)$ or $f\ll g$ to mean that $\lim\sfrac{f}{g}=0$. By $f\sim g$ we mean that $\lim\sfrac{f}{g}=1$. Given a real-valued function $f$, $\left(f\right)_+$ is the positive part and $\left(f\right)_-$ is the absolute value of the negative part. That is, $\left(f\right)_+ := \frac{1}{2}\left(\abs*{f} + f\right)$ and $\left(f\right)_- := \frac{1}{2}\left(\abs*{f} - f\right)$. Given $z\in\Complex$, we write $\overline{z}$ for the complex conjugate of $z$.

\section{The Model and the Results}
\subsection{The Model} 
\label{sec:model}

    \paragraph{Hyperbolic space.}
    For $d\geq 2$, hyperbolic $d$-space (denoted $\HypDim$) is the unique simply connected $d$-dimensional Riemannian manifold with constant sectional curvature $-1$. It can be convenient to describe $\HypDim$ using the Poincar{\'e} ball model (also called the conformal ball model). Let us consider the open Euclidean ball in $\Rd$ with unit radius: $\mathbb{B}:=\left\{x\in\Rd\colon \abs*{x}<1\right\}$. Then for each $x\in\mathbb{B}$ the \emph{hyperbolic metric} on $\mathbb{B}$ assigns hyperbolic distance
    \begin{equation}
        \dist{x,\orig} = 2\, \artanh\abs*{x},
    \end{equation}
    where $\orig$ is the origin of $\Rd$ and $\abs{\cdot}$ is the Euclidean norm. The (hyperbolic-)isometries of $\mathbb{B}$ are then the M{\"o}bius transformations of $\mathbb{B}$, and these can be used to extend this distance to any pair of points in $\mathbb{B}$. For any Lebesgue measurable subset $E\subset \mathbb{B}$, the \emph{hyperbolic measure} on $\mathbb{B}$ assigns mass
    \begin{equation}
        \mu\left(E\right) := \int_{E}\frac{4}{\left(1-\abs*{x}^2\right)^2}\dd x.
    \end{equation}
    In particular, this measure is preserved by the (hyperbolic-)isometries of $\mathbb{B}$.

    Also note that for a $\mu$-integrable function $f\colon \HypDim\to \R$ for which there exists $g\colon \R_{\geq 0}\to \R$ such that $f(x)=g\left(\dist{x,\orig}\right)$,
    \begin{equation}
        \int_\HypDim f(x)\mu\left(\dd x\right) = \mathfrak{S}_{d-1}\int^\infty_0 g(r)\left(\sinh r\right)^{d-1}\dd r,
    \end{equation}
    where $\mathfrak{S}_{d-1}:=\sfrac{2\pi^{\frac{d}{2}}}{\Gamma\left(\frac{d}{2}\right)}$ equals the $\R^{d-1}$-Lebesgue measure of the Euclidean unit radius sphere (i.e. $\mathbb{S}^{d-1}$) in $\R^d$.

    \paragraph{Random connection model.}
    Given $\lambda\geq0$ let the vertex set of the RCM, $\eta$, be distributed as a Poisson point process on $\HypDim$ with intensity measure $\lambda\mu$, where $\mu$ is the hyperbolic measure on $\HypDim$. Given a measurable \emph{adjacency function} $\connf\colon \R_{\geq 0}\to\left[0,1\right]$, for each pair of vertices in $\eta$ (say, $\left\{x,y\right\}$), assign an edge independently of everything else with probability 
    \begin{equation}
        \connf\left(\dist{x,y}\right).
    \end{equation}
    We will sometimes write $\connf\left(x,y\right) := \connf\left(\dist{x,y}\right)$. The overloading of notation should cause minimal confusion. Note that since hyperbolic translations generate hyperbolic rotations (see for example \cite[Exercise~29.13]{martin2012foundations}), this is no great restriction on the family of adjacency functions we can consider if we want some sort of transitivity in the model. If $x$ and $y$ share an edge, we say $x$ and $y$ are \emph{adjacent}, and write $x\sim y$. The configuration of the vertex set and the edge set is then notated with $\xi$. We will also require Palm-type versions of $\eta$ and $\xi$. Given a finite set $\left\{x_1,\ldots,x_k\right\}\subset \HypDim$, we can define the augmented vertex set $\eta^{x_1,\ldots,x_k}= \eta\cup\left\{x_1,\ldots,x_k\right\}$ and define the augmented configuration $\xi^{x_1,\ldots,x_k}$ by using the augmented vertex set and including edges between any two vertices in $\eta^{x_1,\ldots,x_k}$ independently with probability given by $\connf$. The full details of this construction can be found in \cite{HeyHofLasMat19}.

\begin{figure}
    \centering
    \begin{subfigure}[b]{0.49\textwidth}
        \centering
        \includegraphics[width=\textwidth]{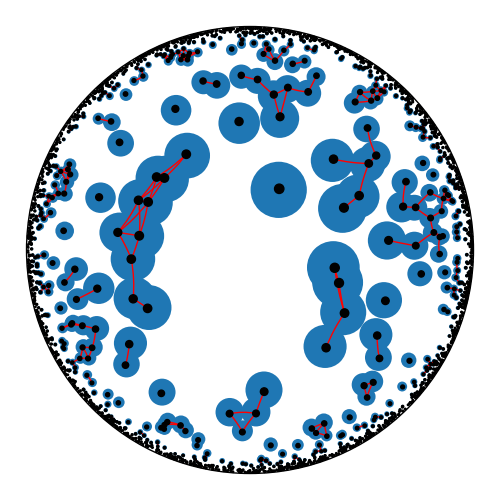}
        \caption{$\lambda=2.0$}
    \end{subfigure}
    \hfill
    \begin{subfigure}[b]{0.49\textwidth}
        \centering
        \includegraphics[width=\textwidth]{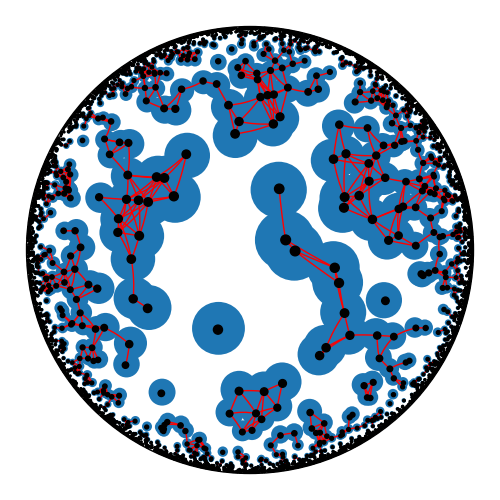}
        \caption{$\lambda=4.0$}
    \end{subfigure}
    \caption{Simulations of the RCM on $\HypTwo$ (represented with the Poincar{\'e} ball model) with adjacency function $\connf\left(r\right) = \Id\left\{r < r_0\right\}$ where $r_0=2\arsinh\left(\frac{1}{2\sqrt{\pi}}\right)\approx0.557$ is the radius of the unit $\mu$-volume disc. The blue region is the union of hyperbolic radius $\frac{1}{2}r_0$ discs centred on each vertex.}
    \label{fig:BooleanSimulation}
\end{figure}

    Given an index $L>0$, let $\connf_L$ be an adjacency function. In the examples we consider in Theorems~\ref{cor:BooleanResult} and \ref{cor:HeatKernelResult}, and in Assumption~\ref{Model_Assumption}, we shall come to view models with larger $L$ as being more `stretched-out' in that longer edges beome more likely. This adjacency function and an intensity $\lambda$ will define a RCM on $\HypDim$. Let $\mathbb{P}_{\lambda,L}$ and $\E_{\lambda,L}$ denote the probability and expectation associated with $\xi$ in the RCM on $\HypDim$ with intensity $\lambda$ and adjacency function $\connf_L$. The assumptions required for our results will impose constraints on the family of functions $\left\{\connf_L\right\}_{L>0}$ we can consider.

    Given two distinct vertices $x,y\in\HypDim$, we say $x$ is connected to $y$ in $\xi^{x,y}$ (written $\conn{x}{y}{\xi^{x,y}}$) if $x=y$ or there exists a finite sequence of edges in $\xi^{x,y}$ leading from $x$ to $y$. This leads to defining the useful \emph{two-point} function:
    \begin{equation}
        \tau_{\lambda,L}\left(x,y\right):= \mathbb{P}_{\lambda,L}\left(\conn{x}{y}{\xi^{x,y}}\right).
    \end{equation}

\paragraph{Critical Intensity.}
Our main question of these hyperbolic random connection models regards the existence of an infinite connected component (a cluster) in the resulting random graph. Let
\begin{align}
    \lambda_\mathrm{c}&:= \inf\left\{\lambda\geq0\colon \text{the RCM on }\HypDim\text{ with intensity }\lambda\text{ has an infinite cluster a.s}\right\},\\
    \lambda_u&:= \inf\left\{\lambda\geq0\colon \text{the RCM on }\HypDim\text{ with intensity }\lambda\text{ has a unique infinite cluster a.s}\right\}.
\end{align}
For RCMs on $\Rd$, \cite{MeeRoy96} proves that the analogous quantities are in fact equal, and \cite{chebunin2024uniqueness} proves that this equality survives even if the vertices are randomly and independently assigned `marks' that can change the adjacency probabilities (with some conditions). However, \cite{tykesson2007number} proves that for the Boolean disc RCM on $\HypTwo$ (see Definition~\ref{defn:BooleanDiscHyp}),
\begin{equation}
    0<\lambda_\mathrm{c}<\lambda_u<\infty.
\end{equation}
In fact \cite{tykesson2007number} proves that if the radius of the Boolean discs is sufficiently large (depending on the dimension $d$), $\lambda_\mathrm{c}<\lambda_u$ for the Boolean disc RCM on $\HypDim$ for any $d\geq 2$. This latter result is generalised in \cite{dickson2024NonUniqueness} to show that $\lambda_\mathrm{c}<\lambda_u$ for many `stretched-out' RCMs on $\HypDim$, even when marks are added to the vertices.

When we have a family of RCMs determined by the family of adjacency functions $\left\{\connf_L\right\}_{L>0}$, we will naturally have $\lambda_\mathrm{c}=\lambda_\mathrm{c}(L)$. In this paper we aim to describe the asymptotics of $\lambda_\mathrm{c}(L)$ as $L\to\infty$. 

    \subsection{Results}
    \label{sec:results}

Before introducing the most general form of our result, we present two concrete examples of it in action: the heat kernel RCM on $\HypThree$, and the Boolean disc RCM on $\HypDim$ (for any $d\geq 2$).

    \begin{definition}[Boolean disc RCM]\label{defn:BooleanDiscHyp}
        The Boolean disc RCM is defined by having the adjacency function
        \begin{equation}
            \connf_L(r) = \Id\left\{r<L\right\}.
        \end{equation}
    \end{definition}

    By $\mathfrak{S}_{d-1}$ we denote the $\R^{d-1}$-Lebesgue measure of the Euclidean unit radius sphere (i.e. $\mathbb{S}^{d-1}$) in $\R^d$, and note we will let $B_L\left(\orig\right)$ denote the hyperbolic radius $L$ ball centred on $\orig\in\HypDim$. Therefore $\mu\left(B_L\left(\orig\right)\right)$ denotes the hyperbolic measure of the hyperbolic radius $L$ ball.
    \begin{theorem}\label{cor:BooleanResult}
        For the Boolean disc RCM, as $L\to\infty$
        \begin{equation}\label{eqn:expectedDegreeBoolean}
            \mu\left(B_L(\orig)\right)\lambda_c(L) = 1 + \frac{\mathfrak{S}_{d-2}}{\mathfrak{S}_{d-1}}\frac{2^{d+2}}{d-1}\e^{-\frac{1}{2}\left(d-1\right)L} + o\left(\e^{-\frac{1}{2}\left(d-1\right)L}\right).
        \end{equation}
        Therefore
        \begin{equation}\label{eqn:BooleanExpansion}
            \lambda_c(L) = \frac{\left(d-1\right)2^{d-1}}{\mathfrak{S}_{d-1}}\e^{-\left(d-1\right)L}\left(1 + \left(1+o\left(1\right)\right)\begin{cases}
                \frac{16}{\pi}\e^{-\frac{1}{2}L}&\colon d=2\\
                8\e^{-L}&\colon d=3\\
                \frac{128}{3\pi}\e^{-\frac{3}{2}L}&\colon d=4\\
                32\e^{-2L}&\colon d=5\\
                \frac{\left(d-1\right)^2}{d-3}\e^{-2L}&\colon d\geq 6
            \end{cases}\right).
        \end{equation}
    \end{theorem}
    The order of the error term in \eqref{eqn:BooleanExpansion} becomes fixed for $d\geq5$ because the correction term for $\mu\left(B_L(\orig)\right)$ becomes relevant for $d=5$ and dominant for $d\geq 6$. There is no indication that this is related to $d=6$ being the upper critical dimension for Euclidean percolation models.

\begin{definition}[Heat Kernel RCM on $\HypThree$]
    Note that the heat kernel $K_3\colon \R_{>0}\times\HypThree\times\HypThree\to \R_{>0}$ is given by
    \begin{equation}
        K_3\left(t,x,y\right) = \frac{1}{\left(2\pi t\right)^{\frac{3}{2}}} \frac{\distThree{x,y}}{\sinh \distThree{x,y}} \exp\left(-\frac{t}{2} - \frac{\distThree{x,y}^2}{2t}\right),
    \end{equation}
    see \cite{davies1988heat}. Note (or see in the proof of Lemma~\ref{lem:heatkernelGood}) that $\int_\HypThree K_3\left(t,x,y\right)\mu\left(\dd x\right)=1$ for all $t>0$ and $y\in\HypThree$. Now let $0<\Acal_L\leq \left(2\pi L\right)^{\frac{3}{2}}\e^{\frac{1}{2}L}$ for all $L>0$. The heat kernel RCM on $\HypThree$ is defined by having the adjacency function
    \begin{equation}
        \connf_L\left(r\right) = \frac{\Acal_L}{\left(2\pi L\right)^{\frac{3}{2}}} \frac{r}{\sinh r} \exp\left(-\frac{1}{2}L - \frac{1}{2L}r^2\right)
    \end{equation}
    for all $L>0$. One can think of this as a hyperbolic analog of the RCM on $\Rd$ with a Gaussian adjacency function with variance parameter $L$ and amplitude $\Acal_L$. Like the Gaussian function on $\Rd$, this heat kernel satisfies the semi-group property, which will make it convenient to evaluate convolutions. The sole purpose of having $\Acal_L\leq \left(2\pi L\right)^{\frac{3}{2}}\e^{\frac{1}{2}L}$ is to ensure $\connf_L\left(x,y\right)\in\left[0,1\right]$ for all $L>0$ and $x,y\in\HypThree$. We have restricted to $d=3$ so that the heat kernel takes its most simple and explicit form.
\end{definition}

    \begin{theorem}\label{cor:HeatKernelResult}
        Consider the heat kernel RCM on $\HypThree$. As $L\to\infty$,
        \begin{equation}\label{eqn:heatkernelexpansion}
            \Acal_L\lambda_\mathrm{c}(L) = 1 + \frac{\Acal_L}{\left(6\pi\right)^\frac{3}{2}}L^{-\frac{3}{2}}\e^{-\frac{3}{2}L} + \frac{3}{2}\frac{\Acal_L}{\left(8\pi\right)^\frac{3}{2}}L^{-\frac{3}{2}}\e^{-2L} + \LandauBigO{\Acal_L L^{-\frac{3}{2}}\e^{-\frac{5}{2}L}+\Acal_L^2L^{-\frac{9}{2}}\e^{-\frac{5}{2}L}}.
        \end{equation}
        In particular, if $\connf_L(\orig,\orig)=1$ for all $L>0$ (i.e. $\Acal_L= \left(2\pi L\right)^{\frac{3}{2}}\e^{\frac{1}{2}L}$) then as $L\to\infty$
        \begin{equation}
            \lambda_\mathrm{c}(L) = \frac{1}{\left(2\pi\right)^{\frac{3}{2}}}L^{-\frac{3}{2}}\e^{-\frac{1}{2}L}\left( 1 + \frac{1}{3\sqrt{3}}\e^{-L} + \frac{3}{16}\e^{-\frac{3}{2}L} + \LandauBigO{L^{-\frac{3}{2}}\e^{-\frac{3}{2}L}}\right).
        \end{equation}
    \end{theorem}

    Note that which of the two last error terms in \eqref{eqn:heatkernelexpansion} is dominant depends upon the order of the amplitude term $\Acal_L$: $\Acal_LL^{-\frac{3}{2}}\e^{-\frac{5}{2}L}$ is dominant if $\Acal_L\ll L^3$, while $\Acal_L^2L^{-\frac{9}{2}}\e^{-\frac{5}{2}L}$ is dominant if $\Acal_L\gg L^3$.

\vspace{0.5cm}

While the above two theorems are for specific models, the following proposition provides a result for a wide class of RCM on $\HypDim$. We only require the following assumption.

\begin{assumptionp}{\textbf{A}}
\label{Model_Assumption}
We require that the family of adjacency functions $\left\{\connf_L\right\}_{L>0}$ satisfies:
        \begin{enumerate}[label=\textbf{(\ref{Model_Assumption}.\arabic*)}]
                \item\label{enum:AssumptionFiniteDegree} for all $L>0$,
        \begin{equation}
            0<\int^\infty_0\connf_L\left(r\right) \exp\left(\left(d-1\right)r\right)\dd r <\infty,
        \end{equation}

            \item\label{enum:AssumptionLongRange} for all $R>0$, 
            \begin{equation}
                \lim_{L\to\infty}\frac{\int^R_0\connf_L\left(r\right)\left(\sinh r\right)^{d-1}\dd r}{\int^\infty_0\connf_L\left(r\right)\left(\sinh r\right)^{d-1}\dd r} = 0.
            \end{equation}
        \end{enumerate}
        
    \end{assumptionp}

\begin{prop}\label{thm:CritIntensityExpansion}
    Suppose $\left\{\connf_L\right\}_{L>0}$ satisfies Assumption~\ref{Model_Assumption}. Then as $L\to\infty$,
    \begin{equation}\label{eqn:TheoremExpansion}
        \lambda_\mathrm{c}(L) \int_\HypDim\connf_L\left(x,\orig\right)\mu\left(\dd x\right)= 1 + o\left(1\right).
    \end{equation}
\end{prop}

\begin{remark}
    Mecke's formula (see Section~\ref{sec:Preliminaries}) implies that $\lambda_\mathrm{c}(L)\int_\HypDim\connf_L\left(x,\orig\right)\mu\left(\dd x\right)$ is equal to the expected degree of a typical vertex in the RCM. Similarly, the left hand sides of \eqref{eqn:expectedDegreeBoolean} and \eqref{eqn:heatkernelexpansion} are equal to the expected degree of a typical vertex in their respective RCMs. We can therefore view Theorems~\ref{cor:BooleanResult} and \ref{cor:HeatKernelResult} and Proposition~\ref{thm:CritIntensityExpansion} as results on the expected degree of a typical vertex in each RCM approaching $1$ as $L\to\infty$.
\end{remark}

\begin{remark}
Here are a few remarks on Assumption~\ref{Model_Assumption}.
\begin{enumerate}[label=\roman*)]

    \item Observe that Assumption~\ref{enum:AssumptionFiniteDegree} with Mecke's formula implies that $0<\E_{\lambda,L}\left[\deg{\orig}\right]<\infty$ for all $\lambda>0$ and $L>L_0$. If instead we had $\int_\HypDim\connf_L\left(x,\orig\right)\mu\left(\dd x\right)=\infty$ for some $L$, then $\E_{\lambda,L}\left[\deg{\orig}\right]=\infty$ for all $\lambda>0$ and $\lambda_\mathrm{c}(L)=0$. On the other hand, if $\int_\HypDim\connf_L\left(x,\orig\right)\mu\left(\dd x\right)=0$, then there are almost surely no edges at all and $\lambda_\mathrm{c}(L)=\infty$. This assumption therefore allows estimating $\lambda_\mathrm{c}$ to be an interesting question.

    \item In the context of Assumption~\ref{Model_Assumption} and Proposition~\ref{thm:CritIntensityExpansion} $L$ is just an index for a family of adjacency functions, and therefore Assumption~\ref{enum:AssumptionFiniteDegree} can be weakened to only hold for $L>L_0$ for some $L_0$. There are some places in the argument where having $L$ ``sufficiently large'' is important for an independent reason, and having this assumption hold for all $L>0$ makes it easier to identify these places.

    \item If one specifically considers adjacency functions of the form $\connf_L\left(r\right) = \connf\left(\frac{1}{L}r\right)$ for some fixed reference adjacency function $\connf$, then \cite[Lemma~A.3]{dickson2024NonUniqueness} proves that Assumption~\ref{enum:AssumptionLongRange} is satisfied as long as $\int^\infty_0\connf(r)\dd r>0$. Assumption~\ref{enum:AssumptionFiniteDegree} is then also satisfied as long as $\connf(r)$ is not $0$ almost everywhere, and decays faster than exponentially as $r\to\infty$. Increasing the parameter $L$ in this case corresponds to increasing the typical length scale of edges of the RCM in $\HypDim$. Therefore (up to an overall scaling of the intensity $\lambda$) increasing $L$ corresponds to increasing the (negative) curvature of the $\HypDim$ space. One may then interpret the $L\to\infty$ regime in this case as a `high curvature' regime.
\end{enumerate}
\end{remark}

Proposition~\ref{thm:CritIntensityExpansion} is actually a simplification of Theorem~\ref{thm:CritIntensityExpansionFull} later in this paper. While that result also requires Assumption~\ref{Model_Assumption}, it provides more detail to the error term in \eqref{eqn:TheoremExpansion}. In particular, a bound on the $o\left(1\right)$ term is provided that depends explicitly on the adjacency function $\connf_L$. The extra precision of Theorem~\ref{thm:CritIntensityExpansionFull} is required to derive the asymptotics in Theorems~\ref{cor:BooleanResult} and \ref{cor:HeatKernelResult}.

    \subsection{Background and Discussion}

    In this paper we give ourselves some flexibility in only asking for an asymptotic expression for the critical threshold in some perturbative regime, rather than asking for the exact value. This approach has been taken many times before. If we let $G_L$ be the graph on $\Z^d$ (with $d\geq 2$) with edge set connecting vertices $x,y\in\Z^d$ such that $\norm*{x-y}_\infty\leq L$, then \cite{penrose1993spread} proves that the critical probability satisfies
\begin{equation}
\label{eqn:BBPcriticaldegree}
    p_\mathrm{c}(L)\deg{G_L} = 1 + o\left(1\right)
\end{equation}
as $L\to\infty$. The quantity $p_\mathrm{c}\deg{G_L}$ can also be interpreted as the expected degree of the percolation sub-graph at criticality -- naturally drawing comparison to \eqref{eqn:TheoremExpansion}. In \cite{spanos2024spread}, the behaviour \eqref{eqn:BBPcriticaldegree} is proven for such spread-out versions of BBP on transitive graphs with superlinear polynomial growth.

The order of the error in \eqref{eqn:BBPcriticaldegree} has been refined for a number of cases. \cite{freiperkins2016lower} and \cite{hong2023upper} compare BBP to the `SIR epidemic' model to find matching lower and upper bounds for the error for spread-out BBP on $\Z^d$ for $d=2,3$, while \cite{hong2025lower} uses a similar approach to find a lower bound for $d=4,5,6$.

Such more precise estimates on the expected degree at criticality have also been found for spread-out BBP on $\Z^d$ for $d > 6$ in \cite{hofstadsakai2005critical} -- see \eqref{eqn:BBPexpansion} above. This work is one of many that have used lace expansion techniques to arrive at estimates for $p_\mathrm{c}$. For example, \cite{HofSla05} uses the lace expansion to get asymptotics for $p_\mathrm{c}$ for BBP on nearest-neighbour $\Z^d$ and the $d$-dimensional hypercube as $d\to\infty$. In \cite{HeyMat19b} asymptotics for the critical probability of site percolation on nearest-neighbour $\Z^d$ as $d\to \infty$ are similarly proven.

For RCMs with background space $\Rd$, \cite{dickson2025expansion} uses lace expansion results from \cite{HeyHofLasMat19} to derive asymptotics for the critical intensity $\lambda_\mathrm{c}$ as $d\to\infty$. In particular, if $\connf=\connf_d$ is given by a Gaussian function with total mass $1$ and variance $\sigma^2(d)$ such that $\liminf_{d\to\infty}\connf(\orig,\orig)^\frac{1}{2}>0$, then
\begin{equation}
    \lambda_\mathrm{c}(d) = 1 + \left(6\pi\sigma^2\right)^{-\frac{d}{2}} + \frac{3}{2}\left(8\pi\sigma^2\right)^{-\frac{d}{2}} + 2\left(10\pi\sigma^2\right)^{-\frac{d}{2}} + \LandauBigO{\left(12\pi\sigma^2\right)^{-\frac{d}{2}}}
\end{equation}
as $d\to\infty$.

Before we compare the result in this paper to results regarding RCMs on $\Rd$, we should highlight a difference between $\lambda_\mathrm{c}$ for $\Rd$-RCMs and $\lambda_\mathrm{c}$ for $\HypDim$-RCMs. Recall that $\tau_{\lambda,L}(x,y)$ is the two-point function that returns the probability that $x$ is connected to $y$ in the RCM. We can use this two-point function to define a two-point convolution operator $\Optau_{\lambda,L}$ (for details, see Section~\ref{sec:Preliminaries}). The operator norm of this operator can depend upon the spaces it is defined on: for $p\in\left[1,\infty\right]$, $\norm*{\Optau_{\lambda,L}}_{p\to p}$ denotes the operator norm of $\Optau_{\lambda,L}\colon L^p\to L^p$. We can then define
\begin{align}
    \lambda_{2\to 2}(L)&:= \inf\left\{\lambda\geq0\colon \norm*{\Optau_{\lambda,L}}_{2\to 2} = \infty\right\},\\
    \lambda_{1\to 1}(L)&:= \inf\left\{\lambda\geq0\colon \norm*{\Optau_{\lambda,L}}_{1\to 1} = \infty\right\}.
\end{align}
For $\Rd$-RCMs one can use the Fourier transform to show $\lambda_{2\to 2}=\lambda_{1\to 1}$. On the other hand, \cite{dickson2024NonUniqueness} shows that for $\HypDim$-RCMs satisfying Assumption~\ref{enum:AssumptionLongRange} with sufficiently large $L$, we have $\lambda_{2\to 2}(L) > \lambda_{1\to 1}(L)$. For $\Rd$-RCMs, \cite{Mee95} proved that $\lambda_\mathrm{c}=\lambda_{1\to 1}$ while \cite{dickson2024NonUniqueness} proved this same equality for $\HypDim$-RCMs with sufficiently large $L$.

    While the central result and high-level proof outline of this paper strongly resembles the central result of \cite{dickson2025expansion}, it is notable the inequality $\lambda_{2\to 2}>\lambda_{1\to 1}$ means that the difficulties are reversed. In the $\Rd$-RCM case, the derivation of the lace expansion was challenging because a sophisticated bootstrap argument was required to extend a property from $\lambda<\lambda_{2\to 2}$ to $\lambda\leq \lambda_{2\to 2}$. Once the uniform convergence to an Ornstein-Zernike equation was established, it was relatively clear (if involved) to expand $\lambda_{2\to 2}$ because the Fourier transform diagonalises the equation. On the other hand, deriving the lace expansion Ornstein-Zernike equation for RCMs on $\HypDim$ is simple for $\lambda\leq \lambda_{1\to 1}$, because the triangle condition is already established in \cite{dickson2024NonUniqueness}. However because we want an expansion for $\lambda_{1\to 1}$ -- and not $\lambda_{2\to 2}$ -- the spherical transform does not give the shortcut that the Fourier transform does in $\Rd$. 

    It should, nevertheless, be possible to derive an expansion for $\lambda_{2\to 2}$ in this $\HypDim$ case. Like for the $\Rd$ case, the main challenge is the derivation of the Ornstein-Zernike equation on the whole domain $\lambda\leq \lambda_{2\to 2}$. Once this is achieved, the spherical transform can then be used to extract $\lambda_{2\to 2}$. While the value $\lambda_{2\to 2}$ currently has less probabilistic interest than the $\lambda_{\mathrm{c}}$ critical intensity, \cite{dickson2024NonUniqueness} shows that it provides a lower bound for the uniqueness threshold $\lambda_\mathrm{u}$.

    Observe that the corrections to the leading term for both the heat kernel RCM on $\HypThree$ and the Boolean disc RCM on $\HypDim$ are exponentially small in $L$, while the corrections to the leading term for spread-out BBP on $\Z^d$ for $d>6$ (see \cite{hofstadsakai2005critical} or \eqref{eqn:BBPexpansion}) follow a power law. This exponential decay is not a consequence of Theorem~\ref{thm:CritIntensityExpansion}, because one can always `slow down' $L$ and Assumption~\ref{Model_Assumption} will still hold. There would need to be a more concrete map between RCMs on $\Rd$ to RCMs on $\HypDim$ for a theorem along the lines of ``corrections to the leading term for RCMs on $\HypDim$ are smaller than corrections to the leading term for RCMs on $\Rd$'' to be made. The closest we have to that here is comparing \eqref{eqn:BBPexpansion} to Theorem~\ref{cor:BooleanResult}, which suggests there may be something to explore here.

    Throughout this paper we have fixed the curvature of $\HypDim$ to be $-1$. It would be an interesting avenue of investigation to see how the approximations of $\lambda_\mathrm{c}$ behave as the curvature of $\HypDim$ approaches $0$. Does the approximation approach something meaningful for $\lambda_\mathrm{c}$ of an RCM on $\Rd$? For $2\leq d \leq 6$, there is good reason to believe the answer is no. This is because the lace expansion for RCMs on $\Rd$ (fundamentally the same as the lace expansion on $\HypDim$) should not converge for $d\leq 6$. It may be that the hidden constants in $o\left(1\right)$ in Proposition~\ref{thm:CritIntensityExpansion} grow and eventually diverge as the curvature vanishes.

\subsection{Outline of the Paper and the Argument}

In Section~\ref{sec:Preliminaries} we establish some preliminary results: specifically we recall Mecke's formula (which we use many times throughout the paper) and provide some simple properties of the operators we are considering in this paper. Note that in Theorem~\ref{thm:TriangleSmall} we also recall a result from \cite{dickson2024NonUniqueness} that will be crucial to our argument -- in particular we already have a bound on the triangle diagram at $\lambda_\mathrm{c}$.

Theorems~\ref{cor:BooleanResult} and \ref{cor:HeatKernelResult} and Proposition~\ref{thm:CritIntensityExpansion} are all corollaries of Theorem~\ref{thm:CritIntensityExpansionFull}. Sections~\ref{sec:LaceExpansionCritIntensity} to \ref{sec:estimatingspectralradius} lay the ground-work for and culminate in a proof of Theorem~\ref{thm:CritIntensityExpansionFull}. In Section~\ref{sec:LaceExpansionCritIntensity} we provide the lace expansion for the hyperbolic RCM for $\lambda\leq \lambda_\mathrm{c}$, leading to Proposition~\ref{prop:LaceExpansion}. This proceeds very similarly to the derivation of the lace expansion for RCMs on $\Rd$ (see \cite{HeyHofLasMat19,DicHey2022triangle}), with a major simplification in the fact that we already have the bound on the triangle diagram from \cite{dickson2024NonUniqueness}.

We then outline how the Ornstein-Zernike equation resulting from the lace expansion can be used to extract the value of $\lambda_\mathrm{c}$ in Section~\ref{sec:characterisingPercolationThreshold}. In particular, this involves understanding the $L^1\to L^1$ spectral radius of the sum of the adjacency operator and the lace expansion coefficient operator.

The issue now is that the lace expansion coefficient operator is an implicit object defined via the model. In Section~\ref{sec:ExpandLaceCoeff} we replicate the arguments of \cite{dickson2025expansion} to obtain an approximation of these lace expansion coefficients in terms of the adjacency function. The main difference between the argument in \cite{dickson2025expansion} and the argument here is that on $\HypDim$ we don't have the usual Fourier transform and the associated tools. We do, however, have access to the spherical transform -- an analog of the Fourier transform for symmetric spaces. This plays a crucial role in \cite{dickson2024NonUniqueness}, and we use it here to prove that the intermediate results in \cite{dickson2025expansion} that made use of the Fourier transform also hold in our case.

In Section~\ref{sec:estimatingspectralradius} we use this approximation of the lace expansion coefficients to get an approximation of the $L^1\to L^1$ spectral radius in terms of $\lambda_\mathrm{c}$ and the adjacency function. This gives $\lambda_\mathrm{c}$ as the solution of a polynomial equation (of which one term is only known asymptotically), and we use this to get the asymptotic behaviour of $\lambda_\mathrm{c}$. This completes the proof of Theorem~\ref{thm:CritIntensityExpansionFull}. Proposition~\ref{thm:CritIntensityExpansion} then follows directly.

In Section~\ref{sec:CalcSpecific} we show how the terms in Theorem~\ref{thm:CritIntensityExpansionFull} can be simplified in three cases: when we restrict to what we call `non-negative definite' RCMs, when we restrict to the heat kernel model (leading to Theorem~\ref{cor:HeatKernelResult}), and when we restrict to the Boolean disc RCM (leading to Theorem~\ref{cor:BooleanResult}).

Appendix~\ref{app:PrelimProofs} contains the proofs for most of the preliminary lemmas in Section~\ref{sec:Preliminaries}. In Appendix~\ref{app:rangesphericaltransform} we prove an elementary property of the spherical transform that applies for the functions we are considering and is used multiple times -- specifically that the range of the transformed functions are subsets of the real line. In Appendix~\ref{app:hyperbolictriangles} we give some standard properties of hyperbolic triangles which are used in Section~\ref{sec:CalcBooleanDisc} to study the Boolean disc RCM.

\section{Preliminaries}
\label{sec:Preliminaries}
\paragraph{Mecke's Formula.}

    Mecke's formula gives us a way to formulate some expectations of integrals (or sums) as integrals of expectations. Given $m\in\N$ and a measurable function $f=f\left(\xi,x_1,\ldots,x_m\right)\in \R_{\geq 0}$,
    \begin{multline}
        \mathbb{E}_{\lambda,L}\left[\sum_{\left(x_1,\ldots,x_m\right)\in \eta^{(m)}} f\left(\xi,x_1,\ldots,x_m\right)\right] \\= \lambda^m \int_{\left(\HypDim\right)^m}\mathbb{E}_{\lambda,L}\left[f\left(\xi^{x_1,\ldots,x_m},x_1,\ldots,x_m\right)\right]\mu\left(\dd x_1\right)\ldots \mu\left(\dd x_m\right),
    \end{multline}
    where $\eta^{(m)} := \left\{\left(y_1,\ldots,y_m\right)\in\eta^m\colon x_i\ne x_j\text{ for }i\ne j\right\}$. For details, see \cite[\S4]{LasPen17}.

    Notable examples of the use of Mecke's formula are to get expressions for the expected degree and expected cluster size of a RCM. Let $\deg{\orig}:= \#\left\{y\in\eta\colon \adja{y}{\orig}{\xi^\orig}\right\}$ and $\C\left(\orig\right):= \left\{\orig\right\}\cup\left\{y\in\eta\colon\conn{y}{\orig}{\xi^\orig}\right\}$. Then Mecke's formula implies
    \begin{align}
        \mathbb{E}_{\lambda,L}\left[\deg{\orig}\right] = \mathbb{E}_{\lambda,L}\left[\sum_{x\in \eta} \Id\left\{\adja{x}{\orig}{\xi^\orig}\right\}\right] &= \lambda\int_\HypDim\connf_L\left(x,\orig\right)\mu\left(\dd x\right),\\
        \mathbb{E}_{\lambda,L}\left[\#\C\left(\orig\right)\right] = 1+ \mathbb{E}_{\lambda,L}\left[\sum_{x\in \eta} \Id\left\{\conn{x}{\orig}{\xi^\orig}\right\}\right] &= 1+\lambda\int_\HypDim\tau_{\lambda,L}\left(x,\orig\right)\mu\left(\dd x\right).
    \end{align}

\paragraph{Operators.}

Let $G$ denote the convolution operator associated with the function $g\colon \HypDim\times\HypDim\to\Complex$. That is, first define
\begin{equation}
    \mathcal{D}\left(G\right) := \left\{f\colon \HypDim\to \mathbb{C} \;\vert\; \int_\HypDim\abs*{g\left(x,y\right)}\abs*{f(y)}\mu\left(\dd y\right)<\infty \text{ for } \mu\text{-almost every }x\in\HypDim\right\}.
\end{equation}
Then $G\colon \mathcal{D}\left(G\right)\to \Complex^{\HypDim}$ is the linear operator defined by
\begin{equation}
    \left(Gf\right)\left(x\right):= \int_{\HypDim}g\left(x,y\right)f(y)\mu\left(\dd y\right).
\end{equation}
Then given $p\in\left[1,\infty\right]$, define the operator norm
\begin{equation}
    \norm*{G}_{p\to p} := \begin{cases}
        \sup_{f\colon \norm*{f}_p=1}\norm*{Gf}_{p} &\colon L^p\left(\HypDim\right)\subset \mathcal{D}\left(G\right),\\
        \infty &\colon \text{otherwise,}
    \end{cases}
\end{equation}
and the spectral radius
\begin{equation}
\label{eqn:Gelfand}
    \rho_{p\to p}\left(G\right) := \lim_{n\to\infty}\norm*{G^n}_{p\to p}^{\frac{1}{n}}.
\end{equation}
Observe that this limit exists (in $\left[0,\infty\right]$) because the operator norm is sub-multiplicative. 

Lemmas~\ref{lem:ConvolutionOperator} and \ref{lem:spectralradiusofPositive} allow us to get explicit formula for the $L^1\to L^1$ operator norm and spectral radius of certain convolution operators. They are proven in Appendix~\ref{app:PrelimProofs}.

\begin{restatable}{lemma}{ConvolutionOperator}
\label{lem:ConvolutionOperator}
    The convolution operator $G$ satisfies
    \begin{equation}
        \label{eqn:OneOneNorm}
        \norm*{G}_{1\to 1} = \esssup_{y\in \HypDim}\int_{\HypDim}\abs*{g(x,y)}\mu\left(\dd x\right).
    \end{equation}
    If $g\left(x,y\right) = g\left(y,x\right)$ for all $x,y\in\HypDim$, then for all $p\in\left[1,\infty\right]$
    \begin{equation}
        \norm*{G}_{p\to p} = \norm*{G}_{\frac{p}{p-1}\to\frac{p}{p-1}},
    \end{equation}
    and $p\mapsto \norm*{G}_{p\to p}$ is a decreasing function of $p$ on $\left[1,2\right]$ and increasing on $\left[2,\infty\right]$.
\end{restatable}

\begin{definition}
    A function $f\colon \HypDim\times\HypDim\to\R$ is \emph{isometry invariant} if, for any isometry $\iota$ of $\HypDim$,
    \begin{equation}
        f\left(\iota(x),\iota(y)\right) = f(x,y)
    \end{equation}
    for all $x,y\in\HypDim$. Observe that if $f$ is isometry invariant, then $f(x,y)=f\left(\dist{x,y}\right)$ only: $f$ is constant on the hyperbolic sphere $\left\{w\in \HypDim \colon \dist{w,x}=\dist{y,x}\right\}$.
\end{definition}

\begin{restatable}{lemma}{spectralradiusofPositive}
\label{lem:spectralradiusofPositive}
    If $g(x,y)$ is isometry invariant and $g(x,y)\geq 0$ for all $x,y\in\HypDim$, then
    \begin{equation}
        \rho_{1\to1 }\left(G\right) = \int_\HypDim g\left(x,\orig\right) \mu\left(\dd x\right) = \norm*{G}_{1\to 1}.
    \end{equation}    
\end{restatable}

In Section~\ref{sec:ExpandLaceCoeff} we will want to bound the perturbation to the spectral radius when we perturb the operator. Given $p\in\left[1,\infty\right]$, the spectrum of the bounded operator $A\colon L^p\left(\HypDim\right)\to L^p\left(\HypDim\right)$ is defined to be
\begin{equation}
    \sigma_{p\to p}\left(A\right):= \left\{z\in\Complex \text{ s.t. } \left(A-z\Id\right)\colon L^p\left(\HypDim\right)\to L^p\left(\HypDim\right)\text{ is not invertible}\right\}.
\end{equation}
By Gelfand's formula,
\begin{equation}
    \rho_{p\to p}(A) = \sup\left\{\abs*{z}\colon z\in \sigma_{p\to p}(A)\right\}.
\end{equation}

The following lemma is an instance of a more general result, but we will only require this $L^1\to L^1$ version in this paper. In this lemma ``$\mathrm{dist}$'' refers to the Hausdorff distance with respect to the Euclidean distance on $\Complex$.
\begin{lemma}
\label{lem:SpectrumPerturbation}
    Let $A,B\colon L^1\left(\HypDim\right)\to L^1\left(\HypDim\right)$ be bounded commuting linear operators. Then
    \begin{equation}
        \mathrm{dist}\left(\sigma_{1\to 1}\left(A+B\right),\sigma_{1\to 1}\left(A\right)\right) \leq \rho_{1\to 1}\left(B\right) \leq \norm*{B}_{1\to 1}.
    \end{equation}
    In particular,
    \begin{equation}
        \rho_{1\to 1}\left(A+B\right) \leq \rho_{1\to 1}\left(A\right) + \norm*{B}_{1\to 1}.
    \end{equation}
\end{lemma}

\begin{proof}
    This follows from the standard result of \cite[Chapter~IV, \S3, Theorem~3.6]{kato1995perturbation}.
\end{proof}

The following lemma gives us the means to easily tackle the ``commuting'' condition of the previous lemma. It is proven in Appendix~\ref{app:PrelimProofs}.

\begin{restatable}{lemma}{IsoInvtoCommute}
\label{lem:IsoInvtoCommute}
    Let $p\in\left[1,\infty\right]$ and $f,g\colon\HypDim\times\HypDim\to\Complex$ be isometry invariant functions such that the associated convolution operators $F,G\colon L^p\left(\HypDim\right)\to L^p\left(\HypDim\right)$ are bounded operators. Then
    \begin{equation}
        FG=GF
    \end{equation}
    on $L^p\left(\HypDim\right)$.
\end{restatable}

\begin{definition}
Recall that for $\lambda\geq0$ and $L>0$, the \emph{two-point function} $\tau_{\lambda,L}\colon \HypDim\times \HypDim \to [0,1]$ is given by
\begin{equation}
    \tau_{\lambda,L}\left(x,y\right) := \mathbb{P}_{\lambda,L}\left(\conn{x}{y}{\xi^{x,y}}\right).
\end{equation}
Observe that from the isometry invariance of $\connf_L$ and $\mu$, $\tau_{\lambda,L}$ is also isometry invariant. 
We can define the operators $\Opconnf_L\colon \mathcal{D}\left(\Opconnf_L\right)\to \Complex^{\HypDim}$ and $\Optau_{\lambda,L}\colon \mathcal{D}\left(\Optau_{\lambda,L}\right)\to \Complex^{\HypDim}$ as the convolution operators associated with the functions $\connf_L$ and $\tau_{\lambda,L}$ respectively. That means they act on test functions $f$ by
\begin{equation}
    \left(\Opconnf_Lf\right)\left(x\right) = \int_\HypDim\connf_L\left(x,y\right)f\left(y\right)\mu\left(\dd y\right),\qquad \left(\Optau_{\lambda,L}f\right)\left(x\right) = \int_\HypDim\tau_{\lambda,L}\left(x,y\right)f\left(y\right)\mu\left(\dd y\right) 
\end{equation}

Given $\lambda\geq0$ and $L>0$, we define the \emph{triangle diagram} by
\begin{equation}
    \triangle_{\lambda,L}:= \lambda^2\esssup_{x\in\HypDim}\tau_{\lambda,L}^{\star 3}\left(x,\orig\right).
\end{equation}
To bound this triangle diagram at $\lambda_\mathrm{c}$, it will be convenient to define
\begin{equation}
    \label{eqn:BetaDefinition}
        \beta(L) := \left(\frac{\norm*{\Opconnf_L}_{2\to 2}}{\norm*{\Opconnf_L}_{1\to 1}}\cdot\frac{\connf_L^{\star 2}\left(\orig,\orig\right)}{\norm*{\Opconnf_L}_{1\to 1}}\right)^\frac{1}{2}.
\end{equation}

\end{definition}

The following theorem from \cite{dickson2024NonUniqueness} will be a cornerstone that allows us to bound the critical triangle diagram, and therefore vastly simplify the derivation of the lace expansion convolution equation in Proposition~\ref{prop:LaceExpansion}.

\begin{theorem}[{\cite[Lemmas~4.4, 4.7]{dickson2024NonUniqueness}}]
\label{thm:TriangleSmall}
    Let $d\geq 2$ and suppose $\left\{\connf_L\right\}_{L>0}$ satisfies Assumption~\ref{Model_Assumption}. Then as $L\to\infty$,
    \begin{equation}\label{eqn:normRatio}
        \frac{\norm*{\Opconnf_L}_{2\to 2}}{\norm*{\Opconnf_L}_{1\to 1}}\to 0
    \end{equation}
    and 
    \begin{equation}\label{eqn:betalimit}
        \beta(L)=\LandauBigO{\sqrt{\frac{\norm*{\Opconnf_L}_{2\to 2}}{\norm*{\Opconnf_L}_{1\to 1}}}}.
    \end{equation}
    Furthermore, as $L\to\infty$
    \begin{align}
        \lambda_\mathrm{c}(L)\norm*{\Opconnf_L}_{1\to 1} &=\LandauBigO{1},\label{eqn:criticalintensityBound}\\
        \lambda_c(L)\norm*{\Opconnf_L}_{2\to 2} &= \LandauBigO{\frac{\norm*{\Opconnf_L}_{2\to 2}}{\norm*{\Opconnf_L}_{1\to 1}}},\label{eqn:critcalIntensityLTwoBound}\\
        \triangle_{\lambda_{\mathrm{c}}(L),L}&=\LandauBigO{\beta(L)^2}.
    \end{align}
\end{theorem}

\begin{remark}
    The results in \cite{dickson2024NonUniqueness} are stated under a slightly different set of assumptions. In that reference the family of adjacency functions $\left\{\connf_L\right\}_{L>0}$ is assumed to be of a specific form. It is assumed that there exists a reference function $\connf\colon\R_{\geq 0}\to\left[0,1\right]$ and a family of bijective non-decreasing scaling functions $\sigma_L\colon \R_{\geq 0}\to\R_{\geq 0}$ such that $\lim_{L\to\infty}\sigma_L(r)=\infty$ for all $r>0$ and 
    \begin{equation}
        \connf_L\left(r\right) = \connf\left(\sigma^{-1}_L(r)\right).
    \end{equation}
    Two choices of scaling functions highlighted in \cite{dickson2024NonUniqueness} are the \emph{length-linear} scaling (where $\sigma_L(r)=Lr$) and the \emph{volume-linear} scaling (that satisfies $\mu\left(B_{\sigma_L(r)}(\orig)\right) = L\mu\left(B_r(\orig)\right)$). This volume-linear scaling becomes relevant for \emph{marked} RCMs on $\HypDim$. These are models where the vertices are independently given some mark, and the adjacency function also depends on the two marks of the vertices in question. The volume-linear scaling function structure is important for the argument when there are infinitely many possible marks -- it ensures that the most relevant marks stay the most relevant as $L\to\infty$.

    In our application of these results, the scaling function structure is not required because there are no marks (or equivalently, only one mark). The limit \eqref{eqn:normRatio} follows from writing $\norm*{\Opconnf_L}_{2\to 2}$ in terms of the spherical transform and using Assumption~\ref{enum:AssumptionLongRange}. Then \eqref{eqn:betalimit} follows by observing that
    \begin{equation}
        \connf_L^{\star 2}\left(\orig,\orig\right) = \int_\HypDim\connf_L\left(x,\orig\right)^2\mu\left(\dd x\right) \leq \int_\HypDim\connf_L\left(x,\orig\right)\mu\left(\dd x\right) = \norm*{\Opconnf_L}_{1\to 1}.
    \end{equation}
    
    The bounds \eqref{eqn:criticalintensityBound} and \eqref{eqn:critcalIntensityLTwoBound} follow by coupling a binary branching process with a subset of the hyperbolic RCM graph. This argument requires that a sufficient fraction of the total mass of the adjacency function lies beyond a given radius (which follows from Assumption~\ref{enum:AssumptionLongRange}). By using Mecke's formula, the probability that there exists a neighbour of $\orig$ in a given region can be bounded below by $1-\exp\left(-c\lambda\norm*{\Opconnf_L}_{1\to 1}\right)$ for some constant $c>0$, and therefore the binary branching process survives with high probability if $\lambda\norm*{\Opconnf_L}_{1\to 1}$ is sufficiently high (independently of $L$). The coupling then implies that the RCM has an infinite connected component if $\lambda\norm*{\Opconnf_L}_{1\to 1}$ is sufficiently high.

    Mecke's formula with a first-step decomposition implies that 
    \begin{equation}
        \tau_{\lambda,L} \leq \connf_L + \lambda \connf_L\star\tau_{\lambda,L}
    \end{equation}
    pointwise (see \cite[Observation~4.3]{HeyHofLasMat19}). Applying this twice to $\tau_{\lambda,L}^{\star 3}$ and converting this into operators implies that for any $f_1,f_2\in L^1\left(\HypDim\right)\cap L^2\left(\HypDim\right)$ such that $f_1,f_2\geq 0$,
    \begin{align}
        &\lambda^2\int_{\HypDim}\int_\HypDim f_1\left(x\right)\tau_{\lambda,L}^{\star 3}\left(x,y\right) f_2\left(y\right)\mu\left(y\right)\mu\left(x\right)\nonumber\\
        &\hspace{1cm}\leq \left(\lambda\norm*{\Optau_{\lambda,L}}_{2\to 2} + 2\lambda^2\norm*{\Optau_{\lambda,L}}_{2\to 2}^2 + \lambda^3\norm*{\Optau_{\lambda,L}}_{2\to 2}^3\right)\lambda\norm*{\Opconnf_Lf_1}_2 \norm*{\Opconnf_Lf_2}_2\nonumber\\
        &\hspace{1cm}\leq \left(\lambda\norm*{\Optau_{\lambda,L}}_{2\to 2} + 2\lambda^2\norm*{\Optau_{\lambda,L}}_{2\to 2}^2 + \lambda^3\norm*{\Optau_{\lambda,L}}_{2\to 2}^3\right)\lambda\int_\HypDim\connf_L\left(x,\orig\right)^2\mu\left(\dd x\right)\norm*{f_1}_1\norm*{f_2}_1.
    \end{align}
    This last inequality followed from a supremum bound. By suitably choosing $f_1$ and $f_2$, one can show that 
    \begin{equation}
        \triangle_{\lambda,L} \leq \left(\lambda\norm*{\Optau_{\lambda,L}}_{2\to 2} + 2\lambda^2\norm*{\Optau_{\lambda,L}}_{2\to 2}^2 + \lambda^3\norm*{\Optau_{\lambda,L}}_{2\to 2}^3\right)\lambda\int_\HypDim\connf_L\left(x,\orig\right)^2\mu\left(\dd x\right).
    \end{equation}
    By bounding $\lambda\tau_{\lambda,L}$ with the Green's function of the branching process with jump kernel $\lambda\connf_L$, one can show that
    \begin{equation}
        \lambda\norm*{\Optau_{\lambda,L}}_{2\to 2} \leq \sum^\infty_{n=1}\lambda^n\norm*{\Opconnf_L}^n_{2\to 2},
    \end{equation}
    and therefore $\triangle_{\lambda_{\mathrm{c}}(L),L}=\LandauBigO{\beta(L)^2}$ holds by applying \eqref{eqn:criticalintensityBound} and \eqref{eqn:critcalIntensityLTwoBound}.
\end{remark}

\section{Lace Expansion for the Critical Intensity}
\label{sec:LaceExpansionCritIntensity}

In this section we will use the lace expansion to derive a convolution equation \eqref{eqn:OZE} (an Ornstein-Zernike equation) for $\Optau_{\lambda,L}$, $\Opconnf_L$, and some other operator $\OpLace_{\lambda,L}$. It is likely that some minor variant of Proposition~\ref{prop:LaceExpansion} will hold for some $\lambda>\lambda_\mathrm{c}$, but the proof is made significantly simpler by the fact that we aren't pushing for a $\lambda$-optimal version -- we only need it for $\lambda\leq \lambda_\mathrm{c}$. The following lemma hints at why we want to formulate these operators as $L^1\to L^1$ operators.

\begin{lemma}\label{lem:LambdaCequalsLambdaT}
    Suppose Assumption~\ref{Model_Assumption} holds. Then for sufficiently large $L$,
    \begin{equation}
        \lambda_\mathrm{c}(L) = \inf\left\{\lambda\geq0\colon \norm*{\Optau_{\lambda,L}}_{1\to 1}=\infty\right\}.
    \end{equation}
\end{lemma}

\begin{proof}
    In \cite[Theorem~2.3]{dickson2024NonUniqueness} it is proven under (an unimportant modification of) Assumption~\ref{Model_Assumption} that for sufficiently large $L$,
    \begin{equation}
        \lambda_\mathrm{c}(L) = \inf\left\{\lambda\geq0\colon \int_\HypDim\tau_{\lambda,L}(x,\orig)\mu\left(\dd x\right)=\infty\right\}.
    \end{equation}
    Lemma~\ref{lem:spectralradiusofPositive} then immediately concludes the proof.
\end{proof}

We now prepare the definitions needed for the lace expansion.
\begin{definition}
\label{defn:lacesetup}
Given $x \in \HypDim$ and locally finite $A \subset \HypDim$, define 
\begin{equation}
    \bar\connf(A,x) :=\prod_{y\in A}(1-\connf(x,y)).\label{eq:LE:def:thinning_probability}
\end{equation}
Then define $\thinning{\eta}{A}$ as an \emph{$A$-thinning of $\eta$} by keeping a point $w \in \eta$ as a point of $\thinning{\eta}{A}$ with probability $\bar\connf(A,w)$ independently of all other points of $\eta$. That is, we keep a vertex if it would not form an edge with any point in $A$. We similarly define $\thinning{\eta}{A}^x$ as an $A$-thinning of $\eta^x$.

Given a vertex set, say $\eta'$, we denote by $\xi\left[\eta'\right]$ the RCM graph restricted to the vertices in $\eta'$. Using thinning and this notation, we can define the event 
\begin{equation}
    \left\{\xconn{x}{y}{\xi}{A}\right\}=\left\{\conn{x}{y}{\xi}\right\}\cap\left\{\nconn{x}{y}{\xi[\thinning{\eta}{A} \cup\{x\}]} \right\}.
\end{equation} 
That is, $\{\xconn{x}{y}{\xi}{A}\}$ is the event that the vertex $x$ is connected to the vertex $y$ in $\xi$, but that this connection is broken by an $A$-thinning of $\eta\setminus\{x\}$.

We can also use thinnings to define \emph{pivotal} points. Given $x,y\in\HypDim$ and a configuration $\xi$, we say $u\in\piv{y,x,\xi}\subset\HypDim$ if $\left\{\conn{y}{x}{\xi^{x,y}}\right\}\cap\left\{\nconn{y}{x}{\xi\left[\eta\setminus\left\{u\right\}\right]}\right\}$ occurs. That is, every path on $\xi^{x,y}$ connecting $x$ and $y$ uses the vertex $u$.

Now define
\begin{equation}
    E(x,y;A,\xi) := \left\{\xconn{x}{y}{\xi}{A} \right\} \cap \left\{\nexists w \in \piv{x,y;\xi}\colon \xconn{x}{w}{\xi}{A} \right\}, \label{eq:LE:def:E_event}
\end{equation}
for a locally finite set $A \subset\HypDim$, and $x,y\in \HypDim$. That is, the vertex $x$ is connected to the vertex $y$ in $\xi$, but that this connection is broken \emph{after} the last pivotal point from $x$ to $y$ and not before.

We also say two vertices $x,y\in\HypDim$ are \emph{two-connected} in $\xi^{x,y}$ if $x$ and $y$ are adjacent or if there are two paths in $\xi^{x,y}$ that only share $x$ and $y$ as common vertices. We denote this event as $\left\{\dconn{x}{y}{\xi^{x,y}}\right\}$.
\end{definition}

\begin{definition}
\label{def:LE:lace_expansion_coefficients}
For $\lambda\geq 0$ and $L>0$ let $\left\{\xi_i\right\}_{i\in\N_0}$ be a sequence of independent $\mathbb{P}_{\lambda,L}$-distributed RCM configurations. Then for $n\in\N$, and $x,y\in\HypDim$, we define `lace expansion coefficient' functions
\begin{align}
    \pi_{\lambda,L}^{(0)}(x,y) &:= \plal \left(\dconn{y}{x}{\xi_0^{y,x}}\right) - \connf_L(x,y), \label{eq:LE:Pi0_def} \\
	\pi_{\lambda,L}^{(n)}(x,y) &:= \lambda^n \int_{\HypDim}\ldots\int_{\HypDim} \plal \left( \left\{\dconn{y}{u_0}{\xi^{y, u_0}_{0}}\right\}\right.\nonumber\\
					& \hspace{4cm}\left.\cap \bigcap_{i=1}^{n} E\left(u_{i-1},u_i; \C_{i-1}, \xi^{u_{i-1}, u_i}_{i}\right) \right) \mu\left(\dd u_0\right) \ldots \mu\left(\dd u_{n-1}\right), \label{eq:LE:Pin_def}
\end{align}
where $u_n=x$, $u_{-1}=y$, and $\C_{i} = \C(u_{i-1}, \xi^{u_{i-1}}_{i})$ is the cluster of $u_{i-1}$ in $\xi^{u_{i-1}}_i$. Further define the remainder functions
\begin{align}
    r_{\lambda, 0,L} (x,y) &:= - \lambda \int_{\HypDim} \plal \left( \left\{\dconn{y}{u_0}{\xi^{y, u_0}_0}\right\} \cap \left\{\xconn{u_0}{x}{\xi^{u_0,x}_1}{\C_0}\right\}  \right) \mu\left(\dd u_0\right), \label{eq:LE:R0_def}\\
		r_{\lambda, n,L}(x,y) &:= (-\lambda)^{n+1} \int_{\HypDim}\ldots\int_{\HypDim} \plal \left( \left\{\dconn{y}{u_0}{\xi^{y, u_0}_0}\right\} \cap \bigcap_{i=1}^{n} E\left(u_{i-1},u_i; \C_{i-1}, \xi^{u_{i-1}, u_i}_{i}\right)\right. \nonumber \\
				& \hspace{6cm}\left.\cap \left\{ \xconn{u_n}{x}{\xi^{u_n,x}_{n+1}}{\C_n} \right\} \right) \mu\left(\dd u_0\right) \ldots \mu\left(\dd u_n\right)
							 \label{eq:LE:Rn_def}.
\end{align}
Additionally, define $\pi_{\lambda, n, L}$ as the alternating partial sum
\begin{equation}
    \pi_{\lambda, n, L}(x,y) := \sum_{m=0}^{n} (-1)^m \pi_{\lambda,L}^{(m)}(x,y). \label{eq:LE:PiN_def}
\end{equation}

Given $\lambda\geq 0$, $n\in\N_0$, and $L>0$, define
\begin{multline}
    \mathcal{D}\left(\OpLace^{(n)}_{\lambda,L}\right) := \left\{f\colon \HypDim\to \mathbb{C} \;\vert\;  \int_\HypDim\abs*{\pi_{\lambda,L}^{(n)}(x,y)}\abs*{f(y)}\mu\left(\dd y\right) <\infty \right.\\
    \left.\text{ for } \mu\text{-almost every }x\in\HypDim \right\}.
\end{multline}
Similarly define $\mathcal{D}\left(R_{\lambda,n,L}\right)$ and $\mathcal{D}\left(\OpLace_{\lambda,n,L}\right)$ using $r_{\lambda,n,L}$ and $\pi_{\lambda, n,L}$ respectively. Then we can define $\OpLace^{(n)}_{\lambda,L}\colon \mathcal{D}\left(\OpLace^{(n)}_{\lambda,L}\right) \to \Complex^\HypDim$ as the convolution operator that acts on $f\in \mathcal{D}\left(\OpLace^{(n)}_{\lambda,L}\right)$ by
\begin{equation}
    \left(\OpLace^{(n)}_{\lambda,L}f\right)(x):= \int_\HypDim \pi^{(n)}_{\lambda,L}(x,y)f(y)\mu\left(\dd y\right)
\end{equation}
for almost every $x\in\HypDim$. Similarly we can define the operators $R_{\lambda,n,L}\colon \mathcal{D}\left(R_{\lambda,n,L}\right) \to \Complex^\HypDim$ and $\OpLace_{\lambda,n,L}\colon \mathcal{D}\left(\OpLace_{\lambda,n,L}\right) \to \Complex^\HypDim$ as the convolution operators (and their domains) using $r_{\lambda,n,L}$ and $\pi_{\lambda, n,L}$ respectively. Observe that for all $n\in\N_0$
\begin{equation}
    \mathcal{D}\left(\OpLace_{\lambda,n,L}\right) =\bigcap_{0\leq m \leq n}\mathcal{D}\left(\OpLace^{(m)}_{\lambda,L}\right).
\end{equation}
Therefore we have
\begin{equation}
\label{eqn:LaceExpTruncated}
    \OpLace_{\lambda, n,L} = \sum_{m=0}^{n} (-1)^m \OpLace^{(m)}_{\lambda,L}.
\end{equation}
\end{definition}

\begin{prop}\label{prop:LaceExpansion}
    Suppose $\left\{\connf_L\right\}_{L>0}$ satisfies Assumption~\ref{Model_Assumption}. Then there exists $L_0$ such that for $L\geq L_0$ the following all hold:
    \begin{enumerate}
        \item for all $\lambda\leq \lambda_\mathrm{c}(L)$
        \begin{equation}
            \OpLace_{\lambda, L}:= \lim_{n\to\infty} \OpLace_{\lambda, n,L}
        \end{equation}
        exists in the $\norm*{\cdot}_{1\to 1}$ topology,
        \item for all $\lambda\leq \lambda_\mathrm{c}(L)$
    \begin{equation}
        \lambda\norm*{\OpLace_{\lambda,L}}_{1\to 1} =\LandauBigO{\beta(L)},
    \end{equation}
    \item for all $\lambda < \lambda_\mathrm{c}(L)$ the two bounded operators $\Opconnf_L,\Optau_{\lambda,L}\colon L^1\left(\HypDim\right)\to  L^1\left(\HypDim\right)$ satisfy the Ornstein-Zernike equation
    \begin{equation}\label{eqn:OZE}
        \Optau_{\lambda,L} = \Opconnf_L + \OpLace_{\lambda,L} + \lambda\Optau_{\lambda,L}\left(\Opconnf_L + \OpLace_{\lambda,L}\right).
    \end{equation}
    \end{enumerate}
\end{prop}

Before we prove this proposition, we need two lemmas to prepare the ground. Lemma~\ref{lem:LaceExpansion} establishes a convolution equation that relates the two-point function, the adjacency function, the lace expansion coefficient functions, and the remainder functions.

\begin{lemma}
\label{lem:LaceExpansion}
    Let $\lambda<\lambda_{\mathrm{c}}(L)$. Then for $n\in\N$ and $x,y\in\HypDim$
    \begin{multline}
         \tau_{\lambda,L}\left(x,y\right) = \connf_L\left(x,y\right) + \pi_{\lambda, n,L}\left(x,y\right) \\+ \lambda \int_{\HypDim}\tau_{\lambda,L}\left(x,u\right)\left(\connf_L + \pi_{\lambda, n,L}\right)\left(u,y\right)\dd u + r_{\lambda, n,L}\left(x,y\right),
    \end{multline}
    and
    \begin{equation}
        \Optau_{\lambda,L} = \Opconnf_L + \OpLace_{\lambda, n,L} +  \lambda\Optau_{\lambda,L}\left(\Opconnf_L + \OpLace_{\lambda, n,L}\right) + R_{\lambda, n,L}.
    \end{equation}
\end{lemma}

\begin{proof}
    This is very nearly exactly \cite[Proposition~3.8]{HeyHofLasMat19} with \cite[Proposition~4.6]{DicHey2022triangle} for the extension to operators. These were working on $\Rd$, but replacing this with $\HypDim$ causes no issue for the parts we require. The condition $\lambda<\lambda_\mathrm{c}$ here ensures that $\pi^{(n)}_{\lambda,L}(x,y)$ are all finite, so $\pi_{\lambda,n,L}(x,y)$ and $\OpLace_{\lambda,n,L}$ are all well-defined.
\end{proof}

This convolution equation would look more manageable if the remainder function wasn't there. One may hope that the remainder function diminishes as $n\to\infty$, and that we can recover an equation that relates $\Optau_{\lambda,L}$, $\Opconnf_L$, and some large $n$ limit of $\OpLace_{\lambda,n,L}$. To confirm this is indeed the case, in Lemma~\ref{lem:LaceBounds} we bound some of these terms with objects defined in Definition~\ref{defn:TrianglesEtc}, and the proof of Proposition~\ref{prop:LaceExpansion} will largely consist of showing that these objects are sufficiently small that the required $n\to\infty$ limits hold.

\begin{definition}\label{defn:TrianglesEtc}
    For $\lambda\geq0$, recall
    \begin{equation}
        \triangle_{\lambda,L} = \lambda^2\esssup_{x\in\HypDim}\tau^{\star 3}_{\lambda,L}\left(x,\orig\right),
    \end{equation}
    and -- using the notation of \cite{DicHey2022triangle} -- define
    \begin{align}
        \triangle^{\circ\circ}_{\lambda,L} &:= \triangle_{\lambda,L} + \lambda\esssup_{x\in\HypDim}\tau^{\star 2}_{\lambda,L}\left(x,\orig\right) + 1,\\
        U_{\lambda,L}&:= \triangle_{\lambda,L}\triangle^{\circ\circ}_{\lambda,L} + \left(\triangle_{\lambda,L}^{\circ\circ}\right)^2 + \lambda\norm*{\Opconnf_L}_{1\to 1},\\
        V_{\lambda,L}&:= \left(\triangle_{\lambda,L}\triangle_{\lambda,L}^{\circ\circ}U_{\lambda,L}\right)^{\frac{1}{2}}.
    \end{align}
\end{definition}

\begin{lemma}
\label{lem:LaceBounds}
For all $\lambda\geq0$ and $n\in\N$,
    \begin{align}
    \label{eqn:RemainderBound}
        \lambda\norm*{R_{\lambda,n,L}}_{1\to 1} &\leq \lambda^2 \norm*{\Optau_{\lambda,L}}_{1\to 1}\norm*{\OpLace^{(n)}_{\lambda,L}}_{1\to 1},\\
        \label{eqn:ZerothLaceBound}
    \lambda\norm*{\OpLace^{(0)}_{\lambda,L}}_{1\to 1} &\leq \frac{1}{2}\lambda \norm*{\Opconnf_L}_{1\to 1}\triangle_{\lambda,L},\\
    \label{eqn:NthLaceBound}
    \lambda\norm*{\OpLace^{(n)}_{\lambda,L}}_{1\to 1} &\leq 6\times 4^{n-1}U_{\lambda,L}V_{\lambda,L}^n.
    \end{align}
\end{lemma}

\begin{proof}
    These inequalities are proven in \cite{DicHey2022triangle} for the $\Rd$ versions, and the $\HypDim$ version is exactly the same. The inequality \eqref{eqn:RemainderBound} follows from Lemma~5.1, \eqref{eqn:ZerothLaceBound} from Lemma~5.2, and \eqref{eqn:NthLaceBound} from Lemma~5.8 therein.
\end{proof}

\begin{remark}
    In Lemma~\ref{lem:LaceExpansion} we required $\lambda<\lambda_\mathrm{c}$ to ensure $\pi^{(n)}_{\lambda,L}\left(x,y\right)$ were finite and the alternating sums defining $\pi_{\lambda,n,L}\left(x,y\right)$ and $\OpLace_{\lambda,n,L}$ are well-defined. However, the bounds \eqref{eqn:ZerothLaceBound} and \eqref{eqn:NthLaceBound} provide us the means to define $\OpLace_{\lambda_c,n,L}$ as the alternating sum of bounded $L^1\to L^1$ operators in \eqref{eqn:LaceExpTruncated} (if $\triangle_{\lambda_\mathrm{c}(L),L}<\infty$).
\end{remark}

\begin{proof}[Proof of Proposition~\ref{prop:LaceExpansion}]
    First observe that for all $\lambda\geq 0$, $L>0$, and $x\in\HypDim$
    \begin{equation}\label{eqn:FirstStep}
        \tau_{\lambda,L}\left(x,\orig\right) \leq \connf_{L}\left(x,\orig\right) + \lambda\connf_L\star \tau_{\lambda,L}\left(x,\orig\right).
    \end{equation}
    This is proven for RCMs on $\Rd$ in \cite[Observation~4.3]{HeyHofLasMat19} by using Mecke's formula and a first-step analysis. The argument for RCMs on $\HypDim$ is exactly the same. Using this inequality twice, 
    \begin{align}
        \lambda\esssup_{x\in\HypDim}\tau_{\lambda,L}^{\star 2}\left(x,\orig\right) &\leq \lambda\esssup_{x\in\HypDim}\connf_L^{\star 2}\left(x,\orig\right) + 2\lambda^2\esssup_{x\in\HypDim}\connf_L^{\star 2}\star \tau_{\lambda,L}\left(x,\orig\right) \nonumber\\
        &\hspace{7cm}+ \lambda^3\esssup_{x\in\HypDim}\connf_L^{\star 2}\star \tau_{\lambda,L}^{\star 2}\left(x,\orig\right)\nonumber\\
        &\leq \lambda\norm*{\Opconnf_L}_{1\to 1} + 2\triangle_{\lambda,L} + \lambda\norm*{\Opconnf_L}_{1\to 1} \triangle_{\lambda,L},\label{eqn:bubblebound}
    \end{align}
    where in the second inequality we have used supremum bounds and $1\geq \tau_{\lambda,L}\geq \connf_L$. From Theorem~\ref{thm:TriangleSmall} we know $\triangle_{\lambda_\mathrm{c}(L),L}=\LandauBigO{\beta(L)^2}$ and $\lambda_{\mathrm{c}}(L)\norm*{\Opconnf}_{1\to 1}=\LandauBigO{1}$ as $L\to\infty$. Therefore $\triangle^{\circ\circ}_{\lambda_\mathrm{c}(L),L} = \LandauBigO{1}$, $U_{\lambda_\mathrm{c}(L),L}= \LandauBigO{1}$, and $V_{\lambda_\mathrm{c}(L),L}= \LandauBigO{\beta(L)}$ as $L\to\infty$.

    These bounds with Lemma~\ref{lem:LaceBounds} and the monotonicity of $\lambda\mapsto\triangle_{\lambda,L}$ imply that there exists a constant $C<\infty$ such that for all $n\geq 1$ and $\lambda\leq \lambda_\mathrm{c}(L)$
    \begin{align}
        \lambda\norm*{\OpLace^{(0)}_{\lambda,L}}_{1\to 1} &\leq C \triangle_{\lambda_{\mathrm{c}}(L),L},\\
        \lambda\norm*{\OpLace^{(n)}_{\lambda,L}}_{1\to 1} &\leq C \left(16\triangle_{\lambda_{\mathrm{c}}(L),L}\right)^\frac{n}{2},\label{eqn:triangleboundsLace}\\
        \lambda\norm*{R_{\lambda,n,L}}_{1\to 1} &\leq C \left(16\triangle_{\lambda_{\mathrm{c}}(L),L}\right)^\frac{n}{2}\lambda\norm*{\Optau_{\lambda,L}}_{1\to 1}.
    \end{align}

    Therefore for $L$ large enough to get $16\triangle_{\lambda_{\mathrm{c}}(L),L}<1$, the sequence $\left\{\lambda\OpLace_{\lambda,n,L}\right\}_{n\in\N}$ is a Cauchy sequence for $\lambda \in\left[0,\lambda_\mathrm{c}(L)\right]$. Hence 
    \begin{equation}
        \OpLace_{\lambda,L} = \lim_{n\to\infty}\OpLace_{\lambda,n,L}
    \end{equation}
    exists and is a bounded linear $L^1\left(\HypDim\right)\to L^1\left(\HypDim\right)$ operator for $\lambda \in\left[0,\lambda_\mathrm{c}(L)\right]$ and $L$ sufficiently large. In particular, there exists a $C'<\infty$ such that $\lambda\norm*{\OpLace_{\lambda,L}}_{1\to 1}\leq C'\beta(L)$ for $\lambda \in\left(0,\lambda_\mathrm{c}(L)\right]$ and $L$ sufficiently large.
    
    For $L$ sufficiently large and $\lambda<\lambda_\mathrm{c}(L)$, we have $\norm*{\Optau_{\lambda,L}}_{1\to 1}<\infty$ (see Lemma~\ref{lem:LambdaCequalsLambdaT}). Therefore for $L$ sufficiently large and $\lambda \in\left[0,\lambda_\mathrm{c}(L)\right)$ we have
    \begin{equation}
        \lim_{n\to\infty}\norm*{R_{\lambda,n,L}}_{1\to 1} = 0.
    \end{equation}
    Therefore Lemma~\ref{lem:LaceExpansion} implies that for $\lambda \in\left[0,\lambda_\mathrm{c}(L)\right)$ and $L$ sufficiently large we have
    \begin{equation}
        \Optau_{\lambda,L} = \left(\Opconnf_L + \OpLace_{\lambda,L}\right) +\lambda\Optau_{\lambda,L}\left(\Opconnf_L + \OpLace_{\lambda,L}\right)
    \end{equation}
    as required.
\end{proof}

We will require one more property of the limiting $\OpLace_{\lambda,L}$ operator.

\begin{lemma}
\label{lem:LaceContinuous}
    Suppose $\left\{\connf_L\right\}_{L>0}$ satisfies Assumption~\ref{Model_Assumption}. For $L$ sufficiently large, the function $\lambda\mapsto \OpLace_{\lambda,L}$ is continuous with respect to $\norm*{\cdot}_{1\to 1}$ on $\left[0,\lambda_\mathrm{c}(L)\right]$.
\end{lemma}

\begin{proof}[Proof Summary]
    This proceeds as in \cite[Corollary~6.1]{HeyHofLasMat19}. The core idea is a dominated convergence argument. We aim to show $\lim_{\lambda\to \lambda^*}\norm*{\OpLace_{\lambda,L}^{(n)}- \OpLace_{\lambda^*,L}^{(n)}}_{1\to 1}=0$ for all $n\in\N_0$, $\lambda^*\leq \lambda_\mathrm{c}(L)$ and $L$ sufficiently large. As in the proof of Lemma~\ref{lem:LaceBounds} (see \cite{DicHey2022triangle}), we can find a function $h^{(n)}_{L}$ such that $\pi^{(n)}_{\lambda,L}\leq h^{(n)}_L$ pointwise for all $n\in\N_0$ and $\lambda^*\leq \lambda_\mathrm{c}(L)$, and $h^{(n)}_{L}$ is integrable for all $n\in\N_0$ and $L$ sufficiently large. If we can show pointwise convergence for $\pi_{\lambda,L}^{(n)}$, then this integrability shows that the required limit holds. The pointwise convergence holds exactly as explained in \cite[Corollary~6.1]{HeyHofLasMat19}, and uses the fact that connected components are almost surely finite. This fact is proven for $\lambda\leq \lambda_\mathrm{c}(L)$ in our context in \cite{dickson2024NonUniqueness}.
\end{proof}

\section{Characterising the Percolation Threshold}
\label{sec:characterisingPercolationThreshold}

In this section we show how the critical threshold $\lambda_\mathrm{c}$ can be extracted from \eqref{eqn:OZE}, and given implicitly in terms of the operators $\Opconnf_L$ and $\OpLace_{\lambda,L}$.

We require a relatively standard percolation result that (via Mecke's formula) shows that the expected cluster size diverges as the critical threshold is approached from below.

\begin{lemma}
\label{lem:continuityofOnetoOne}
    For all $L>0$,
    \begin{equation}
        \lim_{\lambda\nearrow\lambda_\mathrm{c}(L)}\norm*{\Optau_{\lambda,L}}_{1\to 1} = \infty.
    \end{equation}
\end{lemma}

\begin{proof}[Proof Summary]
    This is proven in exactly the same way as is was in the proof of \cite[Theorem~2.5]{caicedo2023criticalexponentsmarkedrandom}, using standard ideas that go back at least as far as \cite{AizNew84}. The argument considers $\Optau_{\lambda,L}^{(n)}$, a spatially truncated version of $\Optau_{\lambda,L}$. By using the Margulis-Russo formula to control the derivative, it can be proven that the sequence (when varying $n$) of functions
    \begin{equation}
        \lambda\mapsto \norm*{\Optau_{\lambda,L}^{(n)}}_{1\to 1}^{-1}
    \end{equation}
    is an equicontinuous sequence of functions. Since $\norm*{\Optau_{\lambda,L}^{(n)}}_{1\to 1}^{-1}\to \norm*{\Optau_{\lambda,L}}_{1\to 1}^{-1}$ monotonically as $n\to \infty$, this proves that
    \begin{equation}
        \lambda\mapsto \norm*{\Optau_{\lambda,L}}_{1\to 1}^{-1}
    \end{equation}
    is continuous for all $\lambda\geq 0$.
\end{proof}

We are now prepared to characterise $\lambda_\mathrm{c}(L)$ in terms of $\Opconnf_L$ and $\OpLace_{\lambda,L}$.

\begin{lemma}
\label{lem:CriticalEquation}
    Let $d\geq 2$ and suppose $\left\{\connf_L\right\}_{L>0}$ satisfies Assumption~\ref{Model_Assumption}. Then there exists $L_0$ such that for $L\geq L_0$,
    \begin{equation}
        \lambda\rho_{1\to 1}\left(\Opconnf_L +\OpLace_{\lambda,L}\right) \begin{cases}
            <1 & \text{if }\lambda<\lambda_{\mathrm{c}}(L),\\
            =1 & \text{if }\lambda=\lambda_\mathrm{c}(L).
        \end{cases}
    \end{equation}
\end{lemma}

\begin{proof}
    First define
    \begin{equation}
        \tilde{\lambda}:= \inf\left\{\lambda\in\left(0,\lambda_\mathrm{c}\right]\colon \lambda\rho_{1\to 1}\left(\Opconnf_L + \OpLace_{\lambda,L}\right)\geq 1\right\}.
    \end{equation}
    First note that by Lemma~\ref{lem:LaceContinuous}, $\lambda\mapsto \lambda\Opconnf_L + \lambda\OpLace_{\lambda,L}$ is continuous with respect to $\norm*{\cdot}_{1\to 1}$ for all $\lambda\leq \lambda_\mathrm{c}$. Since $\connf_L$ and $\pi_{\lambda,n,L}$ are all isometry invariant, lemmas~\ref{lem:SpectrumPerturbation} and \ref{lem:IsoInvtoCommute} give that the map
    \begin{equation}
        \lambda\mapsto \lambda\rho_{1\to 1}\left(\Opconnf_L + \OpLace_{\lambda,L}\right)
    \end{equation}
    is continuous for $L>L_0$ and $\lambda\leq \lambda_\mathrm{c}$. Also note that this map takes $0\mapsto0$ (since $\Opconnf_L$ and $\OpLace_{0,L}$ are both bounded operators), and therefore $\tilde{\lambda}>0$.

    Suppose $\lambda<\tilde{\lambda}$. Then $\Id - \lambda\left(\Opconnf_L + \OpLace_{\lambda,L}\right)\colon L^1\left(\HypDim\right)\to L^1\left(\HypDim\right)$ is invertible and
    \begin{equation}
        \left(\Id - \lambda\left(\Opconnf_L + \OpLace_{\lambda,L}\right)\right)^{-1} = \Id + \sum^\infty_{n=1}\lambda^n\left(\Opconnf_L + \OpLace_{\lambda,L}\right)^n.
    \end{equation}
    Therefore Proposition~\ref{prop:LaceExpansion} implies that if $\lambda<\min\left\{\tilde{\lambda},\lambda_\mathrm{c}\right\}$ and $L\geq L_0$ then
    \begin{equation}
    \label{eqn:TauInfiniteSum}
        \Optau_{\lambda,L} = \left(\Opconnf_L+\OpLace_{\lambda,L}\right) \left(\Id - \lambda\left(\Opconnf_L + \OpLace_{\lambda,L}\right)\right)^{-1} = \sum^\infty_{n=1}\lambda^{n-1}\left(\Opconnf_L + \OpLace_{\lambda,L}\right)^n.
    \end{equation}

    By Lemma~\ref{lem:spectralradiusofPositive} and $\tau_{\lambda,L}\geq 0$, $\norm*{\Optau_{\lambda,L}}_{1\to 1}=\rho_{1\to 1}\left(\Optau_{\lambda,L}\right)$. The spectral mapping theorem then tells us that for $\lambda<\min\left\{\tilde{\lambda},\lambda_\mathrm{c}\right\}$ and $L\geq L_0$,
    \begin{equation}\label{eqn:spectralmapping}
        \lambda\norm*{\Optau_{\lambda,L}}_{1\to 1}= \lambda\rho_{1\to 1}\left(\Optau_{\lambda,L}\right) = \frac{\lambda\rho_{1\to 1}\left(\Opconnf_L + \OpLace_{\lambda,L}\right)}{1-\lambda\rho_{1\to 1}\left(\Opconnf_L + \OpLace_{\lambda,L}\right)}.
    \end{equation}

    If $\tilde{\lambda} < \lambda_\mathrm{c}$, then the continuity property on $\left[0,\lambda_\mathrm{c}\right]$ implies $\lambda\rho_{1\to 1}\left(\Opconnf_L+\OpLace_{\lambda,L}\right)\to 1$ as $\lambda\nearrow\tilde{\lambda}$, and with \eqref{eqn:spectralmapping} this implies $\norm*{\Optau_{\lambda,L}}_{1\to 1}\to\infty$ as $\lambda\nearrow\tilde{\lambda}$. Therefore Lemma~\ref{lem:continuityofOnetoOne} produces a contradiction and we have $\tilde{\lambda}\geq  \lambda_\mathrm{c}$.

    Since $\tilde{\lambda}\geq \lambda_\mathrm{c}$, Lemma~\ref{lem:continuityofOnetoOne} and \eqref{eqn:spectralmapping} imply $\lambda\rho_{1\to 1}\left(\Opconnf_L+\OpLace_{\lambda,L}\right)\nearrow 1$ as $\lambda\nearrow \lambda_\mathrm{c}$. Therefore $\tilde{\lambda}= \lambda_\mathrm{c}$.

    The equality $\lambda_\mathrm{c}=\tilde{\lambda}$ and \eqref{eqn:spectralmapping} imply $\lim_{\lambda\nearrow\lambda_\mathrm{c}}\lambda\rho_{1\to 1}\left(\Opconnf_L+\OpLace_{\lambda,L}\right)=1$, and the result then follows from the continuity of $\lambda\mapsto \lambda\rho_{1\to 1}\left(\Opconnf_L+\OpLace_{\lambda,L}\right)$ for $\lambda\leq \lambda_\mathrm{c}$.
\end{proof}

\section{Expanding the Lace Expansion Coefficients}

\label{sec:ExpandLaceCoeff}

Now we have characterised $\lambda_\mathrm{c}(L)$ in terms of $\Opconnf_L$ and $\OpLace_{\lambda,L}$, we need a finer understanding of $\OpLace_{\lambda,L}$. The objective here is to replicate many of the arguments in \cite[Section~3]{dickson2025expansion} to approximate $\pi^{(n)}_{\lambda,L}$ functions with combinations of $\connf_L$ functions. That paper was concerned with RCMs on $\Rd$, but most of the arguments did not rely on the underlying space. The exceptions to this came about in that that former paper did often use Fourier transform arguments to bound certain integrals and to determine which integrals will be the most relevant. We therefore first need to replicate these lemmas using the spherical transform on $\HypDim$.

\begin{definition}
    Given $x\in \HypDim$ and $b\in \mathbb{S}^{d-1}$ (i.e. the Euclidean unit radius sphere), let $\mathfrak{h}\left(x,b\right)$ denote the \emph{horocycle} through $x$ with normal $b$, and denote by $A\left(x,b\right)$ the (signed) \emph{composite distance} from $\orig$ to $\mathfrak{h}\left(x,b\right)$. That is,
    \begin{equation}
        A\left(x,b\right) = \begin{cases}
            \dist{\orig,\mathfrak{h}\left(x,b\right)}&\colon \orig \text{ is `outside' }\mathfrak{h}\left(x,b\right),\\
            -\dist{\orig,\mathfrak{h}\left(x,b\right)}&\colon \orig \text{ is `inside' }\mathfrak{h}\left(x,b\right).
        \end{cases}
    \end{equation}
    See Figure~\ref{fig:Horocycles} for some level-sets of $A(x,b)$ in $\HypTwo$.
    For a point $x\in\HypDim$ with Poincar{\'e} ball model coordinate $z\in\mathbb{B}\subset \Rd$ and $b\in \mathbb{S}^{d-1}\subset\Rd$,
    \begin{equation}
        A\left(x,b\right) = \log \frac{1-\abs*{z}^2}{\abs*{z-b}^2}.
    \end{equation}
    This composite distance is the $\HypDim$ analog of the usual Euclidean scalar product $\left(y,\omega\right)$ for $y\in\Rd$ and $\omega\in \mathbb{S}^{d-1}$ (see \cite{helgason1994geometric}).

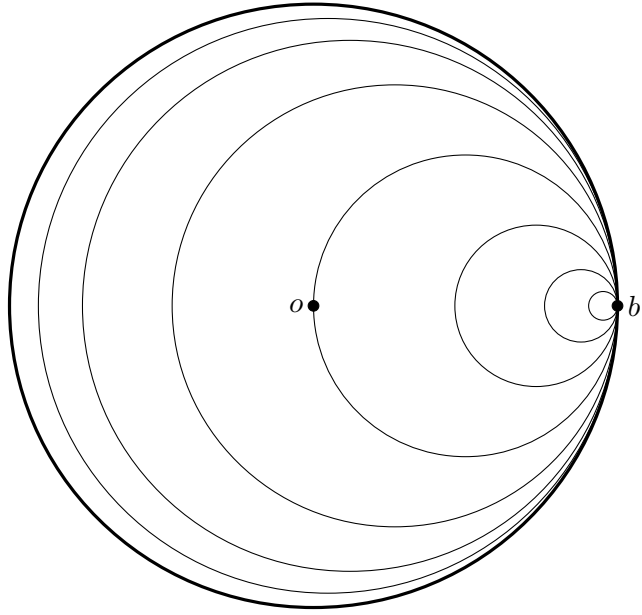
\begin{figure}
    \centering
    \begin{tikzpicture}[scale=1]
        \draw[very thick] (0,0) circle (4);
        \draw (2,0) circle (2);
        \draw (3.81,0) circle (0.19);
        \draw (3.52,0) circle (0.48);
        \draw (2.93,0) circle (1.07);
        \draw (1.07,0) circle (2.93);
        \draw (0.48,0) circle (3.52);
        \draw (0.19,0) circle (3.81);
        \filldraw (0,0) circle (2pt) node[left]{$\orig$};
        \filldraw (4,0) circle (2pt) node[right]{$b$};
    \end{tikzpicture}
    \caption{Contours of $A\left(x,b\right)=-3-2,-1,0,1,2,3$ in the Poincar{\'e} disc model of $\HypTwo$.}
    \label{fig:Horocycles}
\end{figure}

    If $f$ is a function on $\HypDim$, then we define
    \begin{equation}
    \label{eqn:SphericalTransform}
        \widetilde{f}\left(s,b\right) := \int_\HypDim f(x) \e^{\left(-is + \frac{d-1}{2}\right)A\left(x,b\right)}\mu\left(\dd x\right)
    \end{equation}
    for all $s\in\R$, $b\in \mathbb{S}^{d-1}$ for which the integral exists. If $f$ is invariant under rotations about $\orig\in\HypDim$, then $\widetilde{f}$ is $b$-independent and we write $\widetilde{f}(s,b)=\widetilde{f}(s)$ only. Let $L^p_\natural\left(\HypDim\right)$ denote the set of $L^p\left(\HypDim\right)$ functions that are rotation invariant  about $\orig\in\HypDim$.
    
    We call this $f\mapsto \widetilde{f}$ transformation the \emph{spherical transformation} - in some texts it may be called the Fourier transform (on $\HypDim$), but we avoid using this name to avoid confusion with the usual Fourier transform (on $\Rd$).
\end{definition}

\begin{lemma}[Inversion, multiplication, and Plancherel]
\label{lem:InvMultPlancherel}
    There exists a $d$-dependent function $\mathbf{c}\colon \R\to\Complex$ such that the transformation \eqref{eqn:SphericalTransform} extends to an isometry of $L^2\left(\HypDim\right)$ onto $L^2\left(\R\times \mathbb{S}^{d-1}\right)$ with measure $\abs{\mathbf{c}(s)}^{-2}\dd s\dd b$ on $\R\times \mathbb{S}^{d-1}$.

    There exists a $d$-dependent real constant $w>0$ such that for $f\in L^2\left(\HypDim\right)$,
    \begin{equation}
        f(x) = \frac{1}{w} \int_{\mathbb{S}^{d-1}}\int_\R\e^{\left(is + \frac{d-1}{2}\right)A(x,b)}\widetilde{f}(s,b)\abs{\mathbf{c}(s)}^{-2}\dd s\dd b.
    \end{equation}

    For $f_1\in L^2_\natural\left(\HypDim\right)$ and $f_2\in L^2\left(\HypDim\right)$,
    \begin{equation}\label{eqn:diagonalise}
        \left(f_1\star f_2\right)^{\sim}(s,b) = \widetilde{f}_1(s)\widetilde{f}_2(s,b).
    \end{equation}

    For $f_1,f_2\in L^2\left(\HypDim\right)$
    \begin{equation}
        \int_{\HypDim}f_1(x)\overline{f_2(x)} \mu\left(\dd x\right) = \frac{1}{w}\int_{\mathbb{S}^{d-1}}\int_\R\widetilde{f_1}(s,b)\overline{\widetilde{f_2}(s,b)}\abs{\mathbf{c}(s)}^{-2}\dd s \dd b.
    \end{equation}
\end{lemma}

\begin{proof}
    See \cite[Chapter~III, Lemmas 1.3, 1.4, 1.5]{helgason1994geometric}.
\end{proof}

The constant $w$ is the order of the Weyl group and the function $\mathbf{c}(s)$ is known as the Harish-Chandra $\mathbf{c}$-function. We will not require expressions for these in our argument, but their definitions can be found in \cite[Chapter~1, p.75 \& p.108]{helgason1994geometric}.

The following standard lemma is proven in Appendix~\ref{app:rangesphericaltransform}.
\begin{lemma}
\label{lem:RangeSphericalTransform}
    If $G\colon L^2\left(\HypDim\right)\to L^2\left(\HypDim\right)$ is a bounded convolution operator associated with an isometry invariant and real $g$, then
    \begin{equation}
        \widetilde{g}(s)\in\R
    \end{equation}
    for Lebesgue-almost all $s\in\R$, and 
    \begin{equation}
        \norm*{G}_{2\to 2} = \esssup_{s\in\R}\abs*{\widetilde{g}(s)}.
    \end{equation}
\end{lemma}

Let us denote 
\begin{equation}
    Q_d(x) := \frac{1}{\mathfrak{S}_{d-1}}\int_{\mathbb{S}^{d-1}}\e^{
        \frac{d-1}{2}A\left(x,b\right)}\dd b,
\end{equation}
for $x\in\HypDim$. The rotational invariance of $\dd b$ implies $Q_d$ is also rotationally invariant about $\orig$.

\begin{lemma}\label{lem:QfunctionBound}
    There exists a constant $C=C_d<\infty$ such that 
    \begin{equation}
        Q_d\left(x\right) \leq C \max\left\{1,\dist{x,\orig}\right\}\exp\left(-\frac{1}{2}\left(d-1\right)\dist{x,\orig}\right),
    \end{equation}
    for all $x\in\HypDim$.
\end{lemma}
\begin{proof}
    See \cite[Lemma~5.3]{dickson2024NonUniqueness}.
\end{proof}

The following lemma allows us to bound integrals involving factors of the implicitly known $\tau_{\lambda,L}$ with integrals solely of the explicitly known $\connf_L$.

\begin{lemma}
\label{lem:tautoconnf}
Suppose $\left\{\connf_L\right\}_{L>0}$ satisfies Assumption~\ref{Model_Assumption}. Then there exists $L_0$ and $K<\infty$ such that for $L\geq L_0$, $\lambda\leq \lambda_\mathrm{c}(L)$, even $m\geq 2$, and all $n\geq 1$,
    \begin{equation}
        \lambda^{n}\sup_{x\in\HypDim}\connf_L^{\star m}\star\tau_{\lambda,L}^{\star n}\left(x,\orig\right) \leq K^n\connf_L^{\star m}\left(\orig,\orig\right).
    \end{equation}
\end{lemma}

\begin{proof}
    First let $0<\lambda<\lambda_\mathrm{c}(L)$. Then by lemmas~\ref{lem:spectralradiusofPositive} and \ref{lem:LambdaCequalsLambdaT},
    \begin{equation}
        \int_\HypDim\tau_{\lambda,L}(x,\orig)\mu\left(\dd x\right) = \norm*{\Optau_{\lambda,L}}_{1\to 1}<\infty.
    \end{equation}
    Then by the rotational invariance of the model and having $\tau_{\lambda,L}\left(x,\orig\right)\in\left[0,1\right]$, we have $\tau_{\lambda,L}\in L^2_\natural\left(\HypDim\right)$. Therefore we can define $\widetilde{\tau}_{\lambda,L}(s)$. Assumption~\ref{enum:AssumptionFiniteDegree} also implies $\connf_L\in L^2_\natural\left(\HypDim\right)$ and so Lemma~\ref{lem:InvMultPlancherel} implies
    \begin{equation}
        \connf_L^{\star m}\star\tau_{\lambda,L}^{\star n}\left(x,\orig\right) = \frac{1}{w}\int_{\mathbb{S}^{d-1}}\int_\R\e^{\left(is + \frac{d-1}{2}\right)A\left(x,d\right)}\widetilde{\connf}_L(s)^m\widetilde{\tau}_{\lambda,L}(s)^n\abs{\mathbf{c}(s)}^{-2}\dd s\dd b.
    \end{equation}
    By Lemma~\ref{lem:RangeSphericalTransform}, $\widetilde{\connf}_L\left(s\right)\in\R$ and $\abs*{\widetilde{\connf}_L(s)}^m=\widetilde{\connf}_L(s)^m$ (since $m$ is even) for almost all $s\in\R$. Therefore
    \begin{equation}
        \esssup_{x\in\HypDim}\connf_L^{\star m}\star\tau_{\lambda,L}^{\star n}\left(x,\orig\right) \leq \left(\esssup_{x\in\HypDim} Q_d\left(x\right)\right)\frac{\mathfrak{S}_{d-1}}{w}\int_\R\widetilde{\connf}_L(s)^m\abs*{\widetilde{\tau}_{\lambda,L}(s)}^n\abs{\mathbf{c}(s)}^{-2}\dd s.
    \end{equation}

    By applying the spherical transform to \eqref{eqn:OZE} in Proposition~\ref{prop:LaceExpansion},
    \begin{equation}
        \lambda\widetilde{\tau}_{\lambda,L}(s) = \frac{\lambda\left(\widetilde{\connf}_L(s) + \widetilde{\pi}_{\lambda,L}(s)\right)}{1 - \lambda\left(\widetilde{\connf}_L(s) + \widetilde{\pi}_{\lambda,L}(s)\right)}.
    \end{equation}
    From Lemma~\ref{lem:ConvolutionOperator}, Proposition~\ref{prop:LaceExpansion}, and Lemma~\ref{lem:RangeSphericalTransform} we have 
    \begin{equation}
        \lambda\esssup_{s\in\R}\abs*{\widetilde{\pi}_{\lambda,L}(s)} = \lambda\norm*{\OpLace_{\lambda,L}}_{2\to 2}\leq \lambda\norm*{\OpLace_{\lambda,L}}_{1\to 1} = \LandauBigO{\beta(L)}
    \end{equation}
    uniformly in $\lambda\leq \lambda_\mathrm{c}(L)$. We also have 
    \begin{equation}
        \lambda\esssup_{s\in\R}\abs*{\widetilde{\connf}_L(s)} = \lambda\norm*{\Opconnf_{L}}_{2\to 2} = \LandauBigO{\frac{\norm*{\Opconnf_{L}}_{2\to 2}}{\norm*{\Opconnf_{L}}_{1\to 1}}}
    \end{equation}
    uniformly in $\lambda\leq \lambda_\mathrm{c}(L)$.
    Therefore for sufficiently large $L$ there exists $C<\infty$ such that $\lambda\abs*{\widetilde{\tau}_{\lambda,L}(s)}\leq C$ for all $\lambda\leq \lambda_\mathrm{c}(L)$ and almost all $s\in\R$. Therefore by using Lemma~\ref{lem:QfunctionBound},
    \begin{multline}
        \lambda^{n}\sup_{x\in\HypDim}\connf_L^{\star m}\star\tau_{\lambda,L}^{\star n}\left(x,\orig\right) \leq \left(\sup_{x\in\HypDim} Q_d\left(x\right)\right)C^n\frac{\mathfrak{S}_{d-1}}{w}\int_\R\widetilde{\connf}_L(s)^m \abs{\mathbf{c}(s)}^{-2}\dd s\\ \leq C'C^n \connf_L^{\star m}\left(\orig,\orig\right).
    \end{multline}
    The result then follows by taking $K=C\max\left\{1,C'\right\}$.
\end{proof}

It will soon become become convenient to introduce some notation for integrals of $\connf_L$ functions that are not simple convolutions. For integers $n_1,n_2,n_3\geq 1$ we will write 
    \begin{equation}\label{eqn:}
        \connf_L^{\star n_1\star n_2\cdot n_3}\left(\orig,\orig\right):= \connf_L^{\star n_1}\star\left(\connf_L^{\star n_2}\cdot\connf_L^{\star n_3}\right)\left(\orig,\orig\right) = \int_\HypDim\connf_L^{\star n_1}\left(x,\orig\right)\connf_L^{\star n_2}\left(x,\orig\right)\connf_L^{\star n_3}\left(x,\orig\right)\mu\left(\dd x\right).
    \end{equation}
This last expression makes it clear that $\connf_L^{\star n_1\star n_2\cdot n_3}\left(\orig,\orig\right)$ is invariant under permutations of $n_1$, $n_2$, and $n_3$.

The following lemma allows us to bound integrals involving many factors of $\connf_L$ with simpler integrals of fewer factors of $\connf_L$.

\begin{lemma}
\label{lem:DiagramOrderGeneral}
    Suppose $\left\{\connf_L\right\}_{L>0}$ satisfies Assumption~\ref{enum:AssumptionFiniteDegree}. Then there exists $K<\infty$ such that for all $L>0$, even $m\geq 2$, and $n\geq m$,
    \begin{equation}
        \sup_{x\in\HypDim}\connf_L^{\star n}\left(x,\orig\right) \leq K\norm*{\Opconnf_L}_{2\to 2}^{n-m}\connf^{\star m}_L\left(\orig,\orig\right).
    \end{equation}
    If in addition $n_1,n_2,n_3\geq 1$,
    \begin{equation}
        \connf_L^{\star n_1\star n_2\cdot n_3}\left(\orig,\orig\right)
        \leq \connf^{\star \left(n_1+n_2\right)}\left(\orig,\orig\right)\times
        \begin{cases}
            1& \colon n_3=1\\
            K\norm*{\Opconnf_L}_{2\to 2}^{n_3-2}\connf_L^{\star 2}\left(\orig,\orig\right) &\colon n_3\geq 2
        \end{cases}
    \end{equation}
\end{lemma}

\begin{proof}
    Begin by inverting the spherical transform:
    \begin{align}
        \sup_{x\in\HypDim}\connf_L^{\star n}\left(x,\orig\right) &=  \frac{1}{w}\sup_{x\in\HypDim}\int_{\mathbb{S}^{d-1}}\int_\R\e^{(is+d-1)A(x,b)}\widetilde{\connf}_L(s)^n\abs*{\mathbf{c}(s)}^{-2}\dd s\dd b\nonumber\\
        &\leq \left(\sup_{x\in\HypDim}Q_d(x)\right)\left(\esssup_{t\in\R}\abs*{\widetilde{\connf}_L(t)}^{n-m}\right)\frac{\mathfrak{S}_{d-1}}{w}\int_\R\abs*{\widetilde{\connf}_L(s)}^m\abs*{\mathbf{c}(s)}^{-2}\dd s\nonumber\\
        & = K\norm*{\Opconnf_L}_{2\to 2}^{n-m} \connf_L^{\star m}\left(\orig,\orig\right).
    \end{align}
    The role $m$ being even plays is that we necessarily have $\abs*{\widetilde{\connf}_L(s)^m} = \widetilde{\connf}_L(s)^m$ for all $s\in\R$.

    For the second bound, we first take a supremum bound to get
    \begin{align}
        \int_{\HypDim}\connf^{\star n_1}_L\left(x,\orig\right)\connf^{\star n_2}_L\left(x,\orig\right)\connf^{\star n_3}_L\left(x,\orig\right)\dd x &\leq \left(\int_{\HypDim}\connf^{\star n_1}_L\left(x,\orig\right)\connf^{\star n_2}_L\left(x,\orig\right)\dd x\right)\sup_{x\in\HypDim}\connf^{\star n_3}_L\left(x,\orig\right)\nonumber\\
        &=\connf^{\star \left(n_1+n_2\right)}_L\left(\orig,\orig\right) \sup_{x\in\HypDim}\connf^{\star n_3}_L\left(x,\orig\right).
    \end{align}
    From the first part of the lemma, if $n_3\geq 2$ we have
    \begin{equation}
        \sup_{x\in\HypDim}\connf^{\star n_3}_L\left(x,\orig\right) \leq K\norm*{\Opconnf_L}_{2\to 2}^{n_3-2}\connf_L^{\star 2}\left(\orig,\orig\right).
    \end{equation}
    If $n_3=1$, then we can simply bound
    \begin{equation}
        \connf_L(x,\orig) \leq 1,
    \end{equation}
    and the result follows.
\end{proof}

The following corollary allows us to bound more complicated integrals with simpler ones -- in particular to can be used to determine the relevancy of various integrals as $L\to\infty$.

\begin{corollary}\label{cor:diagramBounds}
    Suppose $\left\{\connf_L\right\}_{L>0}$ satisfies Assumption~\ref{enum:AssumptionFiniteDegree}. Then there exists $K<\infty$ such that for all $L>0$,
    \begin{align}
        \frac{1}{\norm*{\Opconnf_L}_{1\to 1}^2}\connf^{\star 3}_L\left(\orig,\orig\right)& \leq K\frac{\norm*{\Opconnf_L}_{2\to 2}}{\norm*{\Opconnf_L}_{1\to 1}},\\
        \frac{1}{\norm*{\Opconnf_L}_{1\to 1}^3}\connf^{\star 4}_L\left(\orig,\orig\right) & \leq K\left(\frac{\norm*{\Opconnf_L}_{2\to 2}}{\norm*{\Opconnf_L}_{1\to 1}}\right)^2,\\
        \frac{1}{\norm*{\Opconnf_L}_{1\to 1}^4}\connf^{\star 5}_L\left(\orig,\orig\right) & \leq K\frac{1}{\norm*{\Opconnf_L}_{1\to 1}^3}\connf^{\star 4}_L\left(\orig,\orig\right)\frac{\norm*{\Opconnf_L}_{2\to 2}}{\norm*{\Opconnf_L}_{1\to 1}},\\
        \frac{1}{\norm*{\Opconnf_L}_{1\to 1}^5}\connf^{\star 6}_L\left(\orig,\orig\right) & \leq K\frac{1}{\norm*{\Opconnf_L}_{1\to 1}^3}\connf^{\star 4}_L\left(\orig,\orig\right)\left(\frac{\norm*{\Opconnf_L}_{2\to 2}}{\norm*{\Opconnf_L}_{1\to 1}}\right)^2,\\
        \frac{1}{\norm*{\Opconnf_L}_{1\to 1}^3}\connf^{\star 1 \star 2 \cdot 2}_L\left(\orig,\orig\right) &\leq \min\left\{\frac{1}{\norm*{\Opconnf_L}_{1\to 1}^3}\connf^{\star 4}_L\left(\orig,\orig\right),\frac{1}{\norm*{\Opconnf_L}_{1\to 1}^2}\connf^{\star 3}_L\left(\orig,\orig\right)\right\},\\
        \frac{1}{\norm*{\Opconnf_L}_{1\to 1}^4}\connf^{\star 2 \star 2 \cdot 2}_L\left(\orig,\orig\right)  & \leq \frac{1}{\norm*{\Opconnf_L}_{1\to 1}^3}\connf^{\star 4}_L\left(\orig,\orig\right).
    \end{align}
    Furthermore,
    \begin{equation}
        \Ecal(L) \leq \frac{C}{\norm*{\Opconnf_L}_{1\to 1}^2}\connf^{\star 3}_L\left(\orig,\orig\right) + \frac{C^2}{\norm*{\Opconnf_L}_{1\to 1}^3}\connf^{\star 4}_L\left(\orig,\orig\right),
    \end{equation}
    where $C$ is the constant appearing in the definition of $\Ecal(L)$ (see \eqref{eqn:ECAL}).
\end{corollary}

\begin{proof}
    The first set of bounds either follow directly from Lemma~\ref{lem:DiagramOrderGeneral}, or with the additional observation that
    \begin{equation}
        \sup_{x\in\HypDim}\connf^{\star 2}\left(x,\orig\right) \leq \norm*{\Opconnf_L}_{1\to 1}\esssup_{x,y\in\HypDim}\connf\left(x,y\right) \leq \norm*{\Opconnf_L}_{1\to 1}.
    \end{equation}
    The bound on $\Ecal(L)$ follows by omitting the $-\frac{1}{2}\connf_L\left(x,\orig\right)$ in the integrand.
\end{proof}

For the following arguments it will be convenient to use the notation
\begin{multline}
\label{eqn:OMEGA}
    \Omega(L) := \frac{1}{\norm*{\Opconnf_L}_{1\to 1}^3}\connf^{\star 2\star 2\cdot 1}_L\left(\orig,\orig\right) + \frac{1}{\norm*{\Opconnf_L}_{1\to 1}^4}\connf^{\star 5}_L\left(\orig,\orig\right)\\
    + \frac{1}{\norm*{\Opconnf_L}_{1\to 1}^4}\connf^{\star 2\star 2\cdot 2}_L\left(\orig,\orig\right) + \frac{1}{\norm*{\Opconnf_L}_{1\to 1}^5}\connf^{\star 6}_L\left(\orig,\orig\right).
\end{multline}

\begin{prop}
\label{prop:PiBounds}
For $\lambda\geq 0$, $L>0$ and $x\in\HypDim$, define $\omega^{(0)}_{\lambda,L}\left(x\right)$ and $\omega^{(1)}_{\lambda,L}\left(x\right)$ by
    \begin{align}
         \omega^{(0)}_{\lambda,L}\left(x\right) &:= \lambda\pi^{(0)}_{\lambda,L}\left(x,\orig\right) - \frac{1}{2}\lambda^3\connf^{\star 2}_L\cdot \connf^{\star 2}_L\left(x,\orig\right) \label{eqn:defn:omegazero}\\
         \omega^{(1)}_{\lambda,L}\left(x\right) &:= \lambda\pi^{(1)}_{\lambda,L}\left(x,\orig\right) - \lambda^2\connf^{\star 2}_L\cdot\connf_L\left(x,\orig\right) - \lambda^3\connf^{\star 2}_L\cdot \connf^{\star 2}_L\left(x,\orig\right) - \lambda^3\connf^{\star 3}_L\cdot \connf_L\left(x,\orig\right).\label{eqn:defn:omegaone}
    \end{align}
    
    Suppose $\left\{\connf_L\right\}_{L>0}$ satisfies Assumption~\ref{Model_Assumption}, and let $n_0\geq 2$ and $N\geq 1$ be fixed. Then there exists $L_0,C_0<\infty$ such that for $L\geq L_0$ and $\lambda\leq \lambda_\mathrm{c}(L)$,
    \begin{align}
        \max\left\{\norm*{\omega^{(0)}_{\lambda,L}}_1, \norm*{\omega^{(1)}_{\lambda,L}}_1 , \lambda\norm*{\sum^{n_0}_{n=3}\left(-1\right)^n\OpLace^{(n)}_{\lambda,L}}_{1\to 1}\right\} &\leq C_0 \Omega(L),\\
         \lambda\norm*{\sum^\infty_{n=N}\left(-1\right)^n\OpLace^{(n)}_{\lambda,L}}_{1\to 1} &\leq C_0 \beta\left(L\right)^N.
    \end{align}
\end{prop}

\begin{proof}[Proof Outline]
    The proof follows essentially the same line as that of \cite[Proposition~3.4]{dickson2025expansion}, except we need point-wise upper and lower bounds for $\pi^{(0)}_{\lambda,L}$ and $\pi^{(1)}_{\lambda,L}$ rather than just integral bounds. In what follows we outline the probabilistic ideas used to construct these bounds in lower detail than in \cite{dickson2025expansion} -- the details can be found therein. We do, however, go into more detail regarding how the \emph{pointwise} bounds can be obtained -- as this is where the argument diverges from \cite{dickson2025expansion}.
    
    \paragraph{Step $1$:} First let us establish the notation that given two events, $E$ and $F$, we use $E\circ F$ to denote the vertex disjoint occurrence of $E$ and $F$. Now to get an upper bound for $\lambda\pi^{(0)}_{\lambda,L}\left(x,\orig\right)$, we begin by writing
    \begin{align}
        &\mathbb{P}_{\lambda,L}\left(\dconn{\orig}{x}{\xi^{x,\orig}}\right) - \connf_L\left(x,\orig\right)\nonumber\\
        &\hspace{2cm}=\left(1-\connf_L\left(x,\orig\right)\right)\mathbb{P}_{\lambda,L}\left(\exists u,v\in\eta\colon u\ne v, \orig\sim u, \orig\sim v,\right. \nonumber\\
        &\hspace{7.5cm}\left.\left\{\conn{u}{x}{\xi^x}\right\}\circ \left\{\conn{u}{x}{\xi^x}\right\}\right).
    \end{align}
    Since we have $\left\{\conn{u}{x}{\xi^x}\right\}\circ \left\{\conn{u}{x}{\xi^x}\right\}\subset \left\{\conn{u}{x}{\xi^x}\right\}\cap \left\{\conn{u}{x}{\xi^x}\right\}$, we can bound
    \begin{multline}
        \mathbb{P}_{\lambda,L}\left(\exists u,v\in\eta\colon u\ne v, \orig\sim u, \orig\sim v,\left\{\conn{u}{x}{\xi^x}\right\}\circ \left\{\conn{u}{x}{\xi^x}\right\}\right) \\\leq 1 - \mathbb{P}_{\lambda,L}\left(\#\left\{u\in\eta\colon\orig\sim u, \conn{u}{x}{\xi^x}\right\} \leq 1\right).
    \end{multline}
    Since $\eta$ is distributed as a Poisson point process, $\#\left\{u\in\eta\colon\orig\sim u, \conn{u}{x}{\xi^x}\right\}$ is distributed as a Poisson random variable and Mecke's formula gives its expectation as $\lambda\connf_L\star\tau_{\lambda,L}\left(x,\orig\right)$. Then since $1-\e^{-x}-x\e^{-x}\leq \frac{1}{2}x^2$ for all $x\geq 0$ (express the remainder as an infinite sum and take its derivative), we have
    \begin{equation}
        \lambda\pi^{(0)}_{\lambda,L}\left(x,\orig\right) \leq \frac{1}{2}\lambda^3\left(1-\connf_L\left(x,\orig\right)\right)\left(\connf_L\star\tau_{\lambda,L}\left(x,\orig\right)\right)^2 \leq \frac{1}{2}\lambda^3\left(\connf_L\star\tau_{\lambda,L}\left(x,\orig\right)\right)^2
    \end{equation}
    for all $x\in\HypDim$. By applying $\tau_{\lambda,L} \leq \connf_L + \lambda\connf_L\star\tau_{\lambda,L}$ (recall from \eqref{eqn:FirstStep}) twice,
    \begin{align}
        \frac{1}{2}\lambda^3\left(\connf_L\star\tau_{\lambda,L}\left(x,\orig\right)\right)^2 &\leq \frac{1}{2}\lambda^3\left(\connf_L^{\star 2}\left(x,\orig\right) + \lambda\connf_L^{\star 3}\left(x,\orig\right) + \lambda^2\connf_L^{\star 3}\star\tau_{\lambda,L}\left(x,\orig\right)\right)^2\nonumber\\
        &= \frac{1}{2}\lambda^3\connf_L^{\star 2}\cdot \connf_L^{\star 2}\left(x,\orig\right) + \mathrm{Rem}^{(0,+)}_{\lambda,L}\left(x\right),
    \end{align}
    where the last equality defines the remainder term $\mathrm{Rem}^{(0,+)}_{\lambda,L}\colon\HypDim \to \R_+$. By using Lemma~\ref{lem:tautoconnf} (and $\tau_{\lambda,L} \leq \connf_L + \lambda\connf_L\star\tau_{\lambda,L}$ an additional time), one can show
    \begin{align}
        \int_\HypDim\abs*{\mathrm{Rem}^{(0,+)}_{\lambda,L}\left(x\right)}\mu\left(\dd x\right) &= \frac{1}{2}\left(2\lambda^4\connf_L^{\star 5} + \lambda^5\connf_L^{\star 6} + 2\lambda^5\connf_L^{\star 5}\star\tau_{\lambda,L} + 2\lambda^6\connf_L^{\star 6}\star\tau_{\lambda,L}\right.\nonumber\\
        &\hspace{6cm}\left.+ \lambda^7\connf_L^{\star 6}\star\tau_{\lambda,L}^{\star 2}\right)\left(\orig,\orig\right) \nonumber\\
        &\leq \frac{1}{2}\left(2\lambda^4\connf_L^{\star 5} + 3\lambda^5\connf_L^{\star 6} + 4\lambda^6\connf_L^{\star 6}\star\tau_{\lambda,L}+ \lambda^7\connf_L^{\star 6}\star\tau_{\lambda,L}^{\star 2}\right)\left(\orig,\orig\right)\nonumber\\
        & \leq  \lambda^4\connf_L^{\star 5}\left(\orig,\orig\right) + \frac{1}{2}\left(3+4K+K^2\right)\lambda^5\connf_L^{\star 6}\left(\orig,\orig\right).
    \end{align}
    This gives the required upper bound on $\lambda\pi^{(0)}_{\lambda,L}\left(x,\orig\right)$ for $\lambda\leq \lambda_\mathrm{c}(L) \leq C\norm*{\Opconnf_L}_{1\to 1}^{-1}$.

    \paragraph{Step $2$:}
    For the lower bound on $\lambda\pi^{(0)}_{\lambda,L}\left(x,\orig\right)$, we observe that
    \begin{equation}
        \left\{\dconn{x}{\orig}{\xi^{x,\orig}}\right\}\setminus\left\{x\sim \orig\right\} \supset \left\{x\not\sim\orig\right\}\cap\left\{\#\left\{u\in\eta\colon \orig\sim u\sim x\right\}\geq 2\right\}.
    \end{equation}
    As above $\#\left\{u\in\eta\colon\orig\sim u, \conn{u}{x}{\xi^x}\right\}$ is distributed as a Poisson random variable with expectation $\lambda\connf_L\star\tau_{\lambda,L}\left(x,\orig\right)$, and furthermore is independent of $\left\{x\not\sim\orig\right\}$. Then since $1-\e^{-x}-x\e^{-x}\geq \frac{1}{2}x^2 - \frac{1}{2}x^3$ for all $x\in\R$,
    \begin{align}
        \lambda\pi^{(0)}_{\lambda,L}\left(x,\orig\right) &\geq \frac{1}{2}\lambda\left(1-\connf_L\left(x,\orig\right)\right)\left(\lambda^2\connf_L^{\star 2}\cdot \connf_L^{\star 2}\left(x,\orig\right) - \lambda^3\connf_L^{\star 2}\cdot \connf_L^{\star 2}\cdot \connf_L^{\star 2}\left(x,\orig\right)\right)\nonumber\\
        &= \frac{1}{2}\lambda^3\connf_L^{\star 2}\cdot \connf_L^{\star 2}\left(x,\orig\right) + \mathrm{Rem}^{(0,-)}_{\lambda,L}\left(x\right),
    \end{align}
    where the last equality defines the remainder term $\mathrm{Rem}^{(0,-)}_{\lambda,L}\colon\HypDim \to \R$. By expanding the parentheses, applying the triangle inequality, and using $\connf_L\in\left[0,1\right]$,
    \begin{equation}
        \int_\HypDim\abs*{\mathrm{Rem}^{(0,-)}_{\lambda,L}\left(x\right)}\mu\left(\dd x\right) \leq \frac{1}{2}\lambda^3\connf_L^{\star 2\star 2\cdot 1}\left(\orig,\orig\right) + \frac{1}{2}\lambda^4\connf_L^{\star 2 \star 2 \cdot 2}\left(\orig,\orig\right).
    \end{equation}
    This gives the required lower bound on $\lambda\pi^{(0)}_{\lambda,L}\left(x,\orig\right)$ for $\lambda\leq \lambda_\mathrm{c}(L) \leq C\norm*{\Opconnf_L}_{1\to 1}^{-1}$.

    \paragraph{Step $3$:}
    To prove the upper bound on $\lambda\pi^{(1)}_{\lambda,L}\left(x,\orig\right)$, we first recall from \cite[Proposition~3.15]{dickson2025expansion} (with the obvious adjustments for our $\HypDim$ case) that
    \begin{equation}\label{eqn:PiOneFullBound}
        0\leq \lambda\pi^{(1)}_{\lambda,L}\left(x,\orig\right) \leq \sum_{j\in\left\{1,2\right\}}\sum_{k\in\left\{1,2,3\right\}}\int_\HypDim \int_\HypDim\psi^{(j)}_1\left(x;w,u\right)\psi^{(k)}_0\left(w,u;\orig\right)\mu\left(\dd w\right)\mu\left(\dd u\right),
    \end{equation}
    where $\psi^{(1)}_1,\psi^{(2)}_1,\psi^{(1)}_0,\psi^{(2)}_0,\psi^{(3)}_0$ are functions given in \cite[Definition~3.11]{dickson2025expansion} explicitly in terms of $\connf_L$, $\tau_{\lambda,L}$ and similar functions. The contributions from most of these can easily be shown to be negligible for our purposes using the inequality $\tau_{\lambda,L} \leq \connf_L + \lambda\connf_L\star\tau_{\lambda,L}$ and lemmas~\ref{lem:tautoconnf} and \ref{lem:DiagramOrderGeneral}. The sole remaining term in \eqref{eqn:PiOneFullBound} comes from $j=2$ and $k=3$ and is
    \begin{equation}
        \int_\HypDim \int_\HypDim\psi^{(2)}_1\left(x;w,u\right)\psi^{(3)}_0\left(w,u;\orig\right)\mu\left(\dd w\right)\mu\left(\dd u\right) = \lambda^2\tau_{\lambda,L}\left(x,\orig\right)\connf_L\star\tau_{\lambda,L}\left(x,\orig\right).
    \end{equation}
    Then by repeatedly applying  $\tau_{\lambda,L} \leq \connf_L + \lambda\connf_L\star\tau_{\lambda,L}$,
    \begin{align}
        \lambda^2\tau_{\lambda,L}\left(x,\orig\right)\connf_L\star\tau_{\lambda,L}\left(x,\orig\right) &\leq \lambda^2\left(\connf_L + \lambda\connf^{\star 2}_L + \lambda^2\connf^{\star 3}_L + \lambda^3\connf_L^{\star 4} + \lambda^4\connf_L^{\star 4}\star \tau_{\lambda,L}\right)\left(x,\orig\right)\nonumber\\
        &\hspace{1cm}\times\left(\connf_L^{\star 2} + \lambda\connf^{\star 3}_L + \lambda^2\connf^{\star 4}_L + \lambda^3\connf_L^{\star 5} + \lambda^4\connf_L^{\star 5}\star \tau_{\lambda,L}\right)\left(x,\orig\right)\nonumber\\
        &= \lambda^2\connf_L\cdot\connf_L^{\star 2}\left(x,\orig\right) + \lambda^3\connf_L\cdot\connf_L^{\star 3}\left(x,\orig\right) + \lambda^3\connf^{\star 2}\cdot\connf_L^{\star 2}\left(x,\orig\right) \nonumber\\
        &\hspace{1cm}+ \mathrm{Rem}^{(1,+)}_{\lambda,L}\left(x\right),
    \end{align}
    where the last equality defines the remainder term $\mathrm{Rem}^{(1,+)}_{\lambda,L}\colon\HypDim \to \R_+$. Like for the remainder term $\mathrm{Rem}^{(0,+)}_{\lambda,L}$, we can bound this by using Lemma~\ref{lem:tautoconnf} (with some applications of Lemma~\ref{lem:DiagramOrderGeneral} to simplify):
    \begin{align}
        \int_\HypDim\abs*{\mathrm{Rem}^{(1,+)}_{\lambda,L}\left(x\right)}\mu\left(\dd x\right) &\leq 3\lambda^4\connf^{\star 5}_L\left(\orig,\orig\right) + \left(4+2K\right)\lambda^5\connf^{\star 6}_L\left(\orig,\orig\right) +\left(3+2K\right)\lambda^6\connf^{\star 7}_L\left(\orig,\orig\right) \nonumber \\
        &\hspace{1cm}+ \left(2+2K\right)\lambda^7\connf^{\star 8}_L\left(\orig,\orig\right) + \left(1+2K+K^2\right)\lambda^8\connf^{\star 9}_L\left(\orig,\orig\right)\nonumber\\
        & \leq 3\lambda^4\connf^{\star 5}_L\left(\orig,\orig\right) \nonumber\\
        & \hspace{0.5cm}+ \left(4 + 2K + \left(3+2K\right)\lambda\norm*{\Opconnf_L}_{2\to 2} + \left(2+2K\right)\lambda^2\norm*{\Opconnf_L}^2_{2\to 2} \right.\nonumber\\
        &\hspace{3cm}\left.+ \left(1+2K+K^2\right)\lambda^3\norm*{\Opconnf_L}^3_{2\to 2}\right)\lambda^5\connf^{\star 6}_L\left(\orig,\orig\right).
    \end{align}
    Since $\lambda\leq \lambda_\mathrm{c}(L) \leq C\norm*{\Opconnf_L}_{1\to 1}^{-1}$ and $\frac{\norm*{\Opconnf_L}_{2\to 2}}{\norm*{\Opconnf_L}_{1\to 1}} \to 0$, this gives the required upper bound on $\lambda\pi^{(1)}_{\lambda,L}\left(x,\orig\right)$ for $\lambda\leq \lambda_\mathrm{c}(L)$.

    \paragraph{Step $4$:}
    For our lower bound on $\lambda\pi^{(1)}_{\lambda,L}\left(x,\orig\right)$, we aim to identify events that are contained in the events in the definition in \eqref{eq:LE:Pin_def}. Recall from Definition~\ref{defn:lacesetup} that $\eta_{\left<A\right>}$ is an the $A$-thinning of $\eta$. Also let the event $\left\{\xconn{x}{y}{\xi^{x,y}}{=n}\right\}$ denote the event that $x$ and $y$ are connected by a path of length exactly $n$ in $\xi^{x,y}$, but not by a path of length $n-1$ or less. Then
    \begin{equation}
        \lambda\pi^{(1)}_{\lambda,L}\left(x,\orig\right) \geq \lambda^2\int_\HypDim\left(\mathbb{P}_{\lambda,L}\left(\Gcal_1\left(x,u,\orig\right)\right) + \mathbb{P}_{\lambda,L}\left(\Gcal_2\left(x,u,\orig\right)\right) + \mathbb{P}_{\lambda,L}\left(\Gcal_3\left(x,u,\orig\right)\right)\right)\mu\left(\dd u\right),
    \end{equation}
    where
    \begin{align}
        \Gcal_1\left(x,u,\orig\right) &= \left\{\adja{\orig}{u}{\xi^{u,\orig}_0}\right\} \cap \left\{\adja{u}{x}{\xi^{x,u}_1}\right\} \cap \left\{x\not\in \eta^{x,u}_{1,\left<\orig\right>}\right\},\\
        \Gcal_2\left(x,u,\orig\right) &= \left\{\adja{\orig}{u}{\xi^{u,\orig}_0}\right\} \cap \left\{\xconn{u}{x}{\xi^{x,u}_1}{=2}\right\} \cap \left\{x\not\in \eta^{x,u}_{1,\left<\orig\right>}\right\},\\
        \Gcal_3\left(x,u,\orig\right) &= \left\{\adja{\orig}{u}{\xi^{u,\orig}_0}\right\} \cap \left\{\adja{u}{x}{\xi^{x,u}_1}\right\} \cap \left\{x\in \eta^{x,u}_{1,\left<\orig\right>}\right\}\nonumber\\
        &\hspace{5cm}\cap \left\{\exists v\in\eta_0\colon \adja{\orig}{v}{\xi^\orig_0},x\not\in \eta^{x,u_0}_{1,\left<v\right>}\right\}.
    \end{align}
    In \cite{dickson2025expansion}, Lemma~2.6 and the proof of Lemma~3.14 prove that
    \begin{align}
        \lambda^2\int_\HypDim\mathbb{P}_{\lambda,L}\left(\Gcal_1\left(x,u,\orig\right)\right)\mu\left(\dd u\right) &= \lambda^2\connf_L^{\star2}\cdot\connf_L\left(x,\orig\right),\\
        \lambda^2\int_\HypDim\mathbb{P}_{\lambda,L}\left(\Gcal_2\left(x,u,\orig\right)\right)\mu\left(\dd u\right) &\geq \lambda^3\connf_L\left(x,\orig\right)\int_\HypDim\left(1-\connf_L\right)\cdot\left(\connf^{\star 2}_L - \frac{1}{2}\lambda\connf^{\star 2}_L\cdot\connf^{\star 2}_L\right)\left(x,u\right)\nonumber\\
        &\hspace{5cm}\times\connf_L\left(u,\orig\right)\mu\left(\dd u\right),\nonumber\\
        &= \lambda^3\connf_L\cdot \connf^{\star 3}_L\left(x,\orig\right) - \lambda^3\connf_L\cdot\left(\connf_L\star\left(\connf_L\cdot \connf^{\star 2}_L\right)\right)\left(x,\orig\right)\nonumber\\
        &\hspace{1cm} -\frac{1}{2}\lambda^4\connf_L\cdot\left(\connf_L\star\left(\connf^{\star 2}_L\cdot \connf^{\star 2}_L\right)\right)\left(x,\orig\right)\nonumber\\
        &\hspace{1cm} + \frac{1}{2}\lambda^4\connf_L\cdot\left(\connf_L\star\left(\connf^{\star 2}_L\cdot \connf^{\star 2}_L\cdot \connf_L\right)\right)\left(x,\orig\right) \label{eqn:Gtwo_bound}\\
        \lambda^2\int_\HypDim\mathbb{P}_{\lambda,L}\left(\Gcal_3\left(x,u,\orig\right)\right)\mu\left(\dd u\right) &\geq \lambda^3\connf^{\star 2}_L\cdot\left(1-\connf_L\right)\cdot\left(\connf^{\star 2}-\frac{1}{2}\lambda\connf^{\star 2}_L\cdot\connf^{\star 2}_L\right)\left(x,\orig\right)\nonumber\\
        & = \lambda^3\connf^{\star 2}_L\cdot\connf^{\star 2}_L\left(x,\orig\right) - \lambda^3\connf^{\star 2}_L\cdot\connf^{\star 2}_L\left(x,\orig\right), \nonumber\\
        &\hspace{1cm} -\frac{1}{2}\lambda^4\connf_L\cdot\connf^{\star 2}_L\cdot\connf^{\star 2}_L\left(x,\orig\right) + \frac{1}{2}\lambda^4\connf_L\cdot\connf^{\star 2}_L\cdot\connf^{\star 2}_L\left(x,\orig\right).\label{eqn:Gthree_bound}
    \end{align}
    Therefore 
    \begin{equation}
        \lambda\pi^{(1)}_{\lambda,L}\left(x,\orig\right) \geq \lambda^2\connf_L^{\star2}\cdot\connf_L\left(x,\orig\right) + \lambda^3\connf_L\cdot \connf^{\star 3}_L\left(x,\orig\right) + \lambda^3\connf^{\star 2}_L\cdot\connf^{\star 2}_L\left(x,\orig\right) + \mathrm{Rem}^{(1,-)}_{\lambda,L}\left(x\right),
    \end{equation}
    where the remainder term $\mathrm{Rem}^{(1,-)}_{\lambda,L}\colon\HypDim \to \R$ is the sum of the remaining terms in \eqref{eqn:Gtwo_bound} and \eqref{eqn:Gthree_bound}. Then by Lemma~\ref{lem:DiagramOrderGeneral}, using the triangle inequality, and occasionally bounding $\connf_L\leq 1$, one can find
    \begin{equation}
        \int_\HypDim\abs*{\mathrm{Rem}^{(1,-)}_{\lambda,L}\left(x\right)}\mu\left(\dd x\right) \leq 2\lambda^3\connf_L^{\star 2 \star 2 \cdot 1}\left(\orig,\orig\right) + 2\lambda^4\connf_L^{\star 2 \star 2 \cdot 2}\left(\orig,\orig\right).
    \end{equation}
    This gives the required lower bound on $\lambda\pi^{(1)}_{\lambda,L}\left(x,\orig\right)$ for $\lambda\leq \lambda_\mathrm{c}(L) \leq C\norm*{\Opconnf_L}_{1\to 1}^{-1}$.

    \paragraph{Step $5$:}
    The upper bounds for the remaining terms follow the arguments in \cite{dickson2025expansion} more closely. For $\lambda\norm*{\sum^{n_0}_{n=3}\left(-1\right)^n\OpLace^{(n)}_{\lambda,L}}_{1\to 1}$, we can bound  $\lambda\norm*{\OpLace^{(n)}_{\lambda,L}}_{1\to 1}$ using integrals like \eqref{eqn:PiOneFullBound}. We can then use Lemma~\ref{lem:tautoconnf} and identify $\connf_L^{\star2\star 2\cdot 1}\left(\orig,\orig\right)$, $\connf_L^{\star2\star 2\cdot 2}\left(\orig,\orig\right)$, or long loops of $\connf_L$ in the resulting integrals. These long loops will at least be of length $5$, and loops of length $\geq 7$ can be simplified to a multiple of $\connf^{\star 6}\left(\orig,\orig\right)$ using Lemma~\ref{lem:DiagramOrderGeneral}. Since we only need to bound finitely many such terms, we have the required 
    \begin{equation}
        \lambda\norm*{\sum^{n_0}_{n=3}\left(-1\right)^n\OpLace^{(n)}_{\lambda,L}}_{1\to 1} = \LandauBigO{\Omega(L)}.
    \end{equation}
    Finally the bounds on $\OpLace^{(n)}_{\lambda,L}$ we found in \eqref{eqn:triangleboundsLace} are sufficient to get the $\LandauBigO{\beta(L)^N}$ bound on the tail of the sum.
\end{proof}

\section{Estimating the Spectral Radius}
\label{sec:estimatingspectralradius}

\begin{definition}
    Recall that we use the notation that given a real-valued function $f$, $\left(f\right)_+$ is the positive part and $\left(f\right)_-$ is the absolute value of the negative part. That is, 
    \begin{align}
        \left(f\right)_+ &:= \frac{1}{2}\left(\abs*{f} + f\right),\\
        \left(f\right)_- &:= \frac{1}{2}\left(\abs*{f} - f\right).
    \end{align}
    This means that $f = \left(f\right)_+ - \left(f\right)_-$, and $\abs*{f} = \left(f\right)_+ + \left(f\right)_-$.

    Also for $L>0$, $\lambda\geq0$, and $n\in\N$ we let $\left(\Opconnf_L + \OpLace_{\lambda,n,L}\right)_+$ and $\left(\Opconnf_L + \OpLace_{\lambda,n,L}\right)_-$ denote the convolution operators using $\left(\connf_L+ \pi_{\lambda,n,L}\right)_+$ and $\left(\connf_L+ \pi_{\lambda,n,L}\right)_-$ respectively.
\end{definition}

\begin{lemma}
\label{lem:SpecRadiusNegative}
For all $L>0$, $\lambda\geq 0$, and $n\in\N$,
\begin{equation}
    \abs*{\rho_{1\to 1}\left(\Opconnf_L + \OpLace_{\lambda,n,L}\right) - \int_\HypDim\left(\connf_L + \pi_{\lambda,n,L}\right)\left(x,\orig\right)\mu\left(\dd x\right)} \leq 2\int_\HypDim\left(\connf_L + \pi_{\lambda,n,L}\right)_-\left(x,\orig\right)\mu\left(\dd x\right)
\end{equation}
\end{lemma}

\begin{proof}
    From the sub-multiplicativity of the operator norm and the triangle inequality,
    \begin{equation}
        \rho_{1\to 1}\left(\Opconnf_L + \OpLace_{\lambda,n,L}\right) \leq \norm*{\Opconnf_L + \OpLace_{\lambda,n,L}}_{1\to 1}\leq \norm*{\left(\Opconnf_L + \OpLace_{\lambda,n,L}\right)_+}_{1\to 1} + \norm*{\left(\Opconnf_L + \OpLace_{\lambda,n,L}\right)_-}_{1\to 1}.
    \end{equation}
    Then from the isometry invariance of $\connf_L$ and $\pi_{\lambda,n,L}$,
    \begin{multline}
        \rho_{1\to 1}\left(\Opconnf_L + \OpLace_{\lambda,n,L}\right) \leq \int_\HypDim\left(\connf_L+ \pi_{\lambda,n,L}\right)_+\left(x,\orig\right)\mu\left(\dd x\right) + \int_\HypDim\left(\connf_L+ \pi_{\lambda,n,L}\right)_-\left(x,\orig\right)\mu\left(\dd x\right)\\
         = \int_\HypDim\left(\connf_L+ \pi_{\lambda,n,L}\right)\left(x,\orig\right)\mu\left(\dd x\right) + 2\int_\HypDim\left(\connf_L+ \pi_{\lambda,n,L}\right)_-\left(x,\orig\right)\mu\left(\dd x\right).
    \end{multline}

    Note that by isometry invariance and Lemma~\ref{lem:IsoInvtoCommute}, $\left(\Opconnf_L + \OpLace_{\lambda,n,L}\right)_+$ and $\left(\Opconnf_L + \OpLace_{\lambda,n,L}\right)_-$ commute. Then from Lemmas~\ref{lem:SpectrumPerturbation} and \ref{lem:spectralradiusofPositive}, 
    \begin{align}
        \rho_{1\to 1}\left(\Opconnf_L + \OpLace_{\lambda,n,L}\right) &\geq \rho_{1\to 1}\left(\left(\Opconnf_L + \OpLace_{\lambda,n,L}\right)_+\right) - \norm*{\left(\Opconnf_L + \OpLace_{\lambda,n,L}\right)_-}_{1\to 1}\nonumber\\
        & = \int_\HypDim\left(\connf_L+ \pi_{\lambda,n,L}\right)_+\left(x,\orig\right)\mu\left(\dd x\right) - \int_\HypDim\left(\connf_L+ \pi_{\lambda,n,L}\right)_-\left(x,\orig\right)\mu\left(\dd x\right) \nonumber\\
        & = \int_\HypDim\left(\connf_L+ \pi_{\lambda,n,L}\right)\left(x,\orig\right)\mu\left(\dd x\right),
    \end{align}
    as required.
\end{proof}

Recall there exists $C<\infty$ such that $\lambda_\mathrm{c}\left(L\right) \leq \frac{C}{\norm*{\Opconnf_L}_{1\to 1}}$ for $L$ sufficiently large (see Theorem~\ref{thm:TriangleSmall}). Then define
\begin{multline}
\label{eqn:ECAL}
    \Ecal(L):= \frac{1}{\norm*{\Opconnf_L}_{1\to 1}}\int_\HypDim \left( \frac{C}{\norm*{\Opconnf_L}_{1\to 1}}\connf_L^{\star 2}\cdot\connf_L\left(x,\orig\right)\right.\\\left. + \frac{C^2}{2\norm*{\Opconnf_L}_{1\to 1}^2}\connf_L^{\star 2}\cdot\connf_L^{\star 2}\left(x,\orig\right) - \frac{1}{2}\connf_L\left(x,\orig\right)\right)_+\mu\left(\dd x\right).
\end{multline}

\begin{lemma}\label{lem:BoundIntegralNegative}
    Suppose $\left\{\connf_L\right\}_{L>0}$ satisfies Assumption~\ref{Model_Assumption}. Then for $L$ sufficiently large and $\lambda\leq \lambda_\mathrm{c}(L)$,
    \begin{equation}
        \lambda\int_\HypDim\left(\connf_L + \pi_{\lambda,1,L}\right)_-\left(x,\orig\right)\mu\left(\dd x\right)
        \leq  C\Ecal(L) + \LandauBigO{\Omega(L)}.
    \end{equation}
\end{lemma}

\begin{proof}
    This result mostly follows from Definition~\ref{def:LE:lace_expansion_coefficients} with Proposition~\ref{prop:PiBounds}. By applying these and observing $\left(f+g\right)_+\leq \left(f\right)_+ + \left(g\right)_+\leq \left(f\right)_+ + \abs*{g}$, we get
    \begin{align}
        &\lambda\int_\HypDim\left(\connf_L + \pi_{\lambda,1,L}\right)_-\left(x,\orig\right)\mu\left(\dd x\right) \nonumber\\
        &\hspace{1cm}= \int_\HypDim\left(\lambda\connf_L - \lambda^2\connf_L^{\star 2}\cdot \connf_L - \frac{1}{2}\lambda^3\connf_L^{\star 2}\cdot\connf_L^{\star 2} - \lambda^3\connf^{\star 3}\cdot \connf_L + \omega^{(0)}_{\lambda,L} - \omega^{(1)}_{\lambda,L}\right)_-\left(x,\orig\right)\mu\left(\dd x\right)\nonumber\\
        &\hspace{1cm} \leq \lambda\int_\HypDim\left(\lambda\connf_L^{\star 2}\cdot \connf_L + \frac{1}{2}\lambda^2\connf_L^{\star 2}\cdot\connf_L^{\star 2} + \lambda^2\connf^{\star 3}\cdot \connf_L - \connf_L\right)_+\left(x,\orig\right)\mu\left(\dd x\right)\nonumber\\
        &\hspace{10cm}+ \norm*{\omega^{(0)}_{\lambda,L}}_1 + \norm*{\omega^{(1)}_{\lambda, L}}_1.
    \end{align}
    Now observe that $\lambda^2\esssup_{x\in\HypDim}\connf^{\star 3}_L\left(x,\orig\right) \leq \triangle_{\lambda,L}$. Therefore
    \begin{align}
        &\lambda\int_\HypDim\left(\connf_L + \pi_{\lambda,1,L}\right)_-\left(x,\orig\right)\mu\left(\dd x\right)\nonumber\\
        &\hspace{1cm} \leq \lambda\int_\HypDim\left(\lambda\connf_L^{\star 2}\cdot \connf_L + \frac{1}{2}\lambda^2\connf_L^{\star 2}\cdot\connf_L^{\star 2} - \frac{1}{2}\connf_L\right)_+\left(x,\orig\right)\mu\left(\dd x\right) \nonumber\\
        &\hspace{4cm}+ \lambda\int_\HypDim\left(\triangle_{\lambda,L}\connf_L - \frac{1}{2}\connf_L\right)_+\left(x,\orig\right)\mu\left(\dd x\right) + \LandauBigO{\Omega(L)}.
    \end{align}
    Since $\triangle_{\lambda,L}$ is monotone increasing in $\lambda$ and Theorem~\ref{thm:TriangleSmall} shows $\lim_{L\to\infty}\triangle_{\lambda_\mathrm{c}(L),L}=0$,
    \begin{equation}
        \lambda\int_\HypDim\left(\triangle_{\lambda,L}\connf_L - \frac{1}{2}\connf_L\right)_+\left(x,\orig\right)\mu\left(\dd x\right) = \lambda\left(\triangle_{\lambda,L} - \frac{1}{2}\right)_+\int_\HypDim\connf_L\left(x,\orig\right)\mu\left(\dd x\right)= 0
    \end{equation}
    for $L$ sufficiently large and $\lambda\leq \lambda_\mathrm{c}(L)$. Finally using $\lambda_\mathrm{c}(L) \leq C\norm*{\Opconnf_L}_{1\to 1}^{-1}$ (see \eqref{eqn:criticalintensityBound} in Theorem~\ref{thm:TriangleSmall}) and replacing each remaining factors of $\lambda$ with this upper bound gives the result. 
\end{proof}

\begin{prop}
\label{prop:radiusExpansion}
    Suppose $\left\{\connf_L\right\}_{L>0}$ satisfies Assumption~\ref{Model_Assumption}. Then there exists $C_0<\infty$ such that for $L$ sufficiently large and $\lambda\leq \lambda_\mathrm{c}(L)$,
    \begin{multline}
        \abs*{\lambda\rho_{1\to 1}\left(\Opconnf_L + \OpLace_{\lambda,L}\right) - \lambda\norm*{\Opconnf_L}_{1\to 1} + \lambda^2\connf^{\star 3}\left(\orig,\orig\right) +\frac{3}{2}\lambda^3\connf^{\star 4}\left(\orig,\orig\right)} \\\leq C_0\left(\Omega(L)+\Ecal(L)+\beta(L)^N\right)
    \end{multline}
    for any $N\geq 2$.
\end{prop}

\begin{proof}
    Begin by writing $\Opconnf_L + \OpLace_{\lambda,L} = \left(\Opconnf_L + \OpLace_{\lambda,1,L}\right) + \sum^\infty_{n=2}\left(-1\right)^n\OpLace^{(n)}_{\lambda,L}$. By isometry invariance and Lemma~\ref{lem:IsoInvtoCommute}, $\Opconnf_L + \OpLace_{\lambda,1,L}$ and $\sum^\infty_{n=2}\left(-1\right)^n\OpLace^{(n)}_{\lambda,L}$ commute and therefore Proposition~\ref{prop:PiBounds} gives
    \begin{equation}
        \abs*{\rho_{1\to 1}\left(\Opconnf_L + \OpLace_{\lambda,L}\right) - \rho_{1\to 1}\left(\Opconnf_L + \OpLace_{\lambda,1,L}\right)} \leq \norm*{\sum^\infty_{n=2}\left(-1\right)^n\OpLace^{(n)}_{\lambda,L}}_{1\to 1}
        =\LandauBigO{\Omega(L) + \beta(L)^N}
    \end{equation}
    for any $N\geq 2$.
    Then by lemmas~\ref{lem:SpecRadiusNegative} and \ref{lem:BoundIntegralNegative} we have
    \begin{equation}
        \abs*{\rho_{1\to 1}\left(\Opconnf_L + \OpLace_{\lambda,1,L}\right) - \int_\HypDim\left(\connf + \pi_{\lambda,1,L}\right)\left(x,\orig\right)\dd x} = \LandauBigO{\Ecal(L) + \Omega(L)}.
    \end{equation}
    Then Proposition~\ref{prop:PiBounds} gives
    \begin{equation}
        \int_\HypDim\left(\connf + \pi_{\lambda,1,L}\right)\left(x,\orig\right)\dd x = \lambda\int_\HypDim\connf_L\left(x,\orig\right)\mu\left(\dd x\right) - \lambda^2\connf^{\star 3}\left(\orig,\orig\right) -\frac{3}{2}\lambda^3\connf^{\star 4}\left(\orig,\orig\right)
        + \LandauBigO{\Omega(L)}
    \end{equation}
    as required.
\end{proof}

Now we ready to prove the main theorem of the paper.

    \begin{theorem}
    \label{thm:CritIntensityExpansionFull}
        Suppose $\left\{\connf_L\right\}_{L>0}$ satisfies Assumption~\ref{Model_Assumption}. Then for all $d\geq 2$ and $N\in\N$,
        \begin{multline}\label{eqn:TheoremExpansionFull}
            \norm*{\Opconnf_L}_{1\to 1}\lambda_\mathrm{c}(L) = 1 + \frac{\connf_L^{\star 3}(\orig,\orig)}{\norm*{\Opconnf_L}_{1\to 1}^2} + \frac{3}{2}\frac{\connf_L^{\star 4}(\orig,\orig)}{\norm*{\Opconnf_L}_{1\to 1}^3} \\+ \LandauBigO{\left(\frac{\connf_L^{\star 3}(\orig,\orig)}{\norm*{\Opconnf_L}_{1\to 1}^2}\right)^2 + \left(\frac{\connf_L^{\star 4}(\orig,\orig)}{\norm*{\Opconnf_L}_{1\to 1}^3}\right)^2  + \Omega(L) + \Ecal(L) + \beta(L)^N}
        \end{multline}
        as $L\to\infty$.
    \end{theorem}

\begin{proof}
    From Proposition~\ref{prop:radiusExpansion} we have
    \begin{equation}
        \lambda\rho_{1\to 1}\left(\Opconnf_L + \OpLace_{\lambda,L}\right) = \lambda a - \lambda^2 b - \lambda^3 c +  r,
    \end{equation}
    where $a=\norm*{\Opconnf_L}_{1\to 1}$, $b= \connf^{\star 3}\left(\orig,\orig\right)$, $c=\frac{3}{2}\connf^{\star 4}\left(\orig,\orig\right)$, and the remainder term is given by $r=\LandauBigO{\Omega(L) + \Ecal(L) + \beta(L)^N}$. From Corollary~\ref{cor:diagramBounds} and \eqref{eqn:normRatio} in Theorem~\ref{thm:TriangleSmall}, we have $\frac{b}{a^2}\ll 1$, $\frac{c}{a^3}\ll 1$, $r\ll 1$, and $\lambda_\mathrm{c}(L)=\LandauBigO{a^{-1}}$. Then from Lemma~\ref{lem:CriticalEquation} we have
    \begin{align}
        \lambda_\mathrm{c}a &= 1 + \lambda_\mathrm{c}^2b + \lambda_\mathrm{c} ^3 c -  r \nonumber\\
        & = 1 + \frac{b}{a^2}\left(1 + \lambda_\mathrm{c}^2b + \lambda_\mathrm{c} ^3 c - r\right)^2 + \frac{c}{a^3}\left(1 + \lambda_\mathrm{c}^2b + \lambda_\mathrm{c} ^3 c -  r\right)^3 -r \nonumber\\
        & = 1 + \frac{b}{a^2}\left(1 + \LandauBigO{\frac{b}{a^2}} + \LandauBigO{\frac{c}{a^3}} -r\right)^2 + \frac{c}{a^3}\left(1 + \LandauBigO{\frac{b}{a^2}} + \LandauBigO{\frac{c}{a^3}} -r\right)^3 -r\nonumber\\
        & = 1 + \frac{b}{a^2} + \frac{c}{a^3} + \LandauBigO{ r + \frac{b^2}{a^4} + \frac{c^2}{a^6}},
    \end{align}
    as required.
\end{proof}

\begin{proof}[Proof of Theorem~\ref{thm:CritIntensityExpansion}]
Corollary~\ref{cor:diagramBounds} lets us simplify Theorem~\ref{thm:CritIntensityExpansionFull} by bounding many of the terms in \eqref{eqn:TheoremExpansionFull}. This quickly gives us
\begin{equation}
    \norm*{\Opconnf_L}_{1\to 1}\lambda_\mathrm{c}(L) = 1 +\LandauBigO{\frac{\norm*{\Opconnf_L}_{2\to 2}}{\norm*{\Opconnf_L}_{1\to 1}} + \Ecal(L) + \beta(L)^N}.
\end{equation}
Now observe that if one omits the $-\frac{1}{2}\connf_L$ term from the definition of $\Ecal(L)$, one gets the bound
\begin{equation}
    \Ecal(L) = \LandauBigO{\frac{\connf_L^{\star 3}(\orig,\orig)}{\norm*{\Opconnf_L}_{1\to 1}^2} + \frac{\connf_L^{\star 4}(\orig,\orig)}{\norm*{\Opconnf_L}_{1\to 1}^3}} = \LandauBigO{\frac{\norm*{\Opconnf_L}_{2\to 2}}{\norm*{\Opconnf_L}_{1\to 1}}},
\end{equation}
and Theorem~\ref{thm:TriangleSmall} implies
\begin{equation}
    \beta(L) = \LandauBigO{\sqrt{\frac{\norm*{\Opconnf_L}_{2\to 2}}{\norm*{\Opconnf_L}_{1\to 1}}}}.
\end{equation}
Therefore
\begin{equation}
    \norm*{\Opconnf_L}_{1\to 1}\lambda_\mathrm{c}(L) = 1 +\LandauBigO{\frac{\norm*{\Opconnf_L}_{2\to 2}}{\norm*{\Opconnf_L}_{1\to 1}} + \left(\frac{\norm*{\Opconnf_L}_{2\to 2}}{\norm*{\Opconnf_L}_{1\to 1}}\right)^\frac{N}{2}} = 1+ \LandauBigO{\frac{\norm*{\Opconnf_L}_{2\to 2}}{\norm*{\Opconnf_L}_{1\to 1}}},
\end{equation}
if we choose $N\geq 2$.

For the $L^1\to L^1$ operator norm on the left hand side, we use Lemma~\ref{lem:ConvolutionOperator} to get
\begin{equation}
    \norm*{\Opconnf_L}_{1\to 1} = \int_\HypDim\connf_L\left(x,\orig\right)\mu\left(\dd x\right).
\end{equation}

Finally, Theorem~\ref{thm:TriangleSmall} states that
\begin{equation}
    \frac{\norm*{\Opconnf_L}_{2\to 2}}{\norm*{\Opconnf_L}_{1\to 1}} = o\left(1\right)
\end{equation}
as $L\to\infty$, and the result follows.
\end{proof}

\section{Calculations for Specific Models}
\label{sec:CalcSpecific}

Theorem~\ref{thm:CritIntensityExpansionFull} has a number of terms in it, and it is not immediately obvious which terms will be the most relevant for any particular model. Corollary~\ref{cor:diagramBounds} gave a partial ordering to the terms, but this is far from complete. Clearly evaluating these terms for specific models reveals the ordering (we do this is Sections~\ref{sec:CalcBooleanDisc} and \ref{sec:CalcHeatKernel} for the Boolean disc RCM and heat kernel RCMs respectively), but we can also impose some order if we consider the family of models which we call `non-negative definite' RCMs.

\subsection{`Non-Negative Definite' RCMs}
\label{sec:CalcNonNegative}

\begin{definition}[Non-Negative Definite Models]\label{defn:non-negativeDefinite}
        Given $f\in L^2\left(\HypDim\right)$, let 
        \begin{equation}
            \left<f,\Opconnf_L f\right> := \int_\HypDim\int_\HypDim f(x) \connf_L\left(x,y\right)\overline{f(y)}\mu\left(\dd y\right) \mu\left(\dd x\right),
        \end{equation}
        where $\overline{f(y)}$ denotes the complex conjugate of $f(y)$.
        Note that since $\connf_L(x,y)=\connf_L(y,x)\in\R$ for all $x,y\in\HypDim$, we have $\left<f,\Opconnf_L f\right>\in\R$.
        A hyperbolic RCM is called \emph{non-negative definite} if the adjacency function $\connf_L$ is non-negative definite for all $L$ sufficiently large. An adjacency function $\connf_L$ is non-negative definite if it satisfies
        \begin{equation}\label{eqn:integralnon-negativedefinite}
            \left<f,\Opconnf_L f\right>  \geq 0
        \end{equation}
        for all $f\in L^2\left(\HypDim\right)$.
    \end{definition}

    \begin{observation}
        The heat kernel RCM is an instance of a non-negative definite model.
        For $t>0$, let $\kappa_t\colon L^2\left(\HypThree\right)\to L^2\left(\HypThree\right)$ denote the convolution operator associated with $K_3\left(t,\cdot,\cdot\right)$. As a heat kernel, $\kappa_t$ satisfies the semi-group property, and the symmetry of $K_3$ in $x$ and $y$ implies that $\kappa_t$ is self-adjoint.
    Therefore
    \begin{equation}
        \left<f, \kappa_t f\right> = \left<\kappa_{\frac{t}{2}} f, \kappa_{\frac{t}{2}}f\right> = \norm*{\kappa_{\frac{t}{2}} f}_2^2\geq 0.
    \end{equation}
    Therefore $\kappa_t$ is non-negative definite for all $t>0$.

    \end{observation}

    In Lemma~\ref{lem:RangeSphericalTransformExpand} below, it is proven that for non-negative definite $\connf_L$, the spherical transform satisfies
        \begin{equation}
            \widetilde{\connf}_L(s)\in\left[0,\infty\right)
        \end{equation}
    for Lebesgue-almost all $s\in \R$. A consequence of this is the following lemma.

\begin{lemma}
    \label{lem:deconvolvenon-negative}
    Suppose $\left\{\connf_L\right\}_{L>0}$ satisfies Assumption~\ref{enum:AssumptionFiniteDegree} and that $\widetilde{\connf}_L(s)\geq 0$ for all $s\in\R$ and $L>0$. Then there exists $K<\infty$ such that for all $L>0$ and $m\geq n\geq 1$
    \begin{equation}
        \sup_{x\in\HypDim}\connf_L^{\star m}\left(x,\orig\right) \leq K\norm*{\Opconnf_L}_{2\to 2}^{m-n}\connf_L^{\star n}\left(\orig,\orig\right).
    \end{equation}
\end{lemma}

\begin{proof}
    The result follows by writing $\connf_L^{\star m}(x,\orig)$ in terms of their spherical transforms. For all $x\in\HypDim$,
    \begin{align}
        \connf_L^{\star m}(x,\orig) & = \frac{1}{w} \int_{\mathbb{S}^{d-1}}\int_\R\e^{\left(is + \frac{d-1}{2}\right)A(x,b)}\widetilde{\connf}_L(s)^m\abs{\mathbf{c}(s)}^{-2}\dd s\dd b\nonumber\\
        &\leq \frac{1}{w} \int_{\mathbb{S}^{d-1}}\int_\R\abs*{\e^{\left(is + \frac{d-1}{2}\right)A(x,b)}}\widetilde{\connf}_L(s)^m\abs{\mathbf{c}(s)}^{-2}\dd s\dd b\nonumber\\
        &\leq \left(\esssup_{s\in\R}\widetilde{\connf}_L(s)^{m-n}\right)\frac{1}{w} \int_{\mathbb{S}^{d-1}}\int_\R\e^{\frac{d-1}{2}A(x,b)}\widetilde{\connf}_L(s)^{n}\abs{\mathbf{c}(s)}^{-2}\dd s\dd b\nonumber\\
        &\leq \norm*{\Opconnf_L}_{2\to 2}^{m-n}\left(\sup_{x\in\HypDim}Q_d(x)\right)\frac{\mathfrak{S}_{d-1}}{w} \int_\R\widetilde{\connf}_L(s)^{n}\abs{\mathbf{c}(s)}^{-2}\dd s \nonumber\\
        &= K\norm*{\Opconnf_L}_{2\to 2}^{m-n}\connf_L^{\star n}\left(\orig,\orig\right).
    \end{align}
    
\end{proof}

Lemma~\ref{lem:deconvolvenon-negative} differs from the first part of Lemma~\ref{lem:DiagramOrderGeneral} in that the addition of the assumption that $\widetilde{\connf}_L(s)\geq 0$ removes the requirement that $m$ is even. This allows us to replace instances of a cycle of $\connf_L$ functions (i.e. $\connf_L^{\star n}\left(\orig,\orig\right)$) with a cycle of $\connf_L$ functions of any shorter length, rather than only even shorter lengths.

As we shall see in Corollary~\ref{cor:NonNegDefFull} below, for non-negative definite models we can simplify the term $\Ecal(L)$ that appeared in Theorem~\ref{thm:CritIntensityExpansionFull}. Define
    \begin{equation}
        \Ecal^*(L) := \frac{1}{\norm*{\Opconnf_L}_{1\to 1}}\int_\HypDim \left( \frac{C^2}{\norm*{\Opconnf_L}_{1\to 1}^2}\connf_L^{\star 2}\cdot \connf_L^{\star 2}\left(x,\orig\right) - \frac{1}{2}\connf_L\left(x,\orig\right)\right)_+\mu\left(\dd x\right).
    \end{equation}

\begin{corollary}
\label{cor:DiagramOrderPositive}
    Suppose $\left\{\connf_L\right\}_{L>0}$ satisfies Assumption~\ref{enum:AssumptionFiniteDegree} and that $\widetilde{\connf}_L(s)\geq 0$ for all $s\in\R$ and $L>0$. Then there exists $K<\infty$ such that for all $L>0$,
    \begin{align}
        \frac{1}{\norm*{\Opconnf_L}_{1\to 1}^2}\connf^{\star 3}_L\left(\orig,\orig\right) &\leq K\left(\frac{\norm*{\Opconnf_L}_{2\to 2}}{\norm*{\Opconnf_L}_{1\to 1}}\right)^2,\\
        \frac{1}{\norm*{\Opconnf_L}_{1\to 1}^3}\connf^{\star 4}_L\left(\orig,\orig\right) & \leq K\frac{1}{\norm*{\Opconnf_L}_{1\to 1}^2}\connf^{\star 3}_L\left(\orig,\orig\right)\frac{\norm*{\Opconnf_L}_{2\to 2}}{\norm*{\Opconnf_L}_{1\to 1}},\label{eqn:loopfourTOloopthree}\\
        \frac{1}{\norm*{\Opconnf_L}_{1\to 1}^3}\connf^{\star 1 \star 2 \cdot 2}_L\left(\orig,\orig\right) &\leq \frac{1}{\norm*{\Opconnf_L}_{1\to 1}^3}\connf^{\star 4}_L\left(\orig,\orig\right),\label{eqn:fourcrossoneTOloopfour}\\
        \frac{1}{\norm*{\Opconnf_L}_{1\to 1}^4}\connf^{\star 5}_L\left(\orig,\orig\right) &  \leq K\frac{1}{\norm*{\Opconnf_L}_{1\to 1}^3}\connf^{\star 4}_L\left(\orig,\orig\right)\frac{\norm*{\Opconnf_L}_{2\to 2}}{\norm*{\Opconnf_L}_{1\to 1}},\\
        \frac{1}{\norm*{\Opconnf_L}_{1\to 1}^4}\connf^{\star 2 \star 2 \cdot 2}_L\left(\orig,\orig\right)  & \leq \frac{1}{\norm*{\Opconnf_L}_{1\to 1}^3}\connf^{\star 4}_L\left(\orig,\orig\right)\frac{\norm*{\Opconnf_L}_{2\to 2}}{\norm*{\Opconnf_L}_{1\to 1}},\label{eqn:phitwotwotwoTOloopfour}\\
        \frac{1}{\norm*{\Opconnf_L}_{1\to 1}^5}\connf^{\star 6}_L\left(\orig,\orig\right) &\leq K\frac{1}{\norm*{\Opconnf_L}_{1\to 1}^4}\connf^{\star 5}_L\left(\orig,\orig\right)\frac{\norm*{\Opconnf_L}_{2\to 2}}{\norm*{\Opconnf_L}_{1\to 1}}.
    \end{align}
    Furthermore,
    \begin{equation}
        \Ecal^*(L) \leq \frac{C^2}{\norm*{\Opconnf_L}_{1\to 1}^3}\connf^{\star 4}_L\left(\orig,\orig\right),
    \end{equation}
    where $C$ is the constant appearing in the definitions of $\Ecal(L)$ and $\Ecal^*(L)$ (see \eqref{eqn:ECAL}).
\end{corollary}

\begin{proof}
    The first inequality follows from applying Lemma~\ref{lem:deconvolvenon-negative} with $m=3$ and $n=1$, and using $\connf(\orig,\orig)\leq 1$. Inequality \eqref{eqn:fourcrossoneTOloopfour} simply follows from $\connf(\orig,\orig)\leq 1$, and \eqref{eqn:phitwotwotwoTOloopfour} follows from $\esssup_{x,y\in\HypDim}\connf^{\star 2}\left(x,y\right) \leq \widetilde{\connf}_L(0)$ (via Lemma~\ref{lem:deconvolvenon-negative} with $m=2$ and $n=1$). Note that in this non-negative definite case we cannot get a better bound for \eqref{eqn:fourcrossoneTOloopfour} by using the $\connf$-triangle diagram rather than the $\connf$-quadrilateral diagram because \eqref{eqn:loopfourTOloopthree} means that the resulting bound will not be better than what we have here. The remaining inequalities in the first block are direct applications of Lemma~\ref{lem:deconvolvenon-negative} with appropriate choices of $m$ and $n$. For the upper bound on $\Ecal^*(L)$, omit the $-\frac{1}{2}\connf_L$ term in the integrand.
\end{proof}

Like we defined $\Ecal^*(L)$ as a reduced version of $\Ecal(L)$, we define $\Omega^*(L)$ as a reduced version of $\Omega(L)$ (defined in \eqref{eqn:OMEGA}):
\begin{equation}
    \Omega^*(L):= 
    \frac{1}{\norm*{\Opconnf_L}_{1\to 1}^3}\connf_L^{\star 2 \star 2 \cdot 1}\left(\orig,\orig\right) + \frac{1}{\norm*{\Opconnf_L}_{1\to 1}^3}\connf_L^{\star 4}\left(\orig,\orig\right)\frac{\norm*{\Opconnf_L}_{2\to 2}}{\norm*{\Opconnf_L}_{1\to 1}}.
\end{equation}

    \begin{corollary}\label{cor:NonNegDefFull}
        Suppose $\left\{\connf_L\right\}_{L>0}$ satisfies Assumption~\ref{Model_Assumption} and is non-negative definite for sufficiently large $L$. Then
        \begin{equation}\label{eqn:nonnegdefOrderterms}
            \frac{\connf_L^{\star 4}(\orig,\orig)}{\norm*{\Opconnf_L}_{1\to 1}^3} = o\left(\frac{\connf_L^{\star 3}(\orig,\orig)}{\norm*{\Opconnf_L}_{1\to 1}^2}\right), \qquad \Ecal^*(L) + \Omega^*(L) = \LandauBigO{\frac{\connf_L^{\star 4}(\orig,\orig)}{\norm*{\Opconnf_L}_{1\to 1}^3}},
        \end{equation}
        and for all $N\in\N$,
\begin{multline}\label{eqn:NonNegExpansionFull}
            \norm*{\Opconnf_L}_{1\to 1}\lambda_\mathrm{c}(L) = 1 + \frac{\connf_L^{\star 3}(\orig,\orig)}{\norm*{\Opconnf_L}_{1\to 1}^2} + \frac{3}{2}\frac{\connf_L^{\star 4}(\orig,\orig)}{\norm*{\Opconnf_L}_{1\to 1}^3} \\+ \LandauBigO{\left(\frac{\connf_L^{\star 3}(\orig,\orig)}{\norm*{\Opconnf_L}_{1\to 1}^2}\right)^2  + \Omega^*(L) + \Ecal^*(L) + \beta(L)^N}
        \end{multline}
        as $L\to\infty$.
    \end{corollary}

\begin{proof}
    The bounds on $\norm*{\Opconnf_L}_{1\to 1}^{-3}\connf_L^{\star 4}(\orig,\orig)$, $\Ecal^*(L)$, and $\Omega^*(L)$ follow directly from Corollary~\ref{cor:DiagramOrderPositive} and Theorem~\ref{thm:TriangleSmall}.

    For the approximation of $\lambda_\mathrm{c}$, we start from Theorem~\ref{thm:CritIntensityExpansionFull}. The above bound on $\norm*{\Opconnf_L}_{1\to 1}^{-3}\connf_L^{\star 4}(\orig,\orig)$ means we can neglect the $\left(\norm*{\Opconnf_L}_{1\to 1}^{-3}\connf_L^{\star 4}(\orig,\orig)\right)^2$ term from the error. Simplifying $\Omega(L)$ to $\Omega^*(L)$ follows directly from Corollary~\ref{cor:DiagramOrderPositive}. For the $\Ecal(L)$ term, observe that by the triangle inequality
    \begin{equation}\label{eqn:EcalTOEcalStar}
        \Ecal(L) \leq \frac{1}{2}\Ecal^*(L) + \frac{1}{\norm*{\Opconnf_L}_{1\to 1}}\int_{\HypDim}\left(\frac{C}{\norm*{\Opconnf_L}_{1\to 1}}\connf_L^{\star 2}\left(x,\orig\right) - \frac{1}{4}\right)_+\connf_L\left(x,\orig\right)\mu\left(\dd x\right).
    \end{equation}
    Then by Corollary~\ref{cor:DiagramOrderPositive} we have
    \begin{equation}
        \frac{C}{\norm*{\Opconnf_L}_{1\to 1}}\sup_{x\in\HypDim}\connf_L^{\star 2}\left(x,\orig\right) \leq CK\frac{\norm*{\Opconnf_L}_{2\to 2}}{\norm*{\Opconnf_L}_{1\to 1}}\connf_L\left(\orig,\orig\right).
    \end{equation}
    By Theorem~\ref{thm:TriangleSmall} this vanishes as $L\to\infty$, and therefore the second term in \eqref{eqn:EcalTOEcalStar} vanishes for sufficiently large $L$.   
\end{proof}

\subsection{Heat Kernel RCM}
\label{sec:CalcHeatKernel}

Here we show how Theorem~\ref{cor:HeatKernelResult} follows from Theorem~\ref{thm:CritIntensityExpansionFull}. Recall that for the heat kernel RCM in $d=3$,  $0<\Acal_L\leq \left(2\pi L\right)^{\frac{3}{2}}\e^{\frac{1}{2}L}$ and
\begin{multline}
    \connf_L\left(x,\orig\right) = \Acal_L K_3\left(L,x,\orig\right) \\= \frac{\Acal_L}{\left(2\pi\right)^\frac{3}{2}}L^{-\frac{3}{2}}\exp\left(-\frac{1}{2}L\right)\frac{\distThree{x,\orig}}{\sinh \distThree{x,\orig}}\exp\left(-\frac{1}{2L}\distThree{x,\orig}^2\right)
\end{multline}
for all $L>0$ and $x\in\HypThree$.

\begin{lemma}
\label{lem:heatkernelGood}
    The heat kernel RCM in $d=3$ satisfies Assumption~\ref{Model_Assumption}.
\end{lemma}

\begin{proof}
    First let us address assumption~\ref{enum:AssumptionFiniteDegree}. Note that since $\connf_L(r)\in\left[0,1\right]$ for all $r> 0$, assumption~\ref{enum:AssumptionFiniteDegree} holds if and only if $\norm*{\Opconnf_L}_{1\to 1}<\infty$. Since $\mathfrak{S}_{2}=4\pi$, we have
    \begin{align}
        \norm*{\Opconnf_L}_{1\to 1} &=\int_\HypThree\connf_L\left(x,\orig\right)\mu\left(\dd x\right) \nonumber\\
        &= 4\pi\int^\infty_0\connf_L(r)\left(\sinh r\right)^2\dd r\nonumber\\
        &= \frac{4\pi\Acal_L }{\left(2\pi \right)^\frac{3}{2}} L^{-\frac{3}{2}}\e^{-\frac{1}{2}L}\int^\infty_0 r \sinh r \e^{-\frac{1}{2L}r^2}\dd r \nonumber\\
        & = \frac{\Acal_L }{\sqrt{2\pi}} L^{-\frac{3}{2}}\int^\infty_0 r\left(\e^{-\frac{1}{2L}\left(r-L\right)^2} - \e^{-\frac{1}{2L}\left(r+L\right)^2}\right)\dd r \nonumber\\
        &= \frac{\Acal_L }{\sqrt{2\pi}} L^{-\frac{3}{2}}\left(\int^\infty_{-L} \left(r_1+L\right)\e^{-\frac{1}{2L}r_1^2} \dd r_1 - \int^\infty_{L} \left(r_2-L\right)\e^{-\frac{1}{2L}r_2^2} \dd r_2 \right),
    \end{align}
    where in this last step we have performed two substitutions. Bringing together like terms then gives
    \begin{align}
        \norm*{\Opconnf_L}_{1\to 1} &= \frac{\Acal_L }{\sqrt{2\pi}} L^{-\frac{3}{2}}\left(\int^L_{-L}r\e^{-\frac{1}{2L}r^2}\dd r + L\int^L_{-L}\e^{-\frac{1}{2L}r^2}\dd r + 2L\int^\infty_L \e^{-\frac{1}{2L}r^2}\dd r\right)\nonumber\\
        &=\frac{\Acal_L }{\sqrt{2\pi}} L^{-\frac{1}{2}}\left(\int^L_{-L}\e^{-\frac{1}{2L}r^2}\dd r + 2\int^\infty_L \e^{-\frac{1}{2L}r^2}\dd r\right)\nonumber\\
        &= \frac{\Acal_L}{\sqrt{2\pi L}}\int^\infty_{-\infty}\e^{-\frac{1}{2L}r^2}\dd r = \Acal_L .
    \end{align}
    Therefore assumption~\ref{enum:AssumptionFiniteDegree} holds.

    To verify assumption~\ref{enum:AssumptionLongRange}, calculate
    \begin{align}
        \frac{\int^R_0\connf_L(r)\left(\sinh r\right)^2\dd r}{\int^\infty_0\connf_L(r)\left(\sinh r\right)^2\dd r} &= \frac{4\pi}{\left(2\pi L\right)^\frac{3}{2}}\e^{-\frac{1}{2}L}\int^R_0 r\sinh r\e^{-\frac{1}{2L}r^2}\dd r \nonumber\\
        & = \frac{1}{\sqrt{2\pi}}L^{-\frac{3}{2}}\int^R_0 r\left(\e^{-\frac{1}{2L}\left(r-L\right)^2} - \e^{-\frac{1}{2L}\left(r+L\right)^2}\right)\dd r \nonumber\\
        &= \frac{1}{\sqrt{2\pi}}L^{-\frac{3}{2}}\left(\int^L_{L-R}\left(-r_1+L\right)\e^{-\frac{1}{2L}r^2_1}\dd + \int^{L+R}_L\left(L-r_2\right)\e^{-\frac{1}{2L}r^2_2}\dd r_2\right)\nonumber\\
        &= \frac{1}{\sqrt{2\pi}}L^{-\frac{3}{2}}\int^{L+R}_{L-R}\left(L-r\right)\e^{-\frac{1}{2L}r^2}\dd r\nonumber\\
        & \leq \frac{1}{\sqrt{2\pi}}L^{-\frac{3}{2}} \times 2R\times R\e^{-\frac{1}{2L}\left(L-R\right)^2}
    \end{align}
    for $L\geq R$. In particular,
    \begin{equation}
        \frac{\int^R_0\connf_L(r)\left(\sinh r\right)^2\dd r}{\int^\infty_0\connf_L(r)\left(\sinh r\right)^2\dd r} = \LandauBigO{L^{-\frac{3}{2}}\e^{-\frac{1}{2}L}}
    \end{equation}
    as $L\to\infty$ for all fixed $R<\infty$. Therefore assumption~\ref{enum:AssumptionLongRange} holds.
\end{proof}

\begin{lemma}
    For heat kernel RCM in $d=3$, for all $L>0$,
    \begin{equation}
        \norm*{\Opconnf_L}_{1\to 1} = \Acal_L, \qquad \norm*{\Opconnf_L}_{2\to 2} = \Acal_L\e^{-\frac{1}{2}L},
    \end{equation}
    and
    \begin{equation}
        \connf^{\star 2}_L\left(\orig,\orig\right) = \frac{\Acal_L^2}{4\pi^{\frac{3}{2}}}L^{-\frac{3}{2}}\e^{-L}.
    \end{equation}

    Therefore
    \begin{equation}
        \beta(L)=\LandauBigO{\Acal^{\frac{1}{2}}_L L^{-\frac{3}{4}}\e^{-\frac{3}{4}L}}.
    \end{equation}
\end{lemma}

\begin{proof}
    We have already established in the proof of Lemma~\ref{lem:heatkernelGood} that $\norm*{\Opconnf_L}_{1\to 1} = \Acal_L$.

    Now recall
    \begin{equation}
        \norm*{\Opconnf_L}_{2\to 2} = \widetilde{\connf}_L(0) = 4\pi\int^\infty_0\connf_L(r)Q_3(r)\left(\sinh r\right)^2\dd r.
    \end{equation}
    In \cite{dickson2024NonUniqueness} it is noted that $Q_3(r)=\frac{r}{\sinh r}$, and therefore
    \begin{equation}
        \norm*{\Opconnf_L}_{2\to 2} = \frac{4\pi\Acal_L }{\left(2\pi \right)^\frac{3}{2}}L^{-\frac{3}{2}}\e^{-\frac{1}{2}L}\int^\infty_0 r^2\e^{-\frac{1}{2L}r^2}\dd r = \frac{\Acal_L }{\sqrt{2\pi}}L^{-\frac{3}{2}}\e^{-\frac{1}{2}L} \int^\infty_{-\infty} r^2\e^{-\frac{1}{2L}r^2}\dd r.
    \end{equation}
    Then
    \begin{equation}
        \frac{1}{\sqrt{2\pi L}}\int^\infty_{-\infty} r^2\e^{-\frac{1}{2L}r^2}\dd r = L, 
    \end{equation}
    and
    \begin{equation}
        \norm*{\Opconnf_L}_{2\to 2} = \Acal_L \e^{-\frac{1}{2}L}.
    \end{equation}

    Now from the semi-group property,
    \begin{equation}
        \connf^{\star 2}_L\left(\orig,\orig\right) = \Acal_L^2 K_3\left(2L,\orig,\orig\right) = \frac{\Acal^2_L}{\left(4\pi\right)^\frac{3}{2}}L^{-\frac{3}{2}}\e^{-L},
    \end{equation}
    and
    \begin{equation}
        \beta(L) = \left(\frac{\norm*{\Opconnf_L}_{2\to 2}}{\norm*{\Opconnf_L}_{1\to 1}}\frac{\connf^{\star 2}_L\left(\orig,\orig\right) }{\norm*{\Opconnf_L}_{1\to 1}}\right)^\frac{1}{2} = \Acal^\frac{1}{2}_L\frac{1}{2\sqrt{2}\pi^{\frac{3}{4}}}L^{-\frac{3}{4}}\e^{-\frac{3}{4}L}.
    \end{equation}
\end{proof}

\begin{lemma}\label{lem:heatkernelTermsSize}
    For heat kernel RCM in $d=3$, for $n\geq 1$ and all $L>0$,
    \begin{equation}
        \connf_L^{\star n}\left(\orig,\orig\right) = \frac{\Acal_L^n}{\left(2\pi n\right)^\frac{3}{2}} L^{-\frac{3}{2}}\e^{-\frac{n}{2}L}.
    \end{equation}
    As $L\to\infty$,
    \begin{align}
        \connf_L^{\star 1\star 2\cdot 2}\left(\orig,\orig\right) &= \frac{2^{\frac{3}{2}}\Acal_L^5}{\left(4\pi\right)^\frac{7}{2}}L^{-\frac{9}{2}}\e^{-\frac{5}{2}L}\int^\infty_0\frac{r^3}{\sinh r}\dd r + \LandauBigO{\Acal_L^5 L^{-\frac{11}{2}}\e^{-\frac{5}{2}L}}\\
        \connf_L^{\star 2\star 2\cdot 2}\left(\orig,\orig\right) &= \frac{\Acal_L^6}{\left(4\pi\right)^\frac{7}{2}}L^{-\frac{9}{2}}\e^{-3L}\int^\infty_0\frac{r^3}{\sinh r}\dd r + \LandauBigO{\Acal_L ^6L^{-\frac{11}{2}}\e^{-3L}}.
    \end{align}
    For $L$ sufficiently large,
    \begin{equation}
        \Ecal^*\left(L\right)=0.
    \end{equation}
\end{lemma}

\begin{proof}
        The expression for $\connf_L^{\star n}\left(\orig,\orig\right)$ follows from the heat kernel satisfying the semi-group property:
    \begin{equation}
        \int_{\HypThree}K_3\left(L_1,x,y\right)K_3\left(L_2,y,\orig\right)\mu\left(\dd y\right) = K_3\left(L_1+L_2,x,\orig\right)
    \end{equation}
    for all $L_1,L_2>0$ and $x\in \HypThree$. Therefore
    \begin{equation}
        \connf_L^{\star n}\left(x,\orig\right) = \Acal_L^n K_3\left(nL,x,\orig\right),
    \end{equation}
    for all $L>0$, $n\geq 1$, and $x\in\HypThree$. In particular, the result follows from setting $x=\orig$.

    By using $\connf_L^{\star 1\star 2 \cdot 2}\left(\orig,\orig\right) = \int_{\HypThree}\connf_L\left(x,\orig\right)\connf^{\star 2}_L\left(x,\orig\right)^2 \mu\left(\dd x\right)$ and this expression for $\connf_L^{\star n}\left(x,\orig\right)$, we arrive at
    \begin{align}
        \connf_L^{\star 1\star 2 \cdot 2}\left(\orig,\orig\right) &= 4\pi\int^\infty_0\left(\frac{\Acal_L }{\left(2\pi L\right)^\frac{3}{2}}\e^{-\frac{1}{2}L}\frac{r}{\sinh r}\e^{-\frac{1}{2L}r^2}\right)\left(\frac{\Acal_L^2}{\left(4\pi L\right)^\frac{3}{2}}\e^{-L}\frac{r}{\sinh r}\e^{-\frac{1}{4L}r^2}\right)^2\left(\sinh r\right)^2\dd r\nonumber\\
        & = \Acal_L ^5\frac{\pi}{2}\frac{1}{\left(2\pi L\right)^\frac{9}{2}}\e^{-\frac{5}{2}L} \int^\infty_0\frac{r^3}{\sinh r}\e^{-\frac{1}{L}r^2}\dd r\nonumber\\
        & = \Acal_L^5 2^{-\frac{11}{2}}\pi^{-\frac{7}{2}}L^{-\frac{9}{2}}\e^{-\frac{5}{2}L} \left(\int^\infty_0\frac{r^3}{\sinh r}\dd r - \frac{1}{L}\int^\infty_0\frac{r^5}{\sinh r}\dd r + \LandauBigO{\frac{1}{L^2}}\right)
    \end{align}
    as $L\to\infty$. Note that since $\frac{r^n}{\sinh r}\sim 2r^n\e^{-r}$, this error bound follows from dominated convergence. If we similarly use $\connf_L^{\star 2\star 2 \cdot 2}\left(\orig,\orig\right) = \int_{\HypThree}\connf^{\star 2}_L\left(x,\orig\right)^3 \mu\left(\dd x\right)$, we get
    \begin{align}
        \connf_L^{\star 2\star 2 \cdot 2}\left(\orig,\orig\right) &= 4\pi\int^\infty_0\left(\frac{\Acal_L^2}{\left(4\pi L\right)^\frac{3}{2}}\e^{-L}\frac{r}{\sinh r}\e^{-\frac{1}{4L}r^2}\right)^3\left(\sinh r\right)^2\dd r\nonumber\\
        & = \frac{\Acal_L^6}{\left(4\pi\right)^\frac{7}{2}}L^{-\frac{9}{2}}\e^{-3L} \int^\infty_0\frac{r^3}{\sinh r}\e^{-\frac{3}{4L}r^2}\dd r\nonumber\\
        &=\frac{\Acal_L^6}{\left(4\pi\right)^\frac{7}{2}}L^{-\frac{9}{2}}\e^{-3L}\left( \int^\infty_0\frac{r^3}{\sinh r}\dd r - \frac{3}{4L} \int^\infty_0\frac{r^5}{\sinh r}\dd r + \LandauBigO{\frac{1}{L^2}}\right)
    \end{align}

    Since $\norm*{\Opconnf_L}_{1\to 1}=\Acal_L $, we have
    \begin{equation}
    \label{eqn:errorstar}
        \Ecal^*\left(L\right) = \frac{C}{\Acal_L }\int_\HypThree\left(\frac{C^2}{2\Acal_L^2}\left(\connf_L^{\star 2}\left(r\right)\right)^2 - \frac{1}{2}\connf_L\left(r\right)\right)_+\left(\sinh r\right)^2\dd r.
    \end{equation}
    Using our expressions for $\connf_L^{\star 2}(r)$ and $\connf_L(r)$, we have
    \begin{equation}
        \left(\frac{C^2}{2\Acal_L^2}\left(\connf_L^{\star 2}\left(r\right)\right)^2 - \frac{1}{2}\connf_L\left(r\right)\right)_+ = \connf_L(r)\left(\frac{C^2}{16}\frac{\Acal_L}{\left(2\pi L\right)^\frac{3}{2}}\e^{-\frac{3}{2}L}\frac{r}{\sinh r} - \frac{1}{2}\right)_+.
    \end{equation}
    Since $\frac{r}{\sinh r}< 1$ for all $r>0$ and $\Acal_L\leq \left(2\pi L\right)^\frac{3}{2}\e^{\frac{1}{2}L}$, we have
    \begin{equation}
        \frac{C^2}{16}\frac{\Acal_L}{\left(2\pi L\right)^\frac{3}{2}}\e^{-\frac{3}{2}L}\frac{r}{\sinh r} = \LandauBigO{\e^{-L}}
    \end{equation}
    uniformly in $r>0$. Therefore the integrand in \eqref{eqn:errorstar} vanishes for sufficiently large $L$.
\end{proof}

\begin{proof}[Proof of Theorem~\ref{cor:HeatKernelResult}]
    We can use the above results to control the terms in Theorem~\ref{thm:CritIntensityExpansionFull}. The dominant part of the error term is either $\norm*{\Opconnf_L}_{1\to 1}^{-4}\connf_L^{\star 5}\left(\orig,\orig\right)$ or $\norm*{\Opconnf_L}_{1\to 1}^{-3}\connf_L^{\star 2\star 2\cdot 1}\left(\orig,\orig\right)$ in $\Omega(L)$, depending upon the behaviour of $\Acal_L$. Since $\beta(L)$ decays exponentially, the term $\beta(L)^N$ can be made no bigger than the other terms by choosing $N$ sufficiently large ($N=4$ suffices). 
\end{proof}

\subsection{Boolean disc RCM}
\label{sec:CalcBooleanDisc}
Here we show how Theorem~\ref{cor:BooleanResult} follows from Theorem~\ref{thm:CritIntensityExpansionFull}. Recall that for the Boolean disc RCM,
\begin{equation}
    \connf_L(r) = \Id\left\{r<L\right\}.
\end{equation}

\begin{lemma}
    The Boolean disc RCM satisfies Assumption~\ref{Model_Assumption}.
\end{lemma}

\begin{proof}
    Since $\connf_L(r)\in\left[0,1\right]$ and is finite range, assumption~\ref{enum:AssumptionFiniteDegree} clearly holds.

    Assumption~\ref{enum:AssumptionLongRange} is verified by \cite[Lemma~A.2]{dickson2024NonUniqueness}. It is also simple to directly verify by slightly adapting the derivation of the leading order term in \eqref{eqn:BooleanOneNormAsymp} below.
\end{proof}

\begin{lemma}\label{lem:normsBoolean}
    For the Boolean disc RCM, as $L\to\infty$,
    \begin{align}
        \norm*{\Opconnf_L}_{2\to 2} &= \widetilde{\connf}_L(0) = \LandauBigO{L\e^{\frac{1}{2}\left(d-1\right)L}},\\
        \norm*{\Opconnf_L}_{1\to 1} & = \frac{\mathfrak{S}_{d-1}}{\left(d-1\right)2^{d-1}}\e^{\left(d-1\right)L}\left(1 - \left(1+o\left(1\right)\right)\begin{cases}
            \e^{-L}&\colon d=2,\\
            4L\e^{-2L}&\colon d=3,\\
            \frac{\left(d-1\right)^2}{d-3}\e^{-2L}&\colon d\geq 4.
        \end{cases}\right)
        \label{eqn:BooleanOneNormAsymp}
    \end{align}
    In particular, as $L\to\infty$
    \begin{equation}
        \beta(L) = \LandauBigO{L^\frac{1}{2}\e^{-\frac{1}{4}\left(d-1\right)L}}.
    \end{equation}
\end{lemma}

\begin{proof}
    For the Boolean model, 
    \begin{equation}
        \norm*{\Opconnf_L}_{2\to 2} = \widetilde{\connf}_L(0) = \mathfrak{S}_{d-1}\int^L_0Q_d(r)\left(\sinh r\right)^{d-1} \dd r,
    \end{equation}
    and \cite[Lemma~3.3]{dickson2024NonUniqueness} shows that there exists $c_d<\infty$ such that
    \begin{equation}
        Q_d(r) \leq c_d \max\left\{r,1\right\}\e^{-\frac{1}{2}\left(d-1\right)r}.
    \end{equation}
    Therefore there exists $C_d<\infty$ such that
    \begin{align}
        \widetilde{\connf}_L(0) &\leq \mathfrak{S}_{d-1}c_d\int^1_0\left(\sinh r\right)^{d-1}\dd r + \mathfrak{S}_{d-1}c_d\int^L_1r\e^{-\frac{1}{2}\left(d-1\right)r}\frac{1}{2^{d-1}}\e^{\left(d-1\right)r}\dd r\nonumber\\
        & \leq C_d + \frac{\mathfrak{S}_{d-1}c_d}{2^{d-1}}\int^L_0r\e^{\frac{1}{2}\left(d-1\right)r}\dd r\nonumber\\
        & = C_d +\frac{\mathfrak{S}_{d-1}c_d}{2^{d-1}}\left(\frac{2}{d-1} L\e^{\frac{1}{2}\left(d-1\right)L} - \frac{4}{\left(d-1\right)^2}\e^{\frac{1}{2}\left(d-1\right)L} + \frac{4}{\left(d-1\right)^2}\right) \nonumber\\
        & = \LandauBigO{L\e^{\frac{1}{2}\left(d-1\right)L}}.
    \end{align}

    On the other hand,
    \begin{align}
        \norm*{\Opconnf_L}_{1\to 1} &= \mathfrak{S}_{d-1}\int^L_0 \left(\sinh r\right)^{d-1}\dd r \nonumber\\
        &= \frac{\mathfrak{S}_{d-1}}{2^{d-1}}\sum^{d-1}_{k=0}\binom{d-1}{k}\left(-1\right)^k\int^L_0\exp\left(\left(d-1-2k\right)r\right)\dd r\nonumber\\
        &= \frac{\mathfrak{S}_{d-1}}{2^{d-1}}\sum^{d-1}_{k=0}\binom{d-1}{k}\left(-1\right)^k \times\begin{cases}
            \frac{1}{d-1-2k}\left(\e^{\left(d-1-2k\right)L}-1\right)&\colon k\ne\frac{d-1}{2},\\
            L&\colon k=\frac{d-1}{2}.
        \end{cases}
    \end{align}
    This means that
    \begin{equation}
        \norm*{\Opconnf_L}_{1\to 1} = \frac{\mathfrak{S}_{d-1}}{2^{d-1}}\left(\frac{1}{d-1}\e^{\left(d-1\right)L} - \left(1+o\left(1\right)\right)\begin{cases}
            1&\colon d=2,\\
            2L&\colon d=3,\\
            \frac{d-1}{d-3}\e^{\left(d-3\right)L}&\colon d\geq 4.
        \end{cases}\right)
    \end{equation}

    Therefore
    \begin{equation}
        \frac{\norm*{\Opconnf_L}_{2\to 2}}{\norm*{\Opconnf_L}_{1\to 1}} = \LandauBigO{L \e^{-\frac{1}{2}\left(d-1\right)L}}.
    \end{equation}
    We also have $\connf_L^{\star 2}\left(\orig,\orig\right) = \int_\HypDim\connf_L(x,\orig)^2\mu\left(\dd x\right) = \int_\HypDim\connf_L(x,\orig)\mu\left(\dd x\right) =\norm*{\Opconnf_L}_{1\to 1}$. Therefore
    \begin{equation}
        \beta(L) = \left(\frac{\norm*{\Opconnf_L}_{2\to 2}}{\norm*{\Opconnf_L}_{1\to 1}}\right)^\frac{1}{2} = \LandauBigO{L^\frac{1}{2}\e^{-\frac{1}{4}\left(d-1\right)L}}.
    \end{equation}
\end{proof}

\begin{lemma}\label{lem:diagramsBoolean}
    For the Boolean disc RCM, as $L\to\infty$,
    \begin{align}
        \frac{1}{\norm*{\Opconnf_L}_{1\to 1}^2}\connf^{\star 3}_L\left(\orig,\orig\right) &\sim \frac{\mathfrak{S}_{d-2}}{\mathfrak{S}_{d-1}}\frac{2^{d+2}}{d-1}\e^{-\frac{1}{2}\left(d-1\right)L},\\
        \frac{1}{\norm*{\Opconnf_L}_{1\to 1}^3}\connf^{\star 4}_L\left(\orig,\orig\right) &\sim \left(\frac{\mathfrak{S}_{d-2}}{\mathfrak{S}_{d-1}}\right)^2\frac{2^{2d+3}}{d-1}L\e^{-\left(d-1\right)L},\\
        \frac{1}{\norm*{\Opconnf_L}_{1\to 1}^3}\connf^{\star 1 \star 2 \cdot 2}_L\left(\orig,\orig\right) &\sim \left(\frac{\mathfrak{S}_{d-2}}{\mathfrak{S}_{d-1}}\right)^2\frac{2^{2d+2}}{d-1}L\e^{-\left(d-1\right)L},
    \end{align}
    and
    \begin{equation}
        \Ecal\left(L\right) = \LandauBigO{L\e^{-\left(d-1\right)L}}.
    \end{equation}
\end{lemma}

\begin{proof}
    Our main tool here will be an estimate of $\connf_L^{\star 2}\left(r\right)$. This quantity can be understood geometrically as the volume of the intersection of two balls with radius $L$ and centres a distance $r$ apart:
    \begin{equation}
        \connf_L^{\star 2}\left(r\right) = \habsd{B_L(\orig)\cap B_L(r)}.
    \end{equation}
    To estimate this volume, we will take a spherical segment of the ball and subtract off a cone with a hyperplane base. The volume of the intersection of the two balls will then be equal to two times this difference.

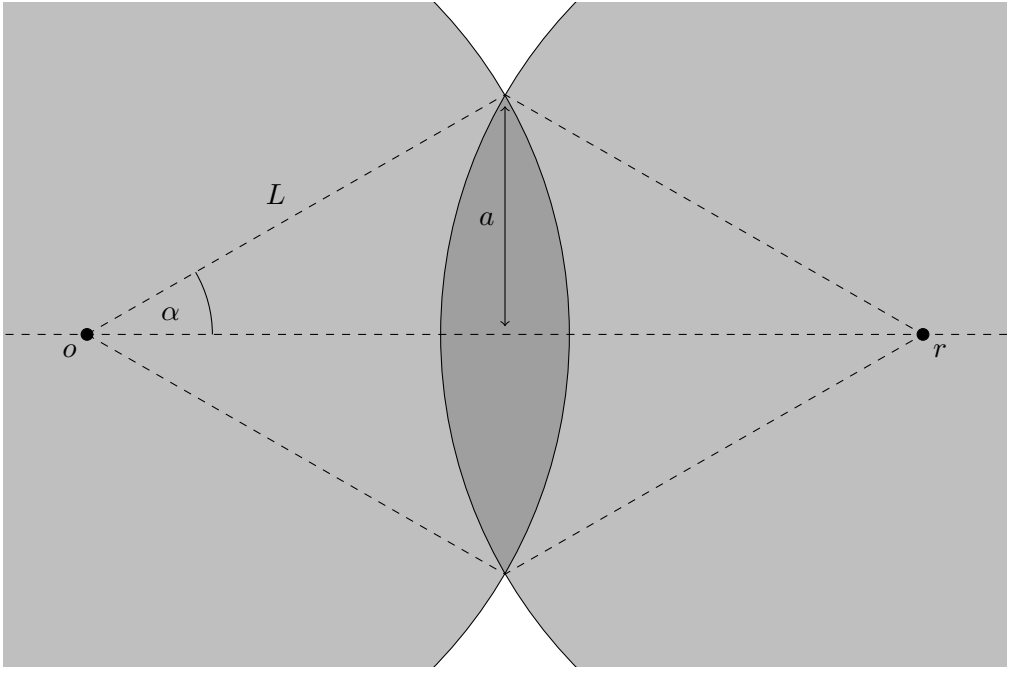
\begin{figure}
    \centering
    \begin{tikzpicture}[scale=1.1]
    \begin{scope}
        \clip (-6,-4) rectangle (6,4);
        \fill[semitransparent,fill=gray] (-5,0)  circle (5.77);
        \fill[semitransparent,fill=gray] (5,0)  circle (5.77);
        \draw (-5,0)  circle (5.77);
        \draw (5,0)  circle (5.77);
        \draw[dashed] (-5,0) -- (0,2.89) -- (5,0) -- (0,-2.89) -- (-5,0);
        \draw[dashed] (-10,0)--(10,0);
        \draw (-4,0.25) node{$\alpha$};
        \draw (-3.5,0) arc (0:30:1.5);
        \draw (-2.5,1.45) node[above left]{$L$};
        \draw[fill] (-5,0) circle (2pt) node[below left]{$\orig$};
        \draw[fill] (5,0) circle (2pt) node[below right]{$r$};
        \draw[<->] (0,0.1) -- (0,2.75);
        \draw (0,1.4) node[left]{$a$};
    \end{scope}
    \end{tikzpicture}
    \caption{Sketch and labelling of the cross-section of the region $B_L(\orig)\cap B_L(r)$.}
    \label{fig:circleIntersection}
\end{figure}

    Let $\alpha=\alpha(r,L)$ denote the internal half-angle of the spherical segment -- see Figure~\ref{fig:circleIntersection}. Then for $r\leq 2L$, by the hyperbolic trigonometric identity Lemma~\ref{lem:RightHyperbolicTriangles}
    \begin{equation}
        \cos \alpha = \frac{\tanh \frac{r}{2}}{\tanh L} = \left(\tanh \frac{r}{2}\right)\left(1 + 2 \e^{-2L} + \LandauBigO{\e^{-4L}}\right),
    \end{equation}
    and $\alpha\to \arccos\left(
    \tanh \frac{r}{2}\right)$ monotonically as $L\to\infty$. The volume of this spherical segment is then given by
    \begin{equation}\label{eqn:VolSphereSegment}
        \mathfrak{S}_{d-2}\int^{\alpha(r,L)}_0\left(\sin \theta\right)^{d-2}\dd \theta \times \int^L_0\left(\sinh \varrho\right)^{d-1}\dd \varrho.
    \end{equation}
    To help estimate $\habsd{B_L(\orig)\cap B_L(r)}$, we want to find an upper bound for the `cone' that is the convex hull of $\orig$ and the intersection of the \emph{spheres} of radius $L$ centred on $\orig$ and $r$. We do this by taking a very crude `volume of rotation' approach. Looking at Figure~\ref{fig:circleIntersection}, we can see a hyperbolic triangle when we take a cross-section that includes the axis of symmetry (the geodesic passing through $\orig$ and $r$). This triangle has side lengths $L$, $\frac{r}{2}$, and the as-yet-unknown $a$. It also has internal angles $\alpha$ at $\orig$, and $\frac{\pi}{2}$ where the edge $a$ meets the geodesic between $\orig$ and $r$. Since this triangle is a hyperbolic right-angled triangle, it has hyperbolic area $\leq \frac{\pi}{2}-\alpha\leq \frac{\pi}{2}$. The point on the triangle that is furthest from the axis of symmetry is the top of the $a$ edge. Therefore we can upper-bound the hyperbolic volume of the cone with
    \begin{equation}
        \left(\frac{\pi}{2}\right)\times\left(\mathfrak{S}_{d-2}\left(\sinh a\right)^{d-2}\right).
    \end{equation}
    The factor $\mathfrak{S}_{d-2}\left(\sinh a\right)^{d-2}$ appears here because it is the $\mathbb{H}^{d-2}$-volume of the hyperbolic sphere $\left\{x\in\mathbb{H}^{d-1}\colon \mathrm{dist}_{\mathbb{H}^{d-1}}\left(x,\orig\right)=a\right\}$. Now the hyperbolic sine rule (Lemma~\ref{lem:sinerule}) implies that $\sinh a = \sin \alpha \sinh L$, and therefore we can upper-bound the volume of the cone with
    \begin{equation}
        \frac{\pi}{2}\mathfrak{S}_{d-2}\left(\sin \alpha\right)^{d-2} \left(\sinh L\right)^{d-2} \\= \frac{\pi}{2^{d-1}}\mathfrak{S}_{d-2}\left(\sin \alpha\right)^{d-2} \,\e^{\left(d-2\right)L}\left(1-\e^{-2L}\right)^{d-2}.
    \end{equation}

    As $r,L\to\infty$ (with $r\leq 2L$),
    \begin{equation}
        \cos \alpha = 1 - 2\e^{-r} + 2\e^{-2L} + \LandauBigO{\e^{-2r}}
    \end{equation}
    and therefore
    \begin{equation}
    \alpha = 2\e^{-\frac{r}{2}}\left(1 - \e^{r-2L}\right)^\frac{1}{2}\left(1 + \LandauBigO{\e^{-\frac{1}{2}r}}\right).
    \end{equation}
    By \eqref{eqn:VolSphereSegment}, the volume of the spherical segment is therefore
    \begin{multline}
        \mathfrak{S}_{d-2}\frac{1}{d-1}\alpha^{d-1}\left(1+\LandauBigO{\alpha^2}\right)\times\frac{1}{d-1}\frac{1}{2^{d-1}}\e^{\left(d-1\right)L}\left(1+\LandauBigO{\e^{-2L}}\right) \\= \mathfrak{S}_{d-2}\frac{2}{\left(d-1\right)^2}\e^{\frac{d-1}{2}\left(2L-r\right)}\left(1-\e^{r-2L}\right)^\frac{d-1}{2}\left(1 + \LandauBigO{\e^{-\frac{1}{2}r}}\right).
    \end{multline}
    On the other hand, the volume of the cone is bounded above by
    \begin{equation}
        \frac{\pi}{2}\mathfrak{S}_{d-2}\e^{\frac{d-2}{2}\left(2L-r\right)}\left(1 - \e^{r-2L}\right)^\frac{d-2}{2}\left(1 +\LandauBigO{\e^{-\frac{1}{2}r}}\right).
    \end{equation}
    In particular, this means that if $2L-r\to\infty$, then the volume of the spherical segment dominates.

    To summarise the above, we can use ($2$ times) the volume of the spherical segment as an upper bound for $\connf_L^{\star 2}(r)$ if $r\to\infty$ and $r\leq 2R$, and ($2$ times) the volume of the spherical segment as a lower bound for $\connf_L^{\star 2}(r)$ if $r\to\infty$ and $2L-r\to\infty$. These bounds are more precisely written as the following. Fix $\varepsilon>0$ and let $\overline{\varepsilon}_L$ be any sequence such that $1\ll\overline{\varepsilon}_L\ll \log L$. Then there exists $L_0$ such that for $L\geq L_0$
    \begin{equation}
        \connf_L^{\star 2}(r) \leq \left(1+\varepsilon\right)\frac{4\mathfrak{S}_{d-2}}{\left(d-1\right)^2}\e^{\frac{d-1}{2}\left(2L-r\right)}
    \end{equation}
    for $r\in\left[\overline{\varepsilon}_L,2L\right]$, and
    \begin{equation}
        \connf_L^{\star 2}(r) \geq \left(1-\varepsilon\right)\frac{4\mathfrak{S}_{d-2}}{\left(d-1\right)^2}\e^{\frac{d-1}{2}\left(2L-r\right)}
    \end{equation}
    for $r\in\left[\overline{\varepsilon}_L,2L-\overline{\varepsilon}_L\right]$. Also note that by using that $\connf_L(r)\in\left\{0,1\right\}$,
    \begin{equation}
        \sup_{r\geq 0}\connf_L^{\star 2}(r) =\connf_L^{\star 2}(0) = \norm*{\Opconnf_L}_{1\to 1} \sim \frac{\mathfrak{S}_{d-1}}{2^{d-1}\left(d-1\right)}\e^{\left(d-1\right)L}.
    \end{equation}

    Therefore
    \begin{align}
        \connf^{\star 3}_L\left(\orig,\orig\right) &= \mathfrak{S}_{d-1}\int^{L}_0\connf_L^{\star 2}(r)\left(\sinh r\right)^{d-1}\dd r \nonumber\\
        & \leq \mathfrak{S}_{d-1}\connf^{\star 2}_L(0)\int^{\overline{\varepsilon}_L}_0\left(\sinh r\right)^{d-1}\dd r \nonumber\\
        & \hspace{4cm} + \mathfrak{S}_{d-1}\left(1+\varepsilon\right)\frac{4\mathfrak{S}_{d-2}}{\left(d-1\right)^2}\e^{\left(d-1\right)L}\int^L_0 \e^{-\frac{d-1}{2}r}\frac{1}{2^{d-1}}\e^{\left(d-1\right)r}\dd r \nonumber\\
        &\leq \left(1+\varepsilon\right)^2\mathfrak{S}_{d-1}\frac{\mathfrak{S}_{d-1}}{2^{d-1}}\e^{\left(d-1\right)L}\frac{\mathfrak{S}_{d-1}}{2^{d-1}}\e^{\left(d-1\right)\overline{\varepsilon}_L} \nonumber\\
        & \hspace{4cm} + \mathfrak{S}_{d-1}\left(1+\varepsilon\right)\frac{2^{3-d}\mathfrak{S}_{d-2}}{\left(d-1\right)^2}\e^{\left(d-1\right)L}\frac{2}{d-1}\e^{\frac{1}{2}\left(d-1\right)L}\nonumber\\
        &= \left(1+\varepsilon\right)^2 2^{2-2d}\left(\mathfrak{S}_{d-1}\right)^3\e^{\left(d-1\right)\left(L+\overline{\varepsilon}_L\right)} \nonumber\\
        & \hspace{4cm} + \left(1+\varepsilon\right)\frac{2^{4-d}\mathfrak{S}_{d-1}\mathfrak{S}_{d-2}}{\left(d-1\right)^3}\e^{\frac{3}{2}\left(d-1\right)L}.
    \end{align}
    Since $\overline{\varepsilon}_L\ll \log L \ll L$, the second term dominates. On the over hand,
    \begin{align}
        \connf^{\star 3}_L\left(\orig,\orig\right) &= \mathfrak{S}_{d-1}\int^{L}_0\connf_L^{\star 2}(r)\left(\sinh r\right)^{d-1}\dd r \nonumber\\
        & \geq \mathfrak{S}_{d-1}\left(1-\varepsilon\right)\frac{4\mathfrak{S}_{d-2}}{\left(d-1\right)^2}\e^{\left(d-1\right)L}\int^L_{\overline{\varepsilon}_L} \e^{-\frac{d-1}{2}r}\frac{1-\varepsilon}{2^{d-1}}\e^{\left(d-1\right)r}\dd r \nonumber\\
        & = \mathfrak{S}_{d-1}\left(1-\varepsilon\right)^2\frac{2^{4-d}\mathfrak{S}_{d-2}}{\left(d-1\right)^3}\e^{\left(d-1\right)L}\left(\e^{\frac{1}{2}\left(d-1\right)L} - \e^{\frac{1}{2}\left(d-1\right)\overline{\varepsilon}_L}\right) \nonumber\\
        &= \left(1-\varepsilon\right)^2\frac{2^{4-d}\mathfrak{S}_{d-1}\mathfrak{S}_{d-2}}{\left(d-1\right)^3}\e^{\frac{3}{2}\left(d-1\right)L} - \left(1-\varepsilon\right)^2\frac{2^{4-d}\mathfrak{S}_{d-1}\mathfrak{S}_{d-2}}{\left(d-1\right)^3}\e^{\left(d-1\right)\left(L+\frac{1}{2}\overline{\varepsilon}_L\right)}.
    \end{align}
    Again, since $\overline{\varepsilon}_L\ll L$, the first term dominates. These two bounds mean that
    \begin{equation}
        \connf^{\star 3}_L\left(\orig,\orig\right) \sim \frac{2^{4-d}\mathfrak{S}_{d-1}\mathfrak{S}_{d-2}}{\left(d-1\right)^3}\e^{\frac{3}{2}\left(d-1\right)L},
    \end{equation}
    and combining this with the asymptotic behaviour of $\norm*{\Opconnf_L}_{1\to 1}$ gives the result for this diagram.

    Let us now repeat this for $\connf^{\star 4}(0)$. For the upper bound we find
    \begin{align}
        \connf^{\star 4}_L\left(\orig,\orig\right) &= \mathfrak{S}_{d-1}\int^{2L}_0\connf^{\star 2}(r)^2\left(\sinh r\right)^{d-1}\dd r \nonumber\\
        & \leq \mathfrak{S}_{d-1}\connf^{\star 2}_L(0)^2\int^{\overline{\varepsilon}_L}_0\left(\sinh r\right)^{d-1}\dd r \nonumber\\
        & \hspace{2cm} + \mathfrak{S}_{d-1}\left(1+\varepsilon\right)^2\frac{4^2\mathfrak{S}_{d-2}^2}{\left(d-1\right)^4}\e^{2\left(d-1\right)L}\int^{2L}_0 \e^{-\left(d-1\right)r}\frac{1}{2^{d-1}}\e^{\left(d-1\right)r}\dd r \nonumber\\
        & \leq \left(1+\varepsilon\right)^3 2^{3-3d}\left(\mathfrak{S}_{d-1}\right)^4\e^{\left(d-1\right)\left(2L+\overline{\varepsilon}_L\right)} \nonumber\\
        & \hspace{2cm} + \left(1+\varepsilon\right)^2\frac{2^{6-d}\mathfrak{S}_{d-1}\mathfrak{S}_{d-2}^2}{\left(d-1\right)^4}L\e^{2\left(d-1\right)L}.
    \end{align}
    Since $\overline{\varepsilon}_L\ll \log L$, the second term dominates. For the lower bound,
    \begin{align}
        \connf^{\star 4}_L\left(\orig,\orig\right) &= \mathfrak{S}_{d-1}\int^{2L}_0\connf_L^{\star 2}(r)^2\left(\sinh r\right)^{d-1}\dd r \nonumber\\
        & \geq \mathfrak{S}_{d-1}\left(1-\varepsilon\right)^2\frac{4^2\mathfrak{S}_{d-2}^2}{\left(d-1\right)^4}\e^{2\left(d-1\right)L}\int^{2L-\overline{\varepsilon}_L}_{\overline{\varepsilon}_L} \e^{-\left(d-1\right)r}\frac{1}{2^{d-1}}\e^{\left(d-1\right)r}\dd r \nonumber\\
        & = \left(1-\varepsilon\right)^2\frac{2^{6-d}\mathfrak{S}_{d-1}\mathfrak{S}_{d-2}^2}{\left(d-1\right)^4}L\e^{2\left(d-1\right)L} - \left(1-\varepsilon\right)^2\frac{2^{6-d}\mathfrak{S}_{d-1}\mathfrak{S}_{d-2}^2}{\left(d-1\right)^4}\overline{\varepsilon}_L\e^{2\left(d-1\right)L}.
    \end{align}
    Again $\overline{\varepsilon}_L\ll L$ means the first term dominates, and these bounds mean
    \begin{equation}
        \connf^{\star 4}_L\left(\orig,\orig\right) \sim \frac{2^{6-d}\mathfrak{S}_{d-1}\mathfrak{S}_{d-2}^2}{\left(d-1\right)^4}L\e^{2\left(d-1\right)L}.
    \end{equation}
    Combining this with the asymptotic behaviour of $\norm*{\Opconnf_L}_{1\to 1}$ gives the result for this diagram.

    For the third diagram, first note
    \begin{equation}
        \connf^{\star 1 \star 2 \cdot 2}_L\left(\orig,\orig\right) = \mathfrak{S}_{d-1}\int^{L}_0\connf_L^{\star 2}(r)^2\left(\sinh r\right)^{d-1}\dd r.
    \end{equation}
    The analysis then proceeds identically as for $\connf_L^{\star 4}(\orig,\orig)$, with the only difference coming from the missing factor of $2$ in the upper bound of the integral domain.

    To deal with $\Ecal(L)$, note that the above bounds on $\connf^{\star 2}_L(r)$ show that there exists an $L$-independent constant $C_d<\infty$ such that
    \begin{multline}
        \frac{C}{\norm*{\Opconnf_L}_{1\to 1}}\connf^{\star 2}_L\cdot \connf_L(r) + \frac{C^2}{2\norm*{\Opconnf_L}_{1\to 1}^2}\connf^{\star 2}_L\cdot\connf^{\star 2}_L(r) - \frac{1}{2}\connf_L(r)\\
        \leq \begin{cases}
            C_d &\colon r<\overline{\varepsilon}_L,\\
            C_d \e^{-\frac{1}{2}\left(d-1\right)r} -\frac{1}{2} &\colon \overline{\varepsilon}_L\leq r< L,\\
            C_d \e^{-\left(d-1\right)r}&\colon L\leq r < 2L,\\
            0 &\colon r\geq2L.
        \end{cases}
    \end{multline}
    Since $\overline{\varepsilon}_L\gg 1$, the $\overline{\varepsilon}_L\leq r< L$ part vanishes for sufficiently large $L$ and
    \begin{align}
        \Ecal(L) &\leq \frac{1}{\norm*{\Opconnf_L}_{1\to 1}}\mathfrak{S}_{d-1}\left(\int^{\overline{\varepsilon}_L}_0C_d\left(\sinh r\right)^{d-1}\dd r + \int^{2L}_LC_d\e^{-\left(d-1\right)r}\left(\sinh r\right)^{d-1}\dd r\right)\nonumber\\
        & = \LandauBigO{\e^{-\left(d-1\right)L}\left(\e^{\left(d-1\right)\overline{\varepsilon}_L} + L\right)}\nonumber\\
        &= \LandauBigO{L\e^{-\left(d-1\right)L}},
    \end{align}
    where we have used $\overline{\varepsilon}_L\ll \log L$.
\end{proof}

\begin{proof}[Proof of Theorem~\ref{cor:BooleanResult}]
    The first equality in \eqref{eqn:expectedDegreeBoolean} is a trivial application of Mecke's formula. We can obtain estimates for the size of the terms in Theorem~\ref{thm:CritIntensityExpansion} through Corollary~\ref{cor:diagramBounds} and lemmas~\ref{lem:normsBoolean} and \ref{lem:diagramsBoolean}, which immediately leads to the second equality in \eqref{eqn:expectedDegreeBoolean}. The second equation in Proposition~\ref{cor:BooleanResult} then follows by comparing the error from \eqref{eqn:expectedDegreeBoolean} to the error in $\mu\left(B_L(\orig)\right)$ from Lemma~\ref{lem:normsBoolean}.
\end{proof}

\begin{appendix}

\section{Proofs of the Preliminaries}
\label{app:PrelimProofs}

\ConvolutionOperator*
\begin{proof}
    Let $f\in L^1\left(\HypDim\right)$. Then the inequality $\norm*{Gf}_1 \leq \norm*{f}_1\esssup_{y\in\HypDim}\int_\HypDim\abs*{g(x,y)}\mu\left(\dd x\right)$ follows from Tonelli's theorem and supremum bounds. On the other hand, if we let the test function be a Dirac delta function then
    \begin{equation}
        \norm*{G\delta_u}_1 = \int_\HypDim\abs*{g(x,u)}\mu\left(\dd x\right).
    \end{equation}
    This proves the first equality.

    The second part follows from the Riesz--Thorin theorem as explained in \cite{hutchcroft2019percolation}.
\end{proof}

\spectralradiusofPositive*
\begin{proof}
    First, the isometry invariance and non-negativity implies
    \begin{equation}
        g\left(y,x\right) = g\left(x,y\right) = g\left(\iota_y\left(x\right),\orig\right) \geq 0
    \end{equation}
    for some isometry $\iota_y$ and all $x,y\in\HypDim$.
    
    Then from the sub-multiplicativity of the operator norm, the equality \eqref{eqn:OneOneNorm}, and $g\geq 0$ we have
    \begin{equation}
        \rho_{1\to 1}\left(G\right) \leq \norm*{G}_{1\to 1} = \esssup_{y\in\HypDim}\int_\HypDim g\left(x,y\right)\dd y = \int_\HypDim g\left(x,\orig\right) \dd x.
    \end{equation}

    On the other hand, using \eqref{eqn:Gelfand} and \eqref{eqn:OneOneNorm},
    \begin{align}
        \rho_{1\to 1}\left(G\right) &= \limsup_{n\to\infty}\left(\esssup_{u_0\in\HypDim}\int_\HypDim\ldots\int_\HypDim\left(\prod^{n}_{k=1} g\left(u_{k},u_{k-1}\right)\right)\dd u_{1}\ldots \dd u_{n}\right)^\frac{1}{n},\nonumber\\
        &\geq \limsup_{n\to\infty}\left(\left(\esssup_{u_0\in\HypDim}\int_\HypDim g\left(u_1,u_0\right)\dd u_1\right)\prod^{n}_{k=2}\left(\essinf_{u_{k-1}'\in\HypDim}\int_\HypDim g\left(u_k,u_{k-1}'\right)\dd u_k\right)\right)^\frac{1}{n }\nonumber\\
        & = \limsup_{n\to\infty}\left(\int_\HypDim g\left(x,\orig\right)\dd x\right)^\frac{n}{n}\nonumber\\
        & = \int_\HypDim g\left(x,\orig\right)\dd x,
    \end{align}
    as required.
\end{proof}

\IsoInvtoCommute*

\begin{proof}
    Since $F$ and $G$ are bounded operators, $ L^p\left(\HypDim\right)\subset\mathcal{D}(F)\cap\mathcal{D}(G)$, and $FG$ and $GF$ are both defined and are bounded operators on $L^p\left(\HypDim\right)$.

    Let $\left\{t_{x,y}\colon\HypDim\to\HypDim\right\}_{x,y\in\HypDim}$ be a family of $\HypDim$-isometries such that $t_{x,y}(y)=x$ and $t_{x,y}(x)=y$. Then given $h\in L^p\left(\HypDim\right)$ and $x\in\HypDim$, the boundedness of $F$ and $G$ and $ L^p\left(\HypDim\right)\subset\mathcal{D}(F)\cap\mathcal{D}(G)$ means we can use the Fubini-Tonelli theorem to exchange the order of integration
    \begin{multline}
        \left(FGh\right)(x) = \int_\HypDim \left(\int_\HypDim f(x,y)g(y,z)h(z)\mu\left(\dd z\right)\right)\mu\left(\dd y\right)\\
        = \int_\HypDim \left(\int_\HypDim f(x,y)g(y,z)h(z)\mu\left(\dd y\right)\right)\mu\left(\dd z\right).
    \end{multline}
    Then we can apply the isometry $t_{x,z}t_{x,y}$ to $f$ and $t_{x,z}t_{y,z}$ to $g$ to get
    \begin{multline}
        \left(FGh\right)(x) = \int_\HypDim \left(\int_\HypDim f(t_{x,z}(y),z)g(x,t_{x,z}(y))h(z)\mu\left(\dd y\right)\right)\mu\left(\dd z\right)\\
        =\int_\HypDim \left(\int_\HypDim g(x,t_{x,z}(y))f(t_{x,z}(y),z)h(z)\mu\left(\dd y\right)\right)\mu\left(\dd z\right).
    \end{multline}
    Since $t_{x,y}$ is an isometry,
    \begin{equation}
        \int_\HypDim g(x,t_{x,z}(y))f(t_{x,z}(y),z)h(z)\mu\left(\dd y\right) = \int_\HypDim g(x,y')f(y',z)h(z)\mu\left(\dd y'\right)
    \end{equation}
    for all $x,z\in\HypDim$. Therefore
    \begin{multline}
        \left(FGh\right)(x) = \int_\HypDim \left(\int_\HypDim g(x,y')f(y',z)h(z)\mu\left(\dd y'\right)\right)\mu\left(\dd z\right)\\
         \int_\HypDim \left(\int_\HypDim g(x,y')f(y',z)h(z)\mu\left(\dd z\right)\right)\mu\left(\dd y'\right) = \left(GFh\right)(x),
    \end{multline}
    where we have once again used the Fubini-Tonelli theorem.
\end{proof}

\section{Range of the Spherical Transform}
\label{app:rangesphericaltransform}

The following lemma includes Lemma~\ref{lem:RangeSphericalTransform} from Section~\ref{sec:ExpandLaceCoeff}, with an extra result for non-negative definite models.

\begin{lemma}
\label{lem:RangeSphericalTransformExpand}
    If $G\colon L^2\left(\HypDim\right)\to L^2\left(\HypDim\right)$ is a bounded convolution operator associated with an isometry invariant and real $g$, then
    \begin{equation}
        \widetilde{g}(s)\in\R
    \end{equation}
    for Lebesgue-almost all $s\in\R$, and 
    \begin{equation}
        \norm*{G}_{2\to 2} = \esssup_{s\in\R}\abs*{\widetilde{g}(s)}.
    \end{equation}
    Furthermore, for non-negative definite models
    \begin{equation}
        \widetilde{\connf}_L(s)\in\left[0,\infty\right)
    \end{equation}
    for Lebesgue-almost all $s\in\R$.
\end{lemma}

\begin{proof}
    The isometry invariance of $g$ implies that $g(x,y)=g(y,x)$ for all $x,y\in\HypDim$. Therefore for all $f\in L^2\left(\HypDim\right)$ we have
    \begin{multline}
        \left<f,Gf\right> = \int_\HypDim f(x)\left(\int_\HypDim g(x,y)\overline{f(y)}\mu\left(\dd y\right)\right)\mu\left(\dd x\right) \\= \int_\HypDim \left(\int_\HypDim f(x)g(x,y)\mu\left(\dd x\right)\right)\overline{f(y)}\mu\left(\dd y\right) = \left<Gf,f\right>,
    \end{multline}
    where we have used the boundedness of $G$ and Fubini-Tonelli to exchange the order of integration. However, we also have
    \begin{multline}
        \overline{\left<f,Gf\right>} = \int_\HypDim \overline{f(x)}\left(\int_\HypDim g(x,y)f(y)\mu\left(\dd y\right)\right)\mu\left(\dd x\right)\\
        = \int_\HypDim \left(\int_\HypDim f(x')g(x',y')\mu\left(\dd x'\right)\right)\overline{f(y')}\mu\left(\dd y'\right) = \left<Gf,f\right>,
    \end{multline}
    where $x'=y$ and $y'=x$. Therefore $\left<f,Gf\right>=\overline{\left<f,Gf\right>}$, and $\left<f,Gf\right>\in\R$.

    Now by taking the properties of the spherical transform in Lemma~\ref{lem:InvMultPlancherel},
    \begin{align}
        \left<f,G f\right> &= \int_\HypDim f(x)\overline{\left(G f\right)\left(x\right)}\mu\left(\dd x\right)\nonumber\\
        & = \frac{\mathfrak{S}_{d-1}}{w}\int_\R\widetilde{f}(s) \overline{\left(G f\right)^\sim(s)}\abs{\mathbf{c}(s)}^{-2}\dd s\nonumber\\
        &=  \frac{\mathfrak{S}_{d-1}}{w}\int_\R\widetilde{f}(s) \overline{\widetilde{g}(s)}\overline{\widetilde{f}(s)}\abs{\mathbf{c}(s)}^{-2}\dd s\nonumber\\
        &= \frac{\mathfrak{S}_{d-1}}{w}\int_\R\overline{\widetilde{g}(s)}\abs*{\widetilde{f}(s)}^2\abs{\mathbf{c}(s)}^{-2}\dd s.
    \end{align}
    
    Now suppose for contradiction that there exists a Lebesgue-positive subset $E\subset \R$ such that $\widetilde{g}(s)\not\in \R$ for all $s\in E$, then there exists a $z\in\Complex\setminus\R$, $\varepsilon>0$, and Lebesgue-positive and bounded $F\subset E$ such that the complex ball $B_\varepsilon\left(z\right)\subset \Complex\setminus\R$ and $\widetilde{g}(s)\in B_\varepsilon\left(z\right)$ for all $s\in F$. Then let $h(s) = \Id\left\{F\right\}(s)$. Since $F\subset \R$ is bounded,
    \begin{equation}
        \int_\R \abs*{h(s)}^2 \abs{\mathbf{c}(s)}^{-2}\dd s = \int_F \abs{\mathbf{c}(s)}^{-2}\dd s <\infty,
    \end{equation}
    and therefore $h\in L^2\left(\R\times \mathbb{S}^{d-1}\right)$.
    Therefore if we invert the spherical transform on this function, we can let
    \begin{equation}
        f(x) = \frac{\mathfrak{S}_{d-1}}{w} \int_\R\e^{\left(+is + \frac{d-1}{2}\right)A(x,b)}h(s)\abs{\mathbf{c}(s)}^{-2}\dd s = \frac{\mathfrak{S}_{d-1}}{w} \int_F\e^{\left(+is + \frac{d-1}{2}\right)A(x,b)}\abs{\mathbf{c}(s)}^{-2}\dd s,
    \end{equation}
    and know $f\in L^2\left(\HypDim\right)$. Therefore $\widetilde{f}(s) = h(s)$, and we have
    \begin{equation}
        \left<f,G f\right> = \frac{\mathfrak{S}_{d-1}}{w}\int_F\overline{\widetilde{g}(s)}\abs{\mathbf{c}(s)}^{-2}\dd s \not\in \R.
    \end{equation}
    This contradiction proves that $\widetilde{g}(s)\in\R$ for Lebesgue-almost all $s\in\R$.

    Recall that for non-negative definite models, $\left<f,\Opconnf_L f\right>\in\left[0,\infty\right)$ for all $f\in L^2\left(\HypDim\right)$. The above argument can easily be adapted to show that if $\widetilde{\connf}_L(s)\not\in\left[0,\infty\right)$ for a Lebesgue-positive set of $s$, then one can construct some $f\in L^2\left(\HypDim\right)$ such that $\left<f,\Opconnf_L f\right>\not\in \left[0,\infty\right)$. This contradiction proves the result.

    To show $\norm*{G}_{2\to 2} = \esssup_{s\in\R}\abs*{\widetilde{g}(s)}$, we just apply Lemma~\ref{lem:InvMultPlancherel}. Starting with the Plancherel result and applying the multiplication formula:
    \begin{align}
        \norm*{Gf}^2_{2} &= \frac{1}{w}\int_{\mathbb{S}^{d-1}}\int_\R\abs*{\widetilde{g}(s)}^2 \abs*{\widetilde{f}(s,b)}^2\abs{\mathbf{c}(s)}^{-2}\dd s\dd b\nonumber\\
        &\leq \left(\esssup_{s\in\R}\abs*{\widetilde{g}(s)}\right)^2\frac{1}{w}\int_{\mathbb{S}^{d-1}}\int_\R\abs*{\widetilde{f}(s,b)}^2\abs{\mathbf{c}(s)}^{-2}\dd s\dd b\nonumber\\
        &= \left(\esssup_{s\in\R}\abs*{\widetilde{g}(s)}\right)^2\norm*{f}^2_2.
    \end{align}
    This shows $\norm*{G}_{2\to 2} \leq \esssup_{s\in\R}\abs*{\widetilde{g}(s)}$. To get the equality, we want to find some $f$ such that $\widetilde{f}(s)$ approximates a Dirac delta function at the argument essential supremum. The inverse formula gives exactly this.
\end{proof}

\section{Hyperbolic Triangles}
\label{app:hyperbolictriangles}
    In the following lemmas, the parameters $A,B,C,a,b,c,\alpha,\beta,\gamma$ correspond to the vertices, edge lengths, and angles in Figure~\ref{fig:labellingHyperbolicTriangle}.

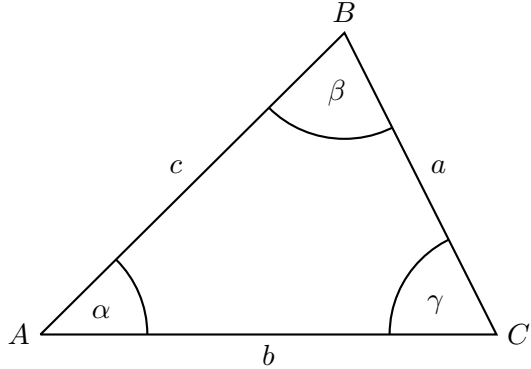
\begin{figure}
    \centering
    \begin{tikzpicture}[scale=2]
    \begin{scope}
        \clip (0,0) -- (2,2) -- (3,0) -- (0,0);
        \draw[thick] (0,0) circle (20pt);
        \draw[thick] (2,2) circle (20pt);
        \draw[thick] (3,0) circle (20pt);
        \draw (0.4,0.15) node{$\alpha$};
        \draw (1.95,1.6) node{$\beta$};
        \draw (2.6,0.2) node{$\gamma$};
    \end{scope}
        \draw[thick] (0,0) node[left]{$A$} -- (2,2) node[above]{$B$} -- (3,0) node[right]{$C$} -- (0,0);
        \draw (1,1) node[above left]{$c$};
        \draw (2.5,1) node[above right]{$a$};
        \draw (1.5,0) node[below]{$b$};
    \end{tikzpicture}
    \caption{Labelling of vertices, angles, and side lengths of the hyperbolic triangle $\Delta ABC$ used in Lemmas \ref{lem:sinerule} and \ref{lem:areaoftriangles}.}
    \label{fig:labellingHyperbolicTriangle}
\end{figure}

\begin{lemma}
    \label{lem:RightHyperbolicTriangles}
    If the hyperbolic triangle $\triangle ABC$ has a right angle at $C$, then
    \begin{equation}
        \cos \alpha = \frac{\tanh b}{\tanh c}.
    \end{equation}
\end{lemma}
\begin{proof}
    See \cite[Corollary~32.13]{martin2012foundations}.
\end{proof}

\begin{lemma}[Sine rule for hyperbolic triangles]\label{lem:sinerule}
    For the hyperbolic triangle $\Delta ABC$,
    \begin{equation}
        \frac{\sin \alpha}{\sinh a} = \frac{\sin \beta}{\sinh b} = \frac{\sin \gamma}{\sinh c}.
    \end{equation}
\end{lemma}
\begin{proof}
    See \cite[Corollary~32.14]{martin2012foundations}.
\end{proof}

\begin{lemma}[Cosine rule for hyperbolic triangles]\label{lem:cosinerule}
    For the hyperbolic triangle $\Delta ABC$,
    \begin{equation}
        \cos \alpha \sinh b \sinh c = \cosh b \cosh c - \cosh a.
    \end{equation}
\end{lemma}
\begin{proof}
    See \cite[Corollary~32.15]{martin2012foundations}.
\end{proof}

\begin{lemma}[Hyperbolic area for hyperbolic triangles]\label{lem:areaoftriangles}
    For the hyperbolic triangle $\Delta ABC\subset\HypTwo$, the hyperbolic area is equal to the angle defect:
    \begin{equation}
        \mu\left(\Delta ABC\right) = \pi - \alpha - \beta - \gamma.
    \end{equation}
\end{lemma}
\begin{proof}
    See \cite[Chapter~33.2]{martin2012foundations}.
\end{proof}
    
\end{appendix}

\small
\bibliography{bibliography}{}
\bibliographystyle{alpha}

\paragraph{Acknowledgements.}
This work was supported by NSERC of Canada. Thanks are due to Markus Heydenreich and Gordon Slade for their advice in presenting the paper.

\end{document}